\documentclass[11pt,reqno]{amsart}

\usepackage{amssymb,amsmath,amsthm,amsfonts}
\usepackage{hyperref, multirow, bm, color, marginnote, graphicx, wrapfig, subcaption, esint, stmaryrd}
\usepackage{diagbox}

\makeatletter
\renewcommand{\tocsection}[3]{%
  \indentlabel{\@ifnotempty{#2}{\bfseries\ignorespaces#1 #2\quad}}\bfseries#3}
\renewcommand{\tocsubsection}[3]{%
  \indentlabel{\@ifnotempty{#2}{\ignorespaces#1 #2\quad}}#3}

\newcommand\@dotsep{4.5}
\def\@tocline#1#2#3#4#5#6#7{\relax
  \ifnum #1>\c@tocdepth 
  \else
    \par \addpenalty\@secpenalty\addvspace{#2}%
    \begingroup \hyphenpenalty\@M
    \@ifempty{#4}{%
      \@tempdima\csname r@tocindent\number#1\endcsname\relax
    }{%
      \@tempdima#4\relax
    }%
    \parindent\z@ \leftskip#3\relax \advance\leftskip\@tempdima\relax
    \rightskip\@pnumwidth plus1em \parfillskip-\@pnumwidth
    #5\leavevmode\hskip-\@tempdima{#6}\nobreak
    \leaders\hbox{$\m@th\mkern \@dotsep mu\hbox{.}\mkern \@dotsep mu$}\hfill
    \nobreak
    \hbox to\@pnumwidth{\@tocpagenum{\ifnum#1=1\bfseries\fi#7}}\par
    \nobreak
    \endgroup
  \fi}
\AtBeginDocument{%
\expandafter\renewcommand\csname r@tocindent0\endcsname{0pt}
}
\def\l@subsection{\@tocline{2}{0pt}{2.5pc}{5pc}{}}
\makeatother

\usepackage[margin=1in]{geometry}

\newcommand{\R}{\mathbb{R}}
\newcommand{\Z}{\mathbb{Z}}

\newcommand{\T}{\mathbb{T}}

\newcommand{\E}{\mathcal{E}}

\newcommand{\bR}{\bm{R}}

\newcommand{\bu}{\bm{u}}

\newcommand{\bx}{\bm{x}}

\newcommand{\X}{\bm{X}}
\newcommand{\Y}{\bm{Y}}
\newcommand{\be}{\bm{e}}

\newcommand{\bv}{\bm{v}}
\newcommand{\p}{\partial}
\renewcommand{\div}{{\rm{div}\,}}

\newcommand{\abs}[1]{\left\lvert #1 \right\rvert}
\newcommand{\norm}[1]{\left\lVert #1 \right\rVert}
\newcommand{\wh}[1]{\widehat{#1}}
\newcommand{\wt}[1]{\widetilde{#1}}
\newcommand{\mc}[1]{\mathcal{#1}}

\newtheorem{theorem}{Theorem}[section]
\newtheorem{lemma}[theorem]{Lemma}
\newtheorem{proposition}[theorem]{Proposition}
\newtheorem{corollary}[theorem]{Corollary}

\theoremstyle{definition}
\newtheorem{remark}{Remark}

\allowdisplaybreaks

\begin{document}
\title{The slender body free boundary problem}

\author{Laurel Ohm}
\address{Department of Mathematics, University of Wisconsin - Madison, Madison, WI 53706}
\email{lohm2@wisc.edu}

\begin{abstract} 
We consider the slender body free boundary problem describing the evolution of an inextensible, closed elastic filament immersed in a Stokes fluid in $\R^3$. The filament elasticity is governed by Euler-Bernoulli beam theory, and the coupling between this 1D elasticity law and the surrounding 3D fluid is governed by the slender body Neumann-to-Dirichlet (NtD) map, which treats the filament as a 3D object with constant cross-sectional radius $0<\epsilon\ll1$. This map serves as a mathematical justification for slender body theories wherein such 3D-1D couplings play a central role. We develop a solution theory for the filament evolution under this coupling. 
Our analysis relies on two main ingredients: (1) an extraction of the principal symbol of the slender body NtD map, from the author's previous work, and (2) a detailed treatment of the tension determination problem for enforcing the inextensibility constraint.
Our work provides a mathematical foundation for various computational models in which a slender filament evolves according to a 1D elasticity law in a 3D fluid. This forms a key development in our broader program to place slender body theories on firm theoretical footing.  
\end{abstract}

\maketitle

\tableofcontents

\setlength{\parskip}{8pt}
\section{Introduction}
We consider the evolution of an inextensible, closed elastic filament immersed in a Stokes fluid in $\R^3$. The filament elasticity is governed by Euler-Bernoulli beam theory, and the coupling between this 1D elasticity law and the surrounding 3D fluid is governed by the \emph{slender body Neumann-to-Dirichlet (NtD) map}, introduced by the author along with Mori and Spirn in \cite{closed_loop,free_ends}, with further developments in \cite{inverse,ohm2025angle,ohm2024free}.
In this coupling, at each time, the filament is considered as a 3D object with constant cross-sectional radius $0<\epsilon\ll1$\,.   

In particular, letting $\T:=\R^3/\Z$, the filament centerline is taken to be a closed, unit-length space curve $\X(s,t):\T\times[0,T]\to \R^3$, parameterized by arclength. Using the subscript $(\cdot)_s$ to denote differentiation in arclength $s$, let $\X_s(s,t)^\perp$ denote the plane perpendicular to $\X$ at cross section $s$ and time $t$.  
For $0<\epsilon\ll1$ sufficiently small (to be made more precise below), we define
\begin{equation}\label{eq:SigmaEps}
  \Sigma_\epsilon(t) = \big\{ \bx\in \X_s(s,t)^\perp\,:\, {\rm dist}(\bx,\X(s,t))<\epsilon\,, \, s\in \T \big\}\,.
\end{equation}
It will often be convenient to think of the surface $\p\Sigma_\epsilon$ as a function of arclength $s$ and angle $\theta$ about the cross section (see figure \ref{fig:filament}). 
To ensure that the geometry of $\Sigma_\epsilon$ makes sense, in particular, that each point $\bx$ in $\Sigma_\epsilon$ belongs to a uniquely determined cross section parameterized by $s$, we will require the following non-self-intersection condition on the curve $\X$. At each time, we define 
\begin{equation}\label{eq:star_norm}
\abs{\X(\cdot,t)}_\star := \inf_{s\neq s'} \frac{\abs{\X(s,t)-\X(s',t)}}{\abs{s-s'}} >0
\end{equation}
and require that this quantity remains positive. When considering the filament shape at a fixed time, we will typically drop the time dependence from our notation and will denote $\abs{\X}_\star=:c_\Gamma>0$. At any fixed time, given a closed $C^2$ curve satisfying \eqref{eq:star_norm}, there exists some maximal radius
\begin{equation}\label{eq:rstar}
  0<r_\star = r_\star(c_\Gamma,\norm{\X_{ss}}_{L^\infty(\T)}) \le \min\bigg\{\frac{c_\Gamma}{2}, \frac{1}{2\norm{\X_{ss}}_{L^\infty(\T)}}\bigg\}
\end{equation}
such that each point $\bx$ satisfying ${\rm dist}(\bx,\X)<r_\star$ belongs to a unique cross section $\X_s(s)^\perp$. We will require $\epsilon\le \frac{r_\star}{4}$ in the definition \eqref{eq:SigmaEps}.

Let $(\bu,p):\R^3\backslash\overline{\Sigma_\epsilon}(t) \to \R^3\times \R$ denote the velocity and pressure fields of a Stokes fluid in $\R^3$, and let $\bm{\sigma}[\bu]=\nabla\bu+\nabla\bu^{\rm T}-p{\bf I}$ denote the fluid stress tensor. 
We define the \emph{slender body free boundary problem} as the evolution problem for the filament centerline $\X(s,t)$ according to 
\begin{align}
-\Delta \bu + \nabla p &=0\,, \quad \div\bu=0  \hspace{2.9cm} \text{in }\R^3\backslash\overline{\Sigma_\epsilon}(t)  \label{eq:SB_FBP1}\\
\int_0^{2\pi}(\bm{\sigma}[\bu]{\bm n})\,\mc{J}_\epsilon(s,\theta,t)\,d\theta &= -\X_{ssss}(s,t)+(\tau(s,t)\X_s)_s \qquad \text{on } \p\Sigma_\epsilon(t)
\label{eq:SB_FBP2}\\
\bu\big|_{\p\Sigma_\epsilon(t)} &= \frac{\p\X}{\p t}(s,t) \label{eq:SB_FBP3}\\
\frac{\p\X_s}{\p t}\cdot\X_s &=0  \label{eq:SB_FBP4}
\end{align}
and $\abs{\bu}\to0$ as $\abs{\bx}\to\infty$.
Here $\mc{J}_\epsilon$ in \eqref{eq:SB_FBP2} is a Jacobian factor corresponding to the surface element on $\p\Sigma_\epsilon$, $\bm{n}$ is the unit normal vector to $\p\Sigma_\epsilon$, and the integral is over the filament surface at each fixed cross section $s$. This angle-averaged condition arises because the no-slip condition \eqref{eq:SB_FBP3} at the filament surface involves a geometric constraint: at each time, we consider the fluid velocity field at the filament surface to be a function of arclength only.\footnote{We understand this as an evolution law for the filament centerline rather than the filament surface. In particular, at each time, the boundary value problem \eqref{eq:SB_PDE} is solved and the centerline $\X$ is updated accordingly. We assume that the 3D body $\Sigma_\epsilon$ follows the centerline shape such that, at each time, the definition \eqref{eq:SigmaEps} is maintained. 
In this way, $\Sigma_\epsilon(t)$ may be considered as a regularization of a literal 1D curve, where the regularization is such that a Stokes boundary value problem is exactly satisfied about $\Sigma_\epsilon$ at each time.} 
These boundary conditions are explained in more detail in section \ref{subsec:SB_NTD}. The driver of the filament evolution is the forcing on the right hand side of \eqref{eq:SB_FBP2}, which comes from Euler-Bernoulli beam theory and models a simple elastic response of the filament to deformation. The scalar function $\tau(s,t)$ is a Lagrange multiplier enforcing the local inextensibility constraint $\abs{\X_s}^2=1$ for all time, which may be rewritten as \eqref{eq:SB_FBP4}. This function, corresponding to the filament tension, is an unknown that must be determined at each time by the instantaneous centerline shape.

\begin{figure}[!ht]
\centering
\includegraphics[scale=0.35]{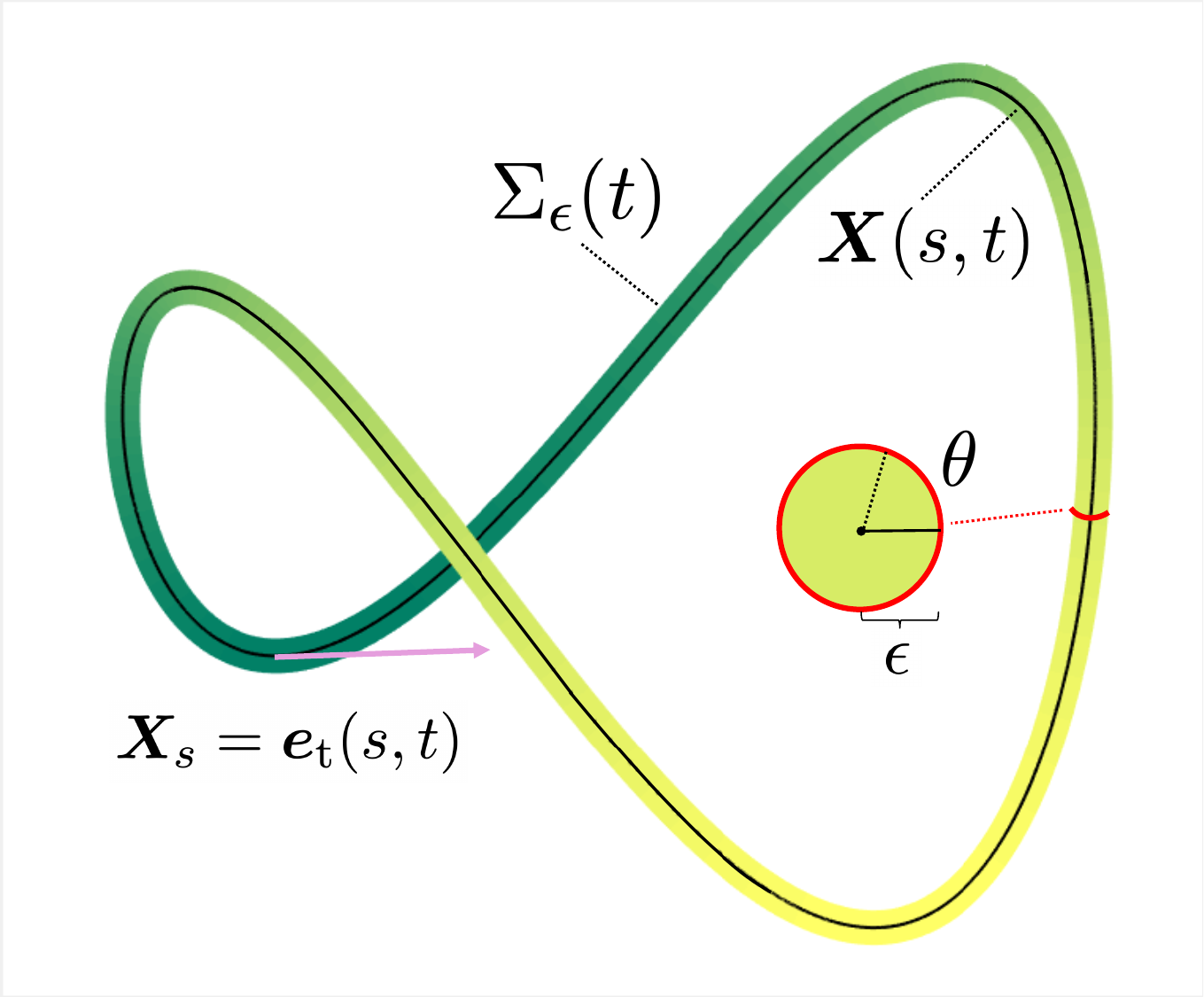}
\caption{Snapshot of a filament $\Sigma_\epsilon(t)$ with centerline $\X(s,t)$. }
\label{fig:filament}
\end{figure}

To analyze \eqref{eq:SB_FBP1}-\eqref{eq:SB_FBP4}, it will be useful for us to reformulate the slender body free boundary problem as the following curve evolution:
\begin{equation}\label{eq:evolution}
\frac{\p\X}{\p t} = -\mc{L}_\epsilon(\X)\big[(\X_{sss}-(\tau\X_s))_s\big] \,, \quad \abs{\X_s}^2=1\,.
\end{equation}
Here the map $\mc{L}_\epsilon(\X)$ is the slender body Neumann-to-Dirichlet (NtD) map, given at each time by the solution to the boundary value problem \eqref{eq:SB_PDE} about the filament $\Sigma_\epsilon$ with centerline $\X$ (see section \ref{subsec:SB_NTD}). This map was proposed in \cite{closed_loop,free_ends} as mathematical justification for (nonlocal) slender body theories, which seek a reduced-dimensional model of the form \eqref{eq:evolution} for the motion of an immersed elastic filament (see section \ref{subsec:lit}). The mapping properties of $\mc{L}_\epsilon(\X)$ for closed, curved $\X$ satisfying \eqref{eq:star_norm} were studied in detail by the author in \cite{ohm2024free}. As the name suggests, the slender body NtD map has the effect of smoothing one derivative, and, via the decompositions shown in this paper and \cite{ohm2024free}, it may be seen that the equation \eqref{eq:evolution} is third-order parabolic.

Here we establish local well-posedness for the evolution \eqref{eq:evolution}. In the following statement, we will use the notation $h^{k,\alpha}(\T)$ to denote the completion of smooth functions in $C^{k,\alpha}(\T)$.
\begin{theorem}[Well-posedness of the slender body free boundary problem]\label{thm:main}
  Given $M,m>0$, there exists $\epsilon_0>0$ depending on $M,m$ such that for all closed initial curves $\X^{\rm in}(s)\in h^{4,\alpha}(\T)$ satisfying 
  \begin{equation}\label{eq:Mm_bounds}
    \norm{\X^{\rm in}}_{C^{4,\alpha}(\T)}< M \quad \text{and} \quad
    \abs{\X^{\rm in}}_\star>m
  \end{equation}
  and for all $0<\epsilon<\epsilon_0$, there exists $T$ depending on $\X^{\rm in}$ and $\epsilon$ such that the slender body free boundary evolution \eqref{eq:evolution} admits a unique solution in $C([0,T],h^{4,\alpha}(\T))$ satisfying \eqref{eq:Mm_bounds}.
\end{theorem}

The proof, which appears in section \ref{sec:evolution}, relies on a detailed decomposition of the slender body NtD operator $\mc{L}_\epsilon(\X)$ established in the author's previous work \cite{ohm2024free}. In particular, at each fixed time (which we will omit from our notation), given a curved, closed filament $\Sigma_\epsilon$ with centerline $\X(s)$, the map $\mc{L}_\epsilon(\X)$ may be decomposed into the slender body NtD map about a straight cylinder (for which we have an explicit symbol) plus remainders which are lower order with respect to regularity or size in $\epsilon$. In \cite{ohm2024free}, this decomposition was then applied to study a toy evolution problem for $\X$ involving the same main part $\mc{L}_\epsilon\p_s^4$. 

The key difficulty addressed in this paper is the incorporation of the local inextensibility constraint $\abs{\X_s}^2=1$. To enforce this constraint, at each time, we will need to solve an auxiliary problem to determine the tension $\tau(s,t)$ based on the instantaneous centerline position $\X(s,t)$. This problem (see equation \eqref{eq:SB_PDE_inex} and its reformulation \eqref{eq:TDP}), is known as the \emph{tension determination problem} and involves a complicated dependence on the curve which must be unpacked.  
Analogous tension determination problems for simpler force-to-velocity maps have been studied previously in the context of the inextensible version of the 2D Peskin problem \cite{kuo2023tension,garcia2025immersed} and resistive force theory \cite{mori2023well,albritton2025rods}.
Much of the present work is devoted to deriving detailed bounds for the solution to the tension determination problem. This is especially delicate because the well-posedness of \eqref{eq:evolution} will rely on the fact that the main behavior of the tension acts only in the tangential direction along the filament centerline, which is more regular for an inextensible curve. 

The inextensibility constraint makes the evolution \eqref{eq:evolution} more physically relevant (compared to the toy evolution studied in \cite{ohm2024free}). In particular, it is physically reasonable to ask about the long time behavior of the filament. This is complicated by the filament's finite radius; in particular, the value of $\epsilon_0$ is set by the initial curve shape but must remain small enough relative to the fiber curvature throughout the entire evolution. Here we choose to prove a general theorem for filaments not necessarily close to equilibrium; however, using this setup, one can ask about the nonlinear stability of various steady states, including the circle and any admissible 3D Euler elasticae. This will be pursued in future work.

The remainder of the paper is structured as follows. Following a review of related curve and interface evolution models and results, we define the slender body NtD map $\mc{L}_\epsilon(\X)$ in section \ref{subsec:SB_NTD} and introduce the tension determination problem for $\tau(s,t)$ in section \ref{subsec:inext}. This section contains the statement of the second main result of this paper, Theorem \ref{thm:TDP}, which gives a decomposition of $\tau$ in terms of the curve. This decomposition is essential for obtaining the well-posedness result of Theorem \ref{thm:main}, and the proof forms the bulk of the analysis in this paper.
In section \ref{sec:straight}, we recall the crucial decomposition of the slender body NtD map $\mc{L}_\epsilon$ from \cite{ohm2024free}, along with additional old and new results for mapping properties of $\mc{L}_\epsilon$ about a straight cylinder.  
Section \ref{sec:TDP} is then devoted to the proof of Theorem \ref{thm:TDP}, and in section \ref{sec:evolution}, we prove Theorem \ref{thm:main}.

\subsection{Related models and results}\label{subsec:lit}
The formulation of the slender body free boundary problem comes from seeking a mathematically reasonable lower-dimensional representation of the evolution of an immersed filament in 3D, with dynamics driven by 1D beam theory. 
A primary contribution of our program is the novel type of NtD operator $\mc{L}_\epsilon(\X)$ coupling the fluid and the filament.

Computationally, evolutions of the form \eqref{eq:evolution} with simpler choices of fluid--filament coupling have enjoyed great success as models for undulatory swimming at low Reynolds number
\cite{cox1970motion, batchelor1970slender, hines1978bend, wiggins1998flexive, camalet2000generic, camalet1999self, johnson1979flagellar, keller1976swimming, pironneau1974optimal, lauga2009hydrodynamics, mori2023well}. A classical choice of coupling, dating back to the 1950s \cite{gray1955propulsion}, is \emph{resistive force theory} (or local slender body theory), given by 
\begin{equation}\label{eq:RFT}
\mc{L}_{\epsilon,\rm loc}(\X)= c\abs{\log\epsilon}({\bf I}+\X_s\otimes\X_s)\,.
\end{equation}
The analogue of \eqref{eq:evolution} using $\mc{L}_{\epsilon,\rm loc}(\X)$ then becomes a fourth-order semilinear parabolic equation for $\X(s,t)$, coupled with an explicit elliptic equation for $\tau(s,t)$. 
This evolution, with its greatly simplified analogue of the tension determination problem, has been studied from a PDE perspective by the author and others in \cite{albritton2025rods,mori2023well,moreau2025n}. In particular, in \cite{albritton2025rods}, we consider the resistive force theory analogue of \eqref{eq:evolution} as a gradient flow of the filament bending energy $\frac{1}{2}\int_\T\abs{\X_{ss}}^2\,ds$ and show global well-posedness as well as long-time convergence to the 3D critical points of the bending energy, known as Euler elasticae \cite{EulerOriginal, levien2008elastica, matsutani2010euler}. The resistive force theory evolution is closely related to ``curve-straightening flows'' in geometric analysis, where useful methods for dealing with the local inextensibility constraint $\abs{\X_s}^2=1$ were developed in \cite{koiso1996motion,oelz2011curve}.

Mathematically, $\mc{L}_{\epsilon,\rm loc}(\X)$ is a convenient choice of coupling operator in \eqref{eq:evolution} because it is simple. Physically, however, this coupling is a perhaps overly simplistic description of the viscous fluid's effects on the filament, capturing only the leading $O(\abs{\log\epsilon})$ effects as $\epsilon\to 0$. A natural candidate for incorporating more fluid effects into the coupling is nonlocal slender body theory, $\mc{L}_{\epsilon,\rm nloc}(\X)$ \cite{keller1976slender,johnson1980improved, tornberg2004simulating}. The operator $\mc{L}_{\epsilon,\rm nloc}(\X)$, which we will not write in full here, involves local slender body theory and a 1D singular integral operator describing nonlocal effects of the filament on itself due to the surrounding fluid. Nonlocal slender body theory is a popular choice of coupling in computational models of filament motion \cite{li2013sedimentation,spagnolie2011comparative,tornberg2006numerical,lauga2009hydrodynamics,cortez2012slender,NorwayPoF,maxian2021integral,maxian2022hydrodynamics}. However, due to well-known issues at high wavenumbers $\sim1/\epsilon$ \cite{gotz2000interactions, inverse, tornberg2004simulating, shelley2000stokesian}, the classical $\mc{L}_{\epsilon,\rm nloc}(\X)$ operator is not viable for making sense of \eqref{eq:evolution} from a PDE perspective. These high frequencies are typically truncated by numerical discretization, which skirts this issue, but the question remains: \emph{Is there a choice of more physically realistic coupling that yields a well-posed PDE of the form \eqref{eq:evolution}?}

To address this, in \cite{closed_loop,free_ends}, the author along with Mori and Spirn proposed the slender body NtD map $\mc{L}_\epsilon(\X)$ as a candidate for coupling the 1D filament dynamics to the 3D fluid. The operator is defined via the solution to a natural type of boundary value problem for the Stokes equations exterior to $\Sigma_\epsilon$ (see equation \eqref{eq:SB_PDE}) and thus incorporates more fluid physics while retaining ``nice'' mapping properties. An in-depth study of this operator's mapping properties was initiated by the author in the simpler Laplace setting \cite{ohm2025angle}, where it was shown that the main behavior about a curved filament is given by the slender body NtD map about a straight cylinder, with remainders that are smoother or small in $\epsilon$. In \cite{ohm2024free}, an analogous decomposition was shown to hold in the Stokes setting and applied to a toy extensible (but not physically realistic) filament evolution equation.

The slender body free boundary problem \eqref{eq:SB_FBP1}-\eqref{eq:SB_FBP4} is closely related to the Peskin problem for the evolution of a 1D closed filament in a 2D Stokes fluid 
\cite{cameron2024critical, chen2023peskin, gancedo2020global, garcia2023critical, garcia2023peskin, lin2019solvability, mori2019well, tong2021regularized, tong2023geometric}. The classical formulation involves an extensible filament enclosing a fluid region whose area must be conserved due to incompressibility. The inextensible version of the Peskin problem has been analyzed in \cite{kuo2023tension,garcia2025immersed} and shares many direct analogues with the slender body problem, including the basic formulation of the tension determination problem and the ultimately third-order parabolic nature of the evolution. The treatment of the tension determination problem in these works serves as the inspiration for the strategy here. 
However, new difficulties arise for the slender body free boundary problem, mostly due to the fact that the slender body NtD map $\mc{L}_\epsilon(\X)$ is much more complex than the analogous coupling operator for the 2D Peskin problem (see Remark \ref{rem:Peskin_TDP}). The main behavior of $\mc{L}_\epsilon(\X)$ is given by an explicit but complicated symbol (see Proposition \ref{prop:Leps_spectrum}), and, in addition to smoother remainder terms, the extraction of this symbol as the main behavior yields remainder terms which are small in $\epsilon$ but not necessarily smoother. This makes the slender body free boundary problem quasilinear.

Finally, we also highlight inspiration from the study of more classical interfacial fluids problems, including the Muskat problem \cite{alazard2020paralinearization, flynn2021vanishing, nguyen2020paradifferential}, Hele-Shaw and water waves problems \cite{alazard2014cauchy, alazard2009paralinearization, beale1993growth,hou1994removing,lannes2005well}. It is possible that the techniques developed to analyze these problems could be used to improve our results.

\subsection{The slender body NtD map}\label{subsec:SB_NTD}
We understand the slender body free boundary problem \eqref{eq:SB_FBP1}-\eqref{eq:SB_FBP4} as an evolution equation \eqref{eq:evolution} for the filament centerline $\X(s,t)$ where, at each time, the slender body NtD map $\mc{L}_\epsilon(\X)[\cdot]$ is defined via the solution to the following boundary value problem, introduced in \cite{closed_loop,free_ends}. 

For any fixed time, we let $\Sigma_\epsilon$ be as in \eqref{eq:SigmaEps}-\eqref{eq:rstar} and omit the $t$-dependence in our notation. We may think of the surface $\p\Sigma_\epsilon$ as a function of the arclength parameter $s$ and the cross sectional angle $\theta$ (see figure \ref{fig:filament}). We may likewise think of the surface Jacobian $\mc{J}_\epsilon=\mc{J}_\epsilon(s,\theta)$ as a function of the cross section $s$ and angle $\theta$.
Given a line force density $\bm{f}(s):\T\to\R^3$, we define the \emph{slender body boundary value problem} as $(\bu,p): \R^3\backslash \overline{\Sigma_\epsilon} \to \R^3\times\R$ satisfying 
\begin{equation}\label{eq:SB_PDE}
\begin{aligned}
-\Delta \bu +\nabla p &= 0\,, \quad \div\bu=0 \qquad \text{in }\R^3\backslash \overline{\Sigma_\epsilon} \\
\int_0^{2\pi}(\bm{\sigma}[\bu]\bm{n}) \, \mc{J}_\epsilon(s,\theta)\,d\theta &= \bm{f}(s) \qquad\qquad\qquad\; \text{on }\p\Sigma_\epsilon \\
\bu\big|_{\p\Sigma_\epsilon} &= \bv(s)\,, \qquad\qquad\quad\;\;\, \text{unknown but independent of }\theta\,,   
\end{aligned}
\end{equation}
along with $\abs{\bu}\to 0$ as $\abs{\bx}\to\infty$. Here the force data $\bm{f}(s)$ is understood as the total surface stress $\bm{\sigma}[\bu]\bm{n}$ per filament cross section, weighted by the fiber surface area through $\mc{J}_\epsilon$ -- an angle-averaged Neumann boundary condition. To uniquely determine a solution, this averaged Neumann condition is supplemented with a geometric constraint on the Dirichlet boundary value of $\bu$. In particular, the value of $\bu$ on the filament surface $\p\Sigma_\epsilon$ is unknown but constrained to be a function of arclength $s$ only.

The \emph{slender body Neumann-to-Dirichlet (NtD) map} $\mc{L}_\epsilon$ is then defined as the operator
\begin{equation}\label{eq:SB_NtD}
\mc{L}_\epsilon(\X) : \bm{f}(s) \mapsto \bm{v}(s)\,,
\end{equation}
i.e., solve the boundary value problem \eqref{eq:SB_PDE} for $\bu$ and evaluate on $\p\Sigma_\epsilon$. The map depends on both the filament radius $\epsilon$ and the centerline geometry $\X(s)$, although to avoid clutter we will typically drop the $\X$ dependence from our notation. Note that $\mc{L}_\epsilon$ is self-adjoint: letting $\bu$, $\bu_h$ be solutions to \eqref{eq:SB_PDE} with force data $\bm{f}(s)$, $\bm{h}(s)$ respectively, then 
\begin{equation}\label{eq:selfadjoint_Leps}
\begin{aligned}
  \int_\T \mc{L}_\epsilon[\bm{f}](s)\cdot\bm{h}(s)\,ds 
  = \int_{\Omega_\epsilon}2\E(\bu):\E(\bu_h)\,d\bx
  = \int_\T \bm{f}(s)\cdot\mc{L}_\epsilon[\bm{h}](s)\,ds\,,
\end{aligned}
\end{equation}
where $\E(\bu)=\frac{1}{2}(\nabla\bu+\nabla\bu^{\rm T})$ is the symmetric gradient.
We additionally recall some important results on the well-posedness of the slender body BVP and the basic mapping properties of $\mc{L}_\epsilon$. By \cite[Theorem 1.2]{closed_loop} and \cite[Lemma 1.14]{ohm2024free}, we have the following. 
\begin{proposition}[Well-posedness of SB BVP \cite{closed_loop,ohm2024free}]\label{prop:SB_BVP}
  Let $\Sigma_\epsilon$ be a filament as in \eqref{eq:SigmaEps}-\eqref{eq:rstar} with centerline $\X(s)\in C^2(\T)$. Given data $\bm{f}(s)\in L^2(\T)$, there exists a unique solution $(\bu,p)\in (\dot H^1\cap L^6)\times L^2(\R^3\backslash \overline{\Sigma_\epsilon})$ to the slender body boundary value problem \eqref{eq:SB_PDE} which satisfies 
  \begin{equation}\label{eq:up_ests}
    \norm{\nabla\bu}_{L^2(\R^3\backslash \overline{\Sigma_\epsilon})} + \norm{p}_{L^2(\R^3\backslash \overline{\Sigma_\epsilon})}\le c(\norm{\X}_{C^2},c_\Gamma)\abs{\log\epsilon}^{1/2}\norm{\bm{f}}_{L^2(\T)}\,.
  \end{equation}
  The slender body Neumann-to-Dirichlet map $\mc{L}_\epsilon$ satisfies the $L^2(\T)\to L^2(\T)$ bound
\begin{equation}
  \norm{\mc{L}_\epsilon[\bm{f}]}_{L^2(\T)} \le c(\norm{\X}_{C^2},c_\Gamma)\abs{\log\epsilon}\norm{\bm{f}}_{L^2(\T)}\,.
\end{equation} 
If $\X(s)\in C^{3,\alpha}(\T)$ and $\bm{f}(s)\in C^{0,\alpha}(\T)$, $0<\alpha<1$, then the slender body NtD map satisfies 
\begin{equation}
  \norm{\mc{L}_\epsilon[\bm{f}]}_{C^{1,\alpha}(\T)} \le c(\epsilon,\norm{\X}_{C^{3,\alpha}},c_\Gamma)\norm{\bm{f}}_{C^{0,\alpha}(\T)}\,.
\end{equation} 
\end{proposition}

The solution $\bu$ to \eqref{eq:SB_PDE} satisfies the energy identity
\begin{equation}\label{eq:energyID}
\int_{\R^3\backslash \overline{\Sigma_\epsilon}} 2\abs{\E(\bu)}^2\,d\bx = \int_\T \bv(s)\cdot\bm{f}(s)\,ds = \int_\T \mc{L}_\epsilon[\bm{f}](s)\cdot\bm{f}(s)\,ds\,,
\end{equation}
from which we may obtain an energy identity for the free boundary problem \eqref{eq:SB_FBP1}-\eqref{eq:SB_FBP4}: 
\begin{equation}\label{eq:dynamic_energy}
\int_{\R^3\backslash \overline{\Sigma_\epsilon}} 2\abs{\E(\bu)}^2\,d\bx = -\int_\T \frac{\p\X}{\p t}\cdot\X_{ssss}\,ds = -\frac{1}{2}\p_t\int_\T \abs{\X_{ss}}^2\,ds\,.
\end{equation}
In particular, we see that the $L^2$ norm of the filament centerline curvature $\abs{\X_{ss}}$ is non-increasing in time, giving \eqref{eq:SB_FBP1}-\eqref{eq:SB_FBP4} a dissipative structure that may be possible to leverage for long-time behavior, although we do not pursue this here.

\subsection{Inextensibility and the tension determination problem}\label{subsec:inext}
The formulation \eqref{eq:SB_FBP1}-\eqref{eq:SB_FBP4} of the slender body free boundary problem involves an additional physical constraint: the filament centerline is locally inextensible and cannot grow or shrink in length over time. In particular, instead of simply solving \eqref{eq:SB_PDE} at each time, we must consider a modified boundary value problem.
More generally, given $\bm{g}(s):\T\to \R^3$, we seek to solve
\begin{equation}\label{eq:SB_PDE_inex}
\begin{aligned}
-\Delta \bu + \nabla p &=0\,, \quad \div\bu=0\qquad \text{in }\R^3\backslash\overline{\Sigma_\epsilon} \\
\int_0^{2\pi}(\bm{\sigma}{\bm n})\,\mc{J}_\epsilon(s,\theta)\,d\theta &= \bm{g}(s)+(\tau\X_s)_s  \qquad \text{on }\p\Sigma_\epsilon \\
\bu\big|_{\p\Sigma_\epsilon} &= \bv(s)\,, \hspace{2.2cm} \text{unknown but independent of }\theta\,, \\
\p_s\bv\cdot\X_s &=0 \hspace{3cm} \text{on }\p\Sigma_\epsilon\,,
\end{aligned}
\end{equation}
along with $\abs{\bu}\to 0$ as $\abs{\bx}\to \infty$. Here the filament tension $\tau(s):\T\to\R$ is a Lagrange multiplier enforcing the local inextensibility constraint $\abs{\X_s}^2=1$.
Note that for $\bu$ in the set 
\begin{equation}
\begin{aligned}
  \bigg\{\bu\in \dot H^1\cap L^6(\R^3\backslash\overline{\Sigma_\epsilon})\,:\, \div\bu=0\,,\; \bu\big|_{\p\Sigma_\epsilon}=\bv(s)\; \text{ independent of }\theta\,, \\
  \int_\T \bv\cdot(\phi\X_s)_s\,ds = 0\;\; \forall\,\phi(s)\in C^\infty(\T)\bigg\}\,,
\end{aligned}
\end{equation}
a unique velocity field satisfying \eqref{eq:SB_PDE_inex} in a weak sense follows by the Lax-Milgram lemma. Existence of a unique pressure field $p\in L^2(\R^3\backslash\overline{\Sigma_\epsilon})$ follows by nearly identical arguments to \cite{closed_loop}, and the same $L^2$-based estimates \eqref{eq:up_ests} hold. However, in order to study the evolution \eqref{eq:evolution} and show Theorem \ref{thm:main}, at each fixed time we will require refined bounds for the tension $\tau$ in terms of the instantaneous curve shape $\X$ and data $\bm{g}$. 

We may obtain an equation for $\tau$ by combining the expression $\bv(s)=\mc{L}_\epsilon[\bm{g}+(\tau\X_s)_s]$ with the constraint $\p_s\bv\cdot\X_s=0$. In particular, the tension must satisfy
\begin{equation}\label{eq:tension_rel}
\p_s\mc{L}_\epsilon[(\tau\X_s)_s]\cdot\X_s = -\p_s\mc{L}_\epsilon[\bm{g}]\cdot\X_s \,.
\end{equation}
Defining the operator $\mc{Q}_\epsilon[\tau]$ as
\begin{equation}\label{eq:Qeps}
\mc{Q}_\epsilon[\tau] :=\big(\mc{L}_\epsilon[(\tau\X_s)_s]\big)_s\cdot\X_s\,,
\end{equation}
we may formulate the \emph{tension determination problem} as follows. Given $\X(s)$ and $\bm{g}(s)$, we seek $\tau$ satisfying 
\begin{equation}\label{eq:TDP}
\mc{Q}_\epsilon[\tau] = -\big(\mc{L}_\epsilon[\bm{g}]\big)_s\cdot\X_s\,.
\end{equation}
Note that the operator $\mc{Q}_\epsilon$ is also self-adjoint. In particular, for any $\phi(s)\in C^{1,\alpha}(\T)$, we have
\begin{equation}
\begin{aligned}
  \int_{\T} \mc{Q}_\epsilon[\tau]\,\phi(s)\,ds &= \int_{\T} \big(\mc{L}_\epsilon[(\tau\X_s)_s]\big)_s\cdot\X_s\,\phi(s)\,ds 
  = -\int_{\T} \mc{L}_\epsilon[(\tau\X_s)_s]\cdot(\phi\X_s)_s\,ds \\
  &= -\int_{\T} (\tau\X_s)_s\cdot\mc{L}_\epsilon[(\phi\X_s)_s]\,ds
  = \int_{\T} \tau\,\X_s\cdot\big(\mc{L}_\epsilon[(\phi\X_s)_s]\big)_s\,ds\\
  &= \int_\T \tau(s)\,\mc{Q}_\epsilon[\phi]\,ds\,,
\end{aligned}
\end{equation}
where we have used that $\mc{L}_\epsilon$ is self-adjoint by \eqref{eq:selfadjoint_Leps}.

Here, using detailed mapping properties of the slender body NtD operator $\mc{L}_\epsilon$ (see section \ref{sec:straight}), we show that \eqref{eq:TDP} admits a unique solution $\tau$ which may be decomposed as follows. 
\begin{theorem}[Tension determination problem]\label{thm:TDP}
Given $\bm{g}(s)\in C^{0,\alpha}(\T)$, $0<\alpha<1$, and a filament $\Sigma_\epsilon$ as in \eqref{eq:SigmaEps}-\eqref{eq:rstar} with centerline $\X(s)\in C^{3,\alpha}(\T)$, there exists a unique solution $\tau\in C^{1,\alpha}(\T)$ to the tension determination problem \eqref{eq:tension_rel}. There exists $0<\epsilon_{\rm t}=\epsilon_{\rm t}(\norm{\X}_{C^{2,\alpha^+}},c_\Gamma)$ such that for any $0<\epsilon<\epsilon_{\rm t}$, this solution satisfies 
\begin{equation}\label{eq:tau_bound}
\begin{aligned}
\tau &= (\mc{G}_0+\mc{G}_\epsilon+\mc{G}_+)[\bm{g}] \,, \\
\mc{G}_0[\bm{g}] &= (I- \overline{\mc{L}}_\epsilon^{\rm tang}\p_{ss})^{-1}\p_s\overline{\mc{L}}_\epsilon^{\rm tang}\bigg[\be_{\rm t}\cdot\bm{g}-\int_{\T}\be_{\rm t}\cdot\bm{g}\,ds \bigg]\\  
\norm{\mc{G}_\epsilon[\bm{g}]}_{C^{1,\alpha}} &\le c(\norm{\X}_{C^{2,\alpha^+}},c_\Gamma)\,\epsilon^{1-\alpha^+}\abs{\log\epsilon}^3\norm{\bm{g}}_{C^{0,\alpha}}\\
\norm{\mc{G}_+[\bm{g}]}_{C^{1,\gamma}} &\le c(\epsilon,\norm{\X}_{C^{3,\alpha}},c_\Gamma)\norm{\bm{g}}_{C^{0,\alpha}}
\end{aligned}
\end{equation}
for any $\alpha^+>\alpha$ and any $\alpha<\gamma<1$. Here $\be_{\rm t}(s)=\X_s(s)$ is the unit tangent vector to the filament centerline and $\overline{\mc{L}}_\epsilon^{\rm tang}$ is given by \eqref{eq:tang_and_nor}.

Moreover, given two nearby\footnote{Here and throughout, we say that two filaments are ``nearby" if $\|\X^{(a)}-\X^{(b)}\|_{C^{2,\alpha}(\T)}=\delta$ for some $\delta$ sufficiently small.} filaments as in \eqref{eq:SigmaEps}-\eqref{eq:rstar} with centerlines $\X^{(a)}$, $\X^{(b)}$ in $C^{3,\alpha}(\T)$, the differences between the remainder terms $\mc{G}_\epsilon^{(a)}-\mc{G}_\epsilon^{(b)}$ and $\mc{G}_+^{(a)}-\mc{G}_+^{(b)}$ in the corresponding tension determination problems satisfy
\begin{equation}\label{eq:tau_bound_lip}
\begin{aligned}
  &\norm{(\mc{G}_\epsilon^{(a)}-\mc{G}_\epsilon^{(b)})[\bm{g}]}_{C^{1,\alpha}}\\
  &\quad\qquad\le c(\|\X^{(a)}\|_{C^{2,\alpha^+}},\|\X^{(b)}\|_{C^{2,\alpha^+}},c_\Gamma)\,\epsilon^{1-\alpha^+}\abs{\log\epsilon}^3\norm{\X^{(a)}-\X^{(b)}}_{C^{2,\alpha^+}}\norm{\bm{g}}_{C^{0,\alpha}}\\
  &\norm{(\mc{G}_+^{(a)}-\mc{G}_+^{(b)})[\bm{g}]}_{C^{1,\gamma}}\le c(\epsilon,\|\X^{(a)}\|_{C^{3,\alpha}},\|\X^{(b)}\|_{C^{3,\alpha}},c_\Gamma)\norm{\X^{(a)}-\X^{(b)}}_{C^{3,\alpha}}\norm{\bm{g}}_{C^{0,\alpha}}\,.
\end{aligned}
\end{equation}

\end{theorem}
The proof of Theorem \ref{thm:TDP} appears in section \ref{sec:TDP}. 
Here $\overline{\mc{L}}_\epsilon^{\rm tang}$ is the tangential part of the slender body NtD operator about a straight filament, which is given by an explicit Fourier multiplier (see Proposition \ref{prop:Leps_spectrum}).
We emphasize that, \emph{a priori}, the term $\mc{G}_\epsilon$ has the same regularity as $\mc{G}_0$, but, crucially, the $\epsilon$-dependence in the bound for $\mc{G}_\epsilon$ is explicit and small. The term $\mc{G}_+$ is \emph{a priori} more regular than $\mc{G}_0$, but the $\epsilon$-dependence in the corresponding bounds is not explicit and may be large.
A difficulty in analyzing the evolution \eqref{eq:evolution} is balancing the effects of both types of remainder terms.

We additionally note that, in contrast to the inextensible version of the 2D Peskin problem \cite{kuo2023tension, garcia2025immersed}, the tension is uniquely determined for any non-intersecting closed centerline shape, including the planar circle (see Remark \ref{rem:Peskin_TDP}). Notably, the slender body free boundary problem does not have an analogue to the area conservation constraint of 2D Peskin, so the planar steady states are different. In particular, the circle is the only non-intersecting planar steady state of \eqref{eq:evolution}, although the family of non-planar non-intersecting steady states is much richer (corresponding to the 3D Euler elasticae \cite{EulerOriginal,levien2008elastica,matsutani2010euler}).

A crucial feature of the tension decomposition \eqref{eq:tau_bound} is the fact that only the tangential part of the data $\bm{g}$ appears in the main term $\mc{G}_0$. We will exploit this in the slender body free boundary problem \eqref{eq:SB_FBP1}-\eqref{eq:SB_FBP4}, where we consider $\bm{g}=\X_{ssss}$ for $\X$ satisfying the inextensibility constraint $\abs{\X_s}^2=1$. By differentiating this constraint in $s$, we have that $\be_{\rm t}\cdot\X_{ssss}=\X_s\cdot\X_{ssss}=-3\X_{sss}\cdot\X_{ss}$. In particular, the tangential direction of $\X_{ssss}$ is more regular than it first appears to be, and $\mc{G}_0[\X_{ssss}]$ can be absorbed into the more regular terms $\mc{G}_+[\X_{ssss}]$. This property is essential for closing our fixed point argument. 
\begin{corollary}[Tension determination with $\X_{ssss}$]\label{cor:TDP_Xssss}
Let $\Sigma_\epsilon$ be as in \eqref{eq:SigmaEps}-\eqref{eq:rstar} with centerline $\X\in C^{4,\alpha}(\T)$, $0<\alpha<1$, satisfying the inextensibility constraint $\abs{\X_s}^2=1$. The solution to the tension determination problem \eqref{eq:TDP} with data $\bm{g}=\X_{ssss}$ then satisfies
\begin{equation}\label{eq:tauXssss}
\begin{aligned}
  \tau[\X_{ssss}] &= \tau_\epsilon[\X_{ssss}] + \tau_+[\X_{ssss}]\\
  \norm{\tau_\epsilon[\X_{ssss}]}_{C^{1,\alpha}(\T)} &\le c(\norm{\X}_{C^{2,\alpha^+}},c_\Gamma)\,\epsilon^{1-\alpha^+}\abs{\log\epsilon}^3\norm{\X}_{C^{4,\alpha}(\T)}\\
  \norm{\tau_+[\X_{ssss}]}_{C^{1,\gamma}(\T)} &\le c(\epsilon,\norm{\X}_{C^{3,\alpha}},c_\Gamma)\norm{\X}_{C^{4,\alpha}(\T)}
\end{aligned}
\end{equation}
for $\alpha^+>\alpha$ and any $\alpha<\gamma<1$. 

Moreover, given two nearby filaments with centerlines $\X^{(a)}$ and $\X^{(b)}$ in $C^{4,\alpha}(\T)$, the difference between the corresponding solutions to the tension determination problem satisfy 
\begin{equation}\label{eq:tauXssss_lip}
\begin{aligned}
  &\norm{\tau_\epsilon^{(a)}[\X_{ssss}^{(a)}]-\tau_\epsilon^{(b)}[\X_{ssss}^{(b)}]}_{C^{1,\alpha}(\T)} \\
  &\quad\le c(\|\X^{(a)}\|_{C^{4,\alpha}},\|\X^{(b)}\|_{C^{2,\alpha^+}},c_\Gamma)\,\epsilon^{1-\alpha^+}\abs{\log\epsilon}^3\norm{\X^{(a)}-\X^{(b)}}_{C^{4,\alpha}(\T)}\\
  &\norm{\tau_+^{(a)}[\X_{ssss}^{(a)}]-\tau_+^{(b)}[\X_{ssss}^{(b)}]}_{C^{1,\gamma}(\T)} \\
  &\quad \le c(\epsilon,\|\X^{(a)}\|_{C^{4,\alpha}},\|\X^{(b)}\|_{C^{2,\alpha^+}},c_\Gamma)\norm{\X^{(a)}-\X^{(b)}}_{C^{4,\alpha}(\T)}\,.
\end{aligned}
\end{equation}
\end{corollary}
A full proof is given in section \ref{subsec:cor_TDP_Xssss}.

To arrive at the tension decomposition of Theorem \ref{thm:TDP} and bounds of Corollary \ref{cor:TDP_Xssss}, we will require more precise information about the mapping properties of the slender body NtD map $\mc{L}_\epsilon$, as this forms the backbone of the tension determination problem \eqref{eq:TDP}. 
For this, we turn to a detailed analysis of the operator $\mc{L}_\epsilon$ from our previous work \cite{ohm2024free,ohm2025angle}. As is typical for similar free boundary problems, the analysis of \eqref{eq:evolution} relies on finding a geometry for which the map $\mc{L}_\epsilon$ is straightforward to understand. In this case, the simplest geometry for which $\mc{L}_\epsilon$ can be defined is a straight cylinder. The proofs of both Theorem \ref{thm:TDP} and Theorem \ref{thm:main} use that we may extract the behavior about a straight cylinder as the main behavior of $\mc{L}_\epsilon$ about a curved filament. We provide the details of this decomposition in the next section.

\section{SB NtD map about the straight cylinder}\label{sec:straight}
Our analysis of the evolution \eqref{eq:evolution} relies on the fact that the slender body NtD map simplifies greatly for a straight cylinder with periodic boundary conditions at the ends. In particular, the behavior of the operator in the tangential and normal directions to the cylinder centerline completely decouples, and moreover is given by explicit Fourier multipliers. 
Throughout, letting 
\begin{equation}\label{eq:Cepsilon} 
\mc{C}_\epsilon :=  \big\{\bx\in \R^2\times\T \; : \; \bx = s\be_z + \epsilon\big(\cos\theta\be_x+\sin\theta\be_y\big)\,,  \; s\in\T\,, \; \theta\in2\pi\T \big\}
\end{equation}
denote the straight cylinder with radius $\epsilon$, we will use $\overline{\mc{L}}_\epsilon$ to denote the slender body NtD map along $\mc{C}_\epsilon$. Here it will also be convenient to separately denote the tangential and normal directions of the map $\overline{\mc{L}}_\epsilon$ along $\mc{C}_\epsilon$ as
\begin{equation}\label{eq:tang_and_nor}
  \overline{\mc{L}}_\epsilon^{\rm tang}[\cdot] = \be_z\cdot\overline{\mc{L}}_\epsilon [\be_z\,\cdot]\,, \quad 
  \overline{\mc{L}}_\epsilon^{\rm nor}[\cdot] = \be_x\cdot\overline{\mc{L}}_\epsilon [\be_x\,\cdot] + \be_y\cdot\overline{\mc{L}}_\epsilon [\be_y\,\cdot]\,.
\end{equation} 
As calculated in \cite{inverse}, we have that 
\begin{equation}\label{eq:eval_prob}
\overline{\mc{L}}_\epsilon^{\rm tang}[e^{2\pi iks}] = m_{\epsilon,{\rm t}}(k)e^{2\pi iks}\,, \qquad 
\overline{\mc{L}}_\epsilon^{\rm nor}[e^{2\pi iks}] = m_{\epsilon,{\rm n}}(k)e^{2\pi iks}\,, 
\end{equation} 
where the multipliers $m_{\epsilon,{\rm t}}(k)$ and $m_{\epsilon,{\rm n}}(k)$ are given by the following explicit expressions.
\begin{proposition}[Slender body NtD spectrum \cite{inverse}]\label{prop:Leps_spectrum}
The eigenvalues $m_{\epsilon,{\rm t}}(k)$, $m_{\epsilon,{\rm n}}(k)$ of the straight slender body NtD map $\overline{\mc{L}}_\epsilon$ satisfying \eqref{eq:eval_prob} are given by
\begin{align}
\label{eq:eigsT}
m_{\epsilon,{\rm t}}(k) &= \frac{2K_0K_1 + 2\pi\epsilon \abs{k} \big( K_0^2 - K_1^2 \big) }{ 8\pi^2\epsilon \abs{k} K_1^2} \\
\label{eq:eigsN}
m_{\epsilon,{\rm n}}(k) &= 
\frac{2K_0K_1K_2 + 2\pi\epsilon \abs{k} \big(K_1^2(K_0+K_2)-2K_0^2K_2 \big)}{4\pi^2\epsilon \abs{k}\big(4K_1^2K_2+2\pi\epsilon \abs{k} K_1(K_1^2-K_0K_2)\big)}
\end{align}
where each $K_j=K_j(2\pi\epsilon \abs{k})$, $j=0,1,2$, is a $j^{\text th}$ order modified Bessel function of the second kind. 
\end{proposition}

The mapping properties of $\overline{\mc{L}}_\epsilon$ may be obtained directly from the expressions in Proposition \ref{prop:Leps_spectrum} and are studied in detail in \cite{ohm2024free}. In particular, we have the following $\epsilon$-dependent result. 
\begin{lemma}[Mapping properties of $\overline{\mc{L}}_\epsilon$ \cite{ohm2024free}]\label{lem:straight_Leps}
Let $\overline{\mc{L}}_\epsilon$ denote the slender body NtD map \eqref{eq:SB_NtD} along the straight filament $\mc{C}_\epsilon$. Given data $\bm{f}(s)\in C^{0,\alpha}(\T)$ along $\mc{C}_\epsilon$ with $\int_\T\bm{f}(s)\,ds=0$, we have 
\begin{equation}\label{eq:holder_NtD}
\norm{\overline{\mc{L}}_\epsilon[\bm{f}]}_{C^{1,\alpha}(\T)} \le c\abs{\log\epsilon}\big(\norm{\bm{f}}_{L^\infty(\T)}+\epsilon^{-1}\abs{\bm{f}}_{\dot C^{0,\alpha}(\T)} \big)\,.
\end{equation}
\end{lemma}

Here, in order to solve the tension determination problem \eqref{eq:TDP}, we will require some additional information about the mapping properties of $\overline{\mc{L}}_\epsilon$. We show the following.
\begin{lemma}[Additional straight filament mapping properties]\label{lem:TDP_auxlem}
Let $\overline{\mc{L}}_\epsilon^{\rm tang}$ denote the tangential direction \eqref{eq:tang_and_nor} of the slender body NtD map along the straight filament $\mc{C}_\epsilon$. Given $h(s)\in C^{\ell,\alpha}(\T)$ for nonnegative integer $\ell$ and $0<\alpha<1$, the operators $(I- \overline{\mc{L}}_\epsilon^{\rm tang}\p_{ss})^{-1}$ and $(I- \overline{\mc{L}}_\epsilon^{\rm tang}\p_{ss})^{-1}\p_s$ satisfy
\begin{equation}\label{eq:M1defs}
\begin{aligned}
  (I- \overline{\mc{L}}_\epsilon^{\rm tang}\p_{ss})^{-1}h &= \mc{M}_{1,\epsilon}[h] + \mc{M}_{1,+}[h]\,,\\
  \norm{\mc{M}_{1,\epsilon}[h]}_{C^{\ell+1,\alpha}(\T)} &\le c\,\epsilon\norm{h}_{C^{\ell,\alpha}(\T)} \\
  \norm{\mc{M}_{1,+}[h]}_{C^{\ell+2,\alpha}(\T)} &\le c(\epsilon)\norm{h}_{C^{\ell,\alpha}(\T)}\,,
\end{aligned}  
\end{equation}
\begin{equation}\label{eq:M2defs}
\begin{aligned}
  (I- \overline{\mc{L}}_\epsilon^{\rm tang}\p_{ss})^{-1}\p_sh &= \mc{M}_{2,\epsilon}[h] + \mc{M}_{2,+}[h]\,,\\
  \norm{\mc{M}_{2,\epsilon}[h]}_{C^{\ell,\alpha}(\T)} &\le c\,\epsilon\norm{h}_{C^{\ell,\alpha}(\T)} \\
  \norm{\mc{M}_{2,+}[h]}_{C^{\ell+1,\alpha}(\T)} &\le c(\epsilon)\norm{h}_{C^{\ell,\alpha}(\T)}\,.
\end{aligned}  
\end{equation}
\end{lemma}
Note that the $c$ appearing in the bounds for $\mc{M}_{j,\epsilon}$ is an absolute constant, and we do not track the explicit $\epsilon$ dependence in the bounds for $\mc{M}_{j,+}$. 
The proof of Lemma \ref{lem:TDP_auxlem} appears in Appendix~\ref{sec:TDP_auxlem_pf}.

A key to obtaining Theorem \ref{thm:main} is the fact that we may extract $\overline{\mc{L}}_\epsilon$ as the main behavior of the operator $\mc{L}_\epsilon$ about an arbitrarily curved filament satisfying \eqref{eq:star_norm}. In particular, in our previous work \cite{ohm2024free}, we show that the behavior of $\mc{L}_\epsilon(\X)$ for curved, closed $\X$ is that of $\overline{\mc{L}}_\epsilon$ plus remainders which are either smoother or small in $\epsilon$.

To actually formulate this decomposition, we will need a map identifying the tangential direction $\be_{\rm t}(s)$ along a curved filament with the tangential direction $\be_z$ along a straight filament, and likewise for the normal directions. 
Given vector-valued $\bm{g}(s)$ defined about the straight filament $\mc{C}_\epsilon$ and vector-valued $\bm{h}(s)$ defined along a curved filament $\Sigma_\epsilon$, we define an identification map $\Phi$ by
\begin{equation}\label{eq:mapPhi_def}
\begin{aligned}
\Phi\bm{g} &= (\bm{g}\cdot\be_z)\be_{\rm t} + (\bm{g}\cdot\be_x)\be_{\rm n_1} + (\bm{g}\cdot\be_y)\be_{\rm n_2} \\
\Phi^{-1}\bm{h} &= (\bm{h}\cdot\be_{\rm t})\be_z + (\bm{h}\cdot\be_{\rm n_1})\be_x + (\bm{h}\cdot\be_{\rm n_2})\be_y\,.
\end{aligned}
\end{equation}
Here $(\be_{\rm t}(s),\be_{\rm n_1}(s),\be_{\rm n_2}(s))$ form an orthonormal frame along $\X(s)$. The map $\Phi$ is explored in more detail in Appendix \ref{sec:Phi_comm}, including some useful estimates from \cite{ohm2024free}. We emphasize that the introduction of a frame is merely a technical tool for stating the most general decomposition theorem below (Theorem \ref{thm:decomp}), and that the necessary results for analyzing the evolution \eqref{eq:evolution}, namely, Corollary \ref{cor:four_derivs} and Corollary \ref{cor:TDP_Xssss}, hold without reference to a frame. The actual proof of local well-posedness for \eqref{eq:evolution} (Theorem \ref{thm:main}) is thus independent of any frame on the filament centerline. 

Now, due to the mean-zero requirement on the data in Lemma \ref{lem:straight_Leps}, given a vector field $\bm{h}(s)$ defined along the centerline of a curved filament, we will also define the notation 
\begin{equation}\label{eq:subtract_mean}
\bm{h}_0^\Phi := \bm{h}-\Phi\int_\T(\Phi^{-1}\bm{h})\,ds\,.
\end{equation}
In particular, note that $\int_\T\Phi^{-1}\bm{h}_0^\Phi\,ds=0$. 
In the following, we take $\abs{\X}_\star=:c_\Gamma>0$, and when comparing nearby filaments with centerline $\X^{(a)}$, $\X^{(b)}$, we take $c_\Gamma=\min\{\abs{\X^{(a)}}_\star,\abs{\X^{(b)}}_\star\}$. 
We may then state the following decomposition.
\begin{theorem}[Slender body NtD decomposition \cite{ohm2024free}]\label{thm:decomp}
Let $0<\alpha<\gamma<1$, $\bm{f}(s)\in C^{0,\alpha}(\T)$, and consider $\Sigma_\epsilon$ as in \eqref{eq:SigmaEps}-\eqref{eq:rstar} with centerline $\X(s)\in C^{3,\alpha}(\T)$. There exists $\epsilon_{\rm n}=\epsilon_{\rm n}(c_\Gamma,\norm{\X}_{C^{2,\gamma}})>0$ such that for $0<\epsilon\le \epsilon_{\rm n}$, the slender body Neumann-to-Dirichlet operator $\mc{L}_\epsilon$ may be decomposed as 
\begin{equation}\label{eq:thm_NtD_decomp}
\mc{L}_\epsilon[\bm{f}](s) = \big({\bf I} + \mc{R}_{\rm n,\epsilon}\big)\big[\Phi\overline{\mc{L}}_\epsilon[\Phi^{-1}\bm{f}_0^\Phi(s)]\big] +\mc{R}_{\rm n,+}[\bm{f}(s)]\,,
\end{equation}
where the remainder terms satisfy
\begin{equation}\label{eq:thm_NtD_ests}
\begin{aligned}
\norm{\mc{R}_{\rm n,\epsilon}[\bm{g}]}_{C^{1,\alpha}(\T)}&\le c(\norm{\X}_{C^{2,\alpha^+}},c_\Gamma)\,\epsilon^{1-\alpha^+}\abs{\log\epsilon}\norm{\bm{g}}_{C^{1,\alpha}(\T)}\\
\norm{\mc{R}_{\rm n,+}[\bm{f}]}_{C^{1,\gamma}(\T)} &\le c(\epsilon,\norm{\X}_{C^{3,\alpha}},c_\Gamma)\norm{\bm{f}}_{C^{0,\alpha}(\T)}
\end{aligned}
\end{equation}
for any $\alpha^+>\alpha$.
Furthermore, given nearby filaments (in the sense of Lemma \ref{lem:XaXb_C2beta}) with centerlines $\X^{(a)}(s)$, $\X^{(b)}(s)$ in $C^{3,\alpha}(\T)$, we have
\begin{equation}\label{eq:thm_NtD_ests_lip}
\begin{aligned}
&\norm{(\mc{R}_{\rm n,\epsilon}^{(a)}-\mc{R}_{\rm n,\epsilon}^{(b)})[\bm{g}]}_{C^{1,\alpha}(\T)}\\
&\quad \le c(\|\X^{(a)}\|_{C^{2,\alpha^+}},\|\X^{(b)}\|_{C^{2,\alpha^+}},c_\Gamma)\, \epsilon^{1-\alpha^+}\abs{\log\epsilon}\norm{\X^{(a)}-\X^{(b)}}_{C^{2,\alpha^+}}\norm{\bm{g}}_{C^{1,\alpha}(\T)} \\
&\norm{(\mc{R}_{\rm n,+}^{(a)}-\mc{R}_{\rm n,+}^{(b)})[\bm{f}]}_{C^{1,\gamma}(\T)}\\
&\quad \le c(\epsilon,\|\X^{(a)}\|_{C^{3,\alpha}},\|\X^{(b)}\|_{C^{3,\alpha}},c_\Gamma)\norm{\X^{(a)}-\X^{(b)}}_{C^{2,\alpha^+}}\norm{\bm{f}}_{C^{0,\alpha}(\T)}\,.
\end{aligned}
\end{equation}
\end{theorem}
Note that the $\epsilon$-dependence is explicit in the bound for the remainder terms $\mc{R}_{\rm n,\epsilon}$, but not for the smoother remainder $\mc{R}_{\rm n,+}$.

Theorem \ref{thm:decomp} will play a major role in obtaining the decomposition of Theorem \ref{thm:TDP} for the solution of the tension determination problem. In addition, we make note of a useful corollary from \cite[Corollary 1.6]{ohm2024free} that will allow us to extract a more convenient expression for the main behavior of $\mc{L}_\epsilon$ when the force data involves many derivatives. 
\begin{corollary}[SB NtD decomposition involving derivatives \cite{ohm2024free}]\label{cor:four_derivs}
For $0<\alpha<\gamma<1$, given $\X(s)\in C^{3,\gamma}(\T)$ and data of the form $\bm{f}(s)=\p_s^4\bm{F}$ for some $\bm{F}\in C^{4,\alpha}(\T)$, the decomposition \eqref{eq:thm_NtD_decomp} of $\mc{L}_\epsilon$ simplifies to 
\begin{equation}\label{eq:F_NtD_decomp}
\begin{aligned}
\mc{L}_\epsilon[\p_s^4\bm{F}]&= ({\bf I}+\mc{R}_{\rm n,\epsilon})\big[\overline{\mc{L}}_\epsilon[\p_s^4\bm{F}]\big] +\wt{\mc{R}}_{\rm n,+}[\p_s^4\bm{F}]\,,\\
\norm{\mc{R}_{\rm n,\epsilon}\big[\overline{\mc{L}}_\epsilon[\p_s^4\bm{F}]\big]}_{C^{1,\alpha}(\T)} &\le c(\|\X\|_{C^{2,\alpha^+}},c_\Gamma)\, \epsilon^{1-\alpha^+}\abs{\log\epsilon}\norm{\overline{\mc{L}}_\epsilon[\p_s^4\bm{F}]}_{C^{1,\alpha}(\T)}\\
\norm{\wt{\mc{R}}_{\rm n,+}[\p_s^4\bm{F}]}_{C^{1,\gamma}(\T)} &\le c(\epsilon,\|\X\|_{C^{3,\gamma}},c_\Gamma)\norm{\bm{F}}_{C^{4,\alpha}(\T)}\,,
\end{aligned}
\end{equation}
Moreover, given two nearby filaments with centerlines $\X^{(a)}(s)$, $\X^{(b)}(s)$ in $C^{3,\gamma}(\T)$, we have
\begin{equation}\label{eq:F_NtD_ests_lip}
\begin{aligned}
&\norm{(\mc{R}_{\rm n,\epsilon}^{(a)}-\mc{R}_{\rm n,\epsilon}^{(b)})[\overline{\mc{L}}_\epsilon[\p_s^4\bm{F}]]}_{C^{1,\alpha}(\T)} \\
&\quad\le c(\|\X^{(a)}\|_{C^{2,\alpha^+}},\|\X^{(b)}\|_{C^{2,\alpha^+}},c_\Gamma)\, \epsilon^{1-\alpha^+}\abs{\log\epsilon}\norm{\X^{(a)}-\X^{(b)}}_{C^{2,\alpha^+}(\T)}\norm{\overline{\mc{L}}_\epsilon[\p_s^4\bm{F}]}_{C^{1,\alpha}(\T)} \\
&\norm{(\wt{\mc{R}}_{\rm n,+}^{(a)}-\wt{\mc{R}}_{\rm n,+}^{(b)})[\p_s^4\bm{F}]}_{C^{1,\gamma}(\T)}\\
&\quad\le c(\epsilon,\|\X^{(a)}\|_{C^{3,\gamma}},\|\X^{(b)}\|_{C^{3,\gamma}},c_\Gamma)\norm{\X^{(a)}-\X^{(b)}}_{C^{3,\gamma}(\T)}\norm{\bm{F}}_{C^{4,\alpha}(\T)}\,.
\end{aligned}
\end{equation}
\end{corollary}

Given the decompositions of Theorem \ref{thm:decomp} and Corollary \ref{cor:four_derivs} for the slender body NtD operator $\mc{L}_\epsilon$, we may now turn to the tension determination problem \eqref{eq:TDP} and the proofs of Theorem \ref{thm:TDP} and Corollary \ref{cor:TDP_Xssss}.


\section{The tension determination problem}\label{sec:TDP}

Following the strategy employed in \cite{garcia2025immersed} in the context of the inextensible 2D Peskin problem, we will prove Theorem \ref{thm:TDP} via the following series of propositions. We begin in section \ref{subsec:decomp_Q} by using Theorem \ref{thm:decomp} to extract the main behavior of the operator $\mc{Q}_\epsilon$ given by \eqref{eq:Qeps}. In section \ref{subsec:reform_TDP}, we show that tension determination problem may be reformulated such that the left hand side operator acting on $\tau$ is of the form $(I+\mc{K})[\tau]$, where $\mc{K}$ is compact from $C^{1,\alpha}(\T)$ to $C^{1,\alpha}(\T)$. We then establish solvability of the tension determination problem in section \ref{subsec:qual_TDP} and obtain a coarse bound for the solution depending on the filament geometry. These results combine to yield brief proofs of Theorem \ref{thm:TDP} in section \ref{subsec:pf_TDP} and Corollary \ref{cor:TDP_Xssss} in section \ref{subsec:cor_TDP_Xssss}.

Throughout, we take $\abs{\X}_\star=:c_\Gamma>0$, and whenever we compare two nearby filaments with centerline $\X^{(a)}$, $\X^{(b)}$, we take $c_\Gamma=\min\{|\X^{(a)}|_\star,|\X^{(b)}|_\star\}$. 

\subsection{Decomposition of $\mc{Q}_\epsilon$}\label{subsec:decomp_Q}
Here we use Theorem \ref{thm:decomp} to show that the main behavior of the operator $\mc{Q}_\epsilon$ in \eqref{eq:Qeps} is given by $\overline{\mc{L}}_\epsilon^{\rm tang}\p_{ss}$, where $\overline{\mc{L}}_\epsilon^{\rm tang}$ is as in \eqref{eq:tang_and_nor}.
\begin{proposition}[Decomposition of $\mc{Q}_\epsilon$]\label{prop:Qdecomp}
Given $\Sigma_\epsilon$ as in \eqref{eq:SigmaEps}-\eqref{eq:rstar} with centerline $\X\in C^{3,\alpha}(\T)$, and given $\psi\in C^{1,\alpha}(\T)$, the operator $\mc{Q}_\epsilon$ given by \eqref{eq:Qeps} may be decomposed as 
\begin{equation}\label{eq:Qdecomp}
  \mc{Q}_\epsilon[\psi] = \big(\overline{\mc{L}}_\epsilon^{\rm tang}\p_{ss} +\p_s\mc{R}_{Q1} +\p_s\mc{R}_{Q2}
  +\mc{R}_{Q3}\big)[\psi]
\end{equation}
where the remainder terms $\mc{R}_{Q1}$, $\mc{R}_{Q2}$, and $\mc{R}_{Q3}$ satisfy
\begin{equation}
\begin{aligned}
  \norm{\mc{R}_{Q1}[\psi]}_{C^{1,\alpha}(\T)} &\le 
  c(\norm{\X}_{C^{2,\alpha^+}},c_\Gamma)\,\epsilon^{-\alpha^+}\abs{\log\epsilon}^2\norm{\psi}_{C^{1,\alpha}} \\
  \norm{\mc{R}_{Q2}[\psi]}_{C^{1,\gamma}(\T)} &\le 
  c(\epsilon,\|\X\|_{C^{3,\alpha}},c_\Gamma)\norm{\psi}_{C^{1,\alpha}}\\
\norm{\mc{R}_{Q3}[\psi]}_{C^{1,\alpha}(\T)} &\le 
 c(\epsilon,\norm{\X}_{C^{3,\alpha}},c_\Gamma)\norm{\psi}_{C^{1,\alpha}}\,.
\end{aligned}
\end{equation}
for any $\alpha^+>\alpha$ and $0<\alpha<\gamma<1$.

Furthermore, given two nearby filaments with centerlines $\X^{(a)}(s)$, $\X^{(b)}(s)$ in $C^{3,\alpha}(\T)$, the differences between the corresponding remainder terms satisfy 
\begin{equation}
\begin{aligned}
&\norm{(\mc{R}_{Q1}^{(a)}-\mc{R}_{Q1}^{(b)})[\psi]}_{C^{1,\alpha}} \\
&\quad\le 
  c(\|\X^{(a)}\|_{C^{2,\alpha^+}},\|\X^{(b)}\|_{C^{2,\alpha^+}},c_\Gamma)\, \epsilon^{-\alpha^+}\abs{\log\epsilon}^2\norm{\X^{(a)}-\X^{(b)}}_{C^{2,\alpha^+}}\|\psi\|_{C^{1,\alpha}}\\
&\norm{(\mc{R}_{Q2}^{(a)}-\mc{R}_{Q2}^{(b)})[\psi]}_{C^{1,\gamma}} \\
&\quad \le
  c(\epsilon,\|\X^{(a)}\|_{C^{3,\alpha}},\|\X^{(b)}\|_{C^{3,\alpha}},c_\Gamma)\norm{\X^{(a)}-\X^{(b)}}_{C^{2,\alpha^+}}\norm{\psi}_{C^{1,\alpha}}\\
&\norm{(\mc{R}_{Q3}^{(a)}-\mc{R}_{Q3}^{(b)})[\psi]}_{C^{1,\alpha}} \\
&\quad \le 
  c(\epsilon,\|\X^{(a)}\|_{C^{3,\alpha}},\|\X^{(b)}\|_{C^{3,\alpha}},c_\Gamma)\norm{\X^{(a)}-\X^{(b)}}_{C^{3,\alpha}}\norm{\psi}_{C^{1,\alpha}}\,.
\end{aligned}
\end{equation}
\end{proposition}

\begin{proof}
We being by using Theorem \ref{thm:decomp} to rewrite $\p_s\mc{L}_\epsilon$. Given $\bm{h}(s)\in C^{0,\alpha}(\T)$, we may decompose $\p_s\mc{L}_\epsilon[\bm{h}]$ as 
\begin{equation}
\begin{aligned}
  \p_s\mc{L}_\epsilon[\bm{h}] &= \p_s\big(\Phi\overline{\mc{L}}_\epsilon[\Phi^{-1}\bm{h}_0^\Phi]\big) +\p_s\mc{R}_{\rm n,\epsilon}\big[\Phi\overline{\mc{L}}_\epsilon[\Phi^{-1}\bm{h}_0^\Phi]\big] +\p_s\mc{R}_{\rm n,+}[\bm{h}]\\
  &= \Phi\overline{\mc{L}}_\epsilon[\Phi^{-1}\p_s\bm{h}] +\p_s\mc{R}_{\rm n,\epsilon}\big[\Phi\overline{\mc{L}}_\epsilon[\Phi^{-1}\bm{h}_0^\Phi]\big] +\p_s\mc{R}_{\rm n,+}[\bm{h}]\\
  &\qquad+ [\p_s,\Phi]\overline{\mc{L}}_\epsilon[\Phi^{-1}\bm{h}_0^\Phi] + \Phi\overline{\mc{L}}_\epsilon\big[[\p_s,\Phi^{-1}]\bm{h}\big]\,,
\end{aligned}
\end{equation}
where we recall the definition \eqref{eq:subtract_mean} of the notation $(\cdot)_0^\Phi$. Recalling the definition \eqref{eq:mapPhi_def} of the map $\Phi$, and taking $\bm{h}=(\psi\X_s)_s$, we may then write $\mc{Q}_\epsilon$ as
\begin{equation}
\begin{aligned}
  \mc{Q}_\epsilon[\psi] 
  &= \be_z\cdot\overline{\mc{L}}_\epsilon[\Phi^{-1}(\psi\X_s)_{ss}] 
  + \p_s\big(\be_{\rm t}\cdot\mc{R}_{\rm n,\epsilon}\big[\Phi\overline{\mc{L}}_\epsilon\big[\big((\psi\X_s)_s \big)_0^\Phi\big]\big]\big) 
  +\p_s\big(\be_{\rm t}\cdot\mc{R}_{\rm n,+}[(\psi\X_s)_s]\big)\\
  &\qquad - \X_{ss}\cdot\mc{R}_{\rm n,\epsilon}\big[\Phi\overline{\mc{L}}_\epsilon\big[\big((\psi\X_s)_s \big)_0^\Phi\big]\big] 
  -\X_{ss}\cdot\mc{R}_{\rm n,+}[(\psi\X_s)_s]\\
  &\qquad+ \be_{\rm t}\cdot\left([\p_s,\Phi]\overline{\mc{L}}_\epsilon\big[\big((\psi\X_s)_s \big)_0^\Phi\big] \right)
  + \be_z\cdot\overline{\mc{L}}_\epsilon\big[[\p_s,\Phi^{-1}](\psi\X_s)_s\big]\,.
\end{aligned}
\end{equation}
Using that $\abs{\X_s}^2=1$, we may write $\be_{\rm t}\cdot(\psi\X_s)_{ss}=\psi_{ss}- \abs{\X_{ss}}^2\psi$ to obtain
\begin{equation}
  \be_z\cdot\overline{\mc{L}}_\epsilon[\Phi^{-1}(\psi\X_s)_{ss}] 
  = \overline{\mc{L}}_\epsilon^{\rm tang}[\be_{\rm t}\cdot(\psi\X_s)_{ss}]
  = \overline{\mc{L}}_\epsilon^{\rm tang}[\psi_{ss}- \abs{\X_{ss}}^2\psi]\,,
\end{equation}
where $\overline{\mc{L}}_\epsilon^{\rm tang}$ is as in \eqref{eq:tang_and_nor}. We may thus decompose $\mc{Q}_\epsilon$ as 
\begin{equation}
\begin{aligned}
  \mc{Q}_\epsilon[\psi] 
  &= \overline{\mc{L}}_\epsilon^{\rm tang}\p_{ss}[\psi] +\p_s\mc{R}_{Q1}[\psi]
  +\p_s\mc{R}_{Q2}[\psi] +\mc{R}_{Q3}[\psi]\,,\\
  \mc{R}_{Q1}[\psi]&= \be_{\rm t}\cdot\mc{R}_{\rm n,\epsilon}\big[\Phi\overline{\mc{L}}_\epsilon\big[\big((\psi\X_s)_s \big)_0^\Phi\big]\big] \\
  \mc{R}_{Q2}[\psi] &=\be_{\rm t}\cdot\mc{R}_{\rm n,+}[(\psi\X_s)_s]\\
  \mc{R}_{Q3}[\psi]&= - \X_{ss}\cdot\mc{R}_{\rm n,\epsilon}\big[\Phi\overline{\mc{L}}_\epsilon\big[\big((\psi\X_s)_s \big)_0^\Phi\big]\big] 
  -\X_{ss}\cdot\mc{R}_{\rm n,+}[(\psi\X_s)_s]\\
  &\quad+ \be_{\rm t}\cdot\left([\p_s,\Phi]\overline{\mc{L}}_\epsilon\big[\big((\psi\X_s)_s \big)_0^\Phi\big] \right)
  + \be_z\cdot\overline{\mc{L}}_\epsilon\big[[\p_s,\Phi^{-1}](\psi\X_s)_s\big]
  - \overline{\mc{L}}_\epsilon^{\rm tang}[\abs{\X_{ss}}^2\psi]\,.
\end{aligned}
\end{equation}

Using the mapping properties \eqref{eq:thm_NtD_ests}, \eqref{eq:holder_NtD} of $\mc{R}_{\rm n,\epsilon}$, $\mc{R}_{\rm n,+}$, and $\overline{\mc{L}}_\epsilon$, we may bound
\begin{equation}
\begin{aligned}
  \norm{\mc{R}_{Q1}[\psi]}_{C^{1,\alpha}(\T)} &\le c(\norm{\X}_{C^{2,\alpha^+}},c_\Gamma)\epsilon^{1-\alpha^+}\abs{\log\epsilon}\|\overline{\mc{L}}_\epsilon\big[\big((\psi\X_s)_s \big)_0^\Phi\big]\big]\|_{C^{1,\alpha}} \\
  &\le c(\norm{\X}_{C^{2,\alpha^+}},c_\Gamma)\epsilon^{1-\alpha^+}\abs{\log\epsilon}^2\big( \norm{(\psi\X_s)_s}_{L^\infty} + \epsilon^{-1}\abs{(\psi\X_s)_s}_{\dot C^{0,\alpha}}\big)\\
  &\le c(\norm{\X}_{C^{2,\alpha^+}},c_\Gamma)\epsilon^{-\alpha^+}\abs{\log\epsilon}^2\norm{\psi}_{C^{1,\alpha}} \,,\\
  \norm{\mc{R}_{Q2}[\psi]}_{C^{1,\gamma}(\T)} &\le 
  c(\epsilon,\|\X\|_{C^{3,\alpha}},c_\Gamma)\norm{(\psi\X_s)_s}_{C^{0,\alpha}} 
  \le c(\epsilon,\|\X\|_{C^{3,\alpha}},c_\Gamma)\norm{\psi}_{C^{1,\alpha}}\,.
\end{aligned}
\end{equation}
In addition, using the commutator estimates \eqref{eq:Phi_commutator_est} for the map $\Phi$, we have 
\begin{equation}
\begin{aligned}
  \norm{\mc{R}_{Q3}[\psi]}_{C^{1,\alpha}(\T)} &\le 
  c(\norm{\X}_{C^{3,\alpha}},c_\Gamma) \big(\epsilon^{1-\alpha^+}\abs{\log\epsilon}+1\big)\norm{\overline{\mc{L}}_\epsilon\big[\big((\psi\X_s)_s \big)_0^\Phi\big]}_{C^{1,\alpha}}\\
  &\quad 
  + c(\epsilon,\norm{\X}_{C^{3,\alpha}},c_\Gamma)\norm{(\psi\X_s)_s}_{C^{0,\alpha}} \\
  &\quad 
  + \norm{\overline{\mc{L}}_\epsilon\big[[\p_s,\Phi^{-1}](\psi\X_s)_s\big]}_{C^{1,\alpha}}
  +\norm{\overline{\mc{L}}_\epsilon^{\rm tang}[\abs{\X_{ss}}^2\psi]}_{C^{1,\alpha}}\\
  &\le c(\epsilon,\norm{\X}_{C^{3,\alpha}},c_\Gamma) \norm{(\psi\X_s)_s}_{C^{0,\alpha}} +c(\epsilon)\norm{\abs{\X_{ss}}^2\psi}_{C^{0,\alpha}}\\
  &\le c(\epsilon,\norm{\X}_{C^{3,\alpha}},c_\Gamma)\norm{\psi}_{C^{1,\alpha}}\,.
\end{aligned}
\end{equation}

Furthermore, using the Lipschitz bounds \eqref{eq:thm_NtD_ests_lip} and \eqref{eq:Phi_lip_est}, for any two nearby filaments with centerlines $\X^{(a)}$ and $\X^{(b)}$, we have that the corresponding remainder terms satisfy 
\begin{equation}
\begin{aligned}
  &\norm{(\mc{R}_{Q1}^{(a)}-\mc{R}_{Q1}^{(b)})[\psi]}_{C^{1,\alpha}} \\
  &\quad \le \epsilon^{1-\alpha^+}\abs{\log\epsilon}\bigg(c(\|\X^{(b)}\|_{C^{2,\alpha^+}},c_\Gamma)\norm{\overline{\mc{L}}_\epsilon\big[\big((\psi\X_s^{(a)})_s \big)_0^{\Phi^{(a)}}- \big((\psi\X_s^{(b)})_s \big)_0^{\Phi^{(b)}}\big]}_{C^{1,\alpha}}\\
  &\qquad + c(\|\X^{(a)}\|_{C^{2,\alpha^+}},\|\X^{(b)}\|_{C^{2,\alpha^+}},c_\Gamma)\norm{\X^{(a)}-\X^{(b)}}_{C^{2,\alpha^+}}\norm{\overline{\mc{L}}_\epsilon\big[\big((\psi\X_s^{(a)})_s \big)_0^{\Phi^{(a)}}\big]\big]}_{C^{1,\alpha}}\bigg) \\
  &\quad \le c(\|\X^{(a)}\|_{C^{2,\alpha^+}},\|\X^{(b)}\|_{C^{2,\alpha^+}},c_\Gamma)\, \epsilon^{-\alpha^+}\abs{\log\epsilon}^2\norm{\X^{(a)}-\X^{(b)}}_{C^{2,\alpha^+}}\|\psi\|_{C^{1,\alpha}}\\
  &\norm{(\mc{R}_{Q2}^{(a)}-\mc{R}_{Q2}^{(b)})[\psi]}_{C^{1,\gamma}} \\
  &\quad \le 
  c(\epsilon,\|\X^{(a)}\|_{C^{3,\alpha}},\|\X^{(b)}\|_{C^{3,\alpha}},c_\Gamma)\norm{\X^{(a)}-\X^{(b)}}_{C^{2,\alpha^+}}\norm{(\psi\X_s^{(a)})_s}_{C^{0,\alpha}} \\
  &\qquad + c(\epsilon,\|\X^{(b)}\|_{C^{3,\alpha}},c_\Gamma)\norm{(\psi\X_s^{(a)})_s- (\psi\X_s^{(b)})_s}_{C^{0,\alpha}}\\
  &\quad\le
  c(\epsilon,\|\X^{(a)}\|_{C^{3,\alpha}},\|\X^{(b)}\|_{C^{3,\alpha}},c_\Gamma)\norm{\X^{(a)}-\X^{(b)}}_{C^{2,\alpha^+}}\norm{\psi}_{C^{1,\alpha}}\,.
\end{aligned}
\end{equation}
Finally, using \eqref{eq:Phi_comm_lip_est}, we may also bound $\mc{R}_{Q3}^{(a)}-\mc{R}_{Q3}^{(b)}$ as 
\begin{equation}
\begin{aligned}
&\norm{(\mc{R}_{Q3}^{(a)}-\mc{R}_{Q3}^{(b)})[\psi]}_{C^{1,\alpha}}\\
&\le (\epsilon^{1-\alpha^+}\abs{\log\epsilon}+1)\bigg(
  c(\|\X^{(b)}\|_{C^{2,\alpha^+}},c_\Gamma)\norm{\overline{\mc{L}}_\epsilon\big[\big((\psi\X_s^{(a)})_s \big)_0^{\Phi^{(a)}} - \big((\psi\X_s^{(b)})_s \big)_0^{\Phi^{(b)}}\big]}_{C^{1,\alpha}} \\
&\quad + 
  c(\|\X^{(a)}\|_{C^{2,\alpha^+}},\|\X^{(b)}\|_{C^{2,\alpha^+}},c_\Gamma)\norm{\X^{(a)}-\X^{(b)}}_{C^{2,\alpha^+}}\norm{\overline{\mc{L}}_\epsilon\big[\big((\psi\X_s^{(a)})_s \big)_0^{\Phi^{(a)}}\big]}_{C^{1,\alpha}}\bigg)\\
&\quad  + 
  c(\epsilon,\|\X^{(a)}\|_{C^{3,\alpha}},\|\X^{(b)}\|_{C^{3,\alpha}},c_\Gamma)\norm{\X^{(a)}-\X^{(b)}}_{C^{3,\alpha}}\norm{(\psi\X_s^{(a)})_s}_{C^{0,\alpha}} \\
&\quad +
  c(\epsilon,\|\X^{(b)}\|_{C^{3,\alpha}},c_\Gamma)\norm{(\psi\X_s^{(a)})_s-(\psi\X_s^{(b)})_s}_{C^{0,\alpha}} 
 + c(\epsilon)\norm{(|\X_{ss}^{(a)}|^2-|\X_{ss}^{(b)}|^2)\psi}_{C^{0,\alpha}}\\
&\le 
  c(\epsilon,\|\X^{(a)}\|_{C^{3,\alpha}},\|\X^{(b)}\|_{C^{3,\alpha}},c_\Gamma)\norm{\X^{(a)}-\X^{(b)}}_{C^{3,\alpha}}\norm{\psi}_{C^{1,\alpha}}\,.
\end{aligned}
\end{equation}
In total, we obtain Proposition \ref{prop:Qdecomp}. 
\end{proof}

\subsection{Reformulation of tension determination problem}\label{subsec:reform_TDP}
Using the decomposition of Proposition \ref{prop:Qdecomp}, we show that the tension determination problem may be reformulated as follows.
\begin{proposition}\label{prop:TDP_reform}
The tension determination problem \eqref{eq:TDP} may be reformulated as 
  \begin{equation}\label{eq:TDP_reform}
    (I +\mc{K})[\tau] = (\mc{G}_0+\mc{G}_\epsilon+\mc{G}_+')[\bm{g}]
  \end{equation}
where, for any $0<\alpha<\gamma<1$, the operator $\mc{K}$ satisfies
\begin{equation}
  \norm{\mc{K}[\tau]}_{C^{1,\gamma}(\T)} \le c(\epsilon,\norm{\X}_{C^{3,\alpha}},c_\Gamma)\norm{\tau}_{C^{1,\alpha}(\T)}\,.
\end{equation}
Given two nearby filaments with centerlines $\X^{(a)}$, $\X^{(b)}$ in $C^{3,\alpha}(\T)$, the difference between the corresponding operators $\mc{K}^{(a)}$ and $\mc{K}^{(b)}$ satisfies
\begin{equation}
  \norm{(\mc{K}^{(a)}-\mc{K}^{(b)})[\tau]}_{C^{1,\gamma}(\T)} \le c(\epsilon,\|\X^{(a)}\|_{C^{3,\alpha}},\|\X^{(b)}\|_{C^{3,\alpha}},c_\Gamma)\norm{\X^{(a)}-\X^{(b)}}_{C^{3,\alpha}}\norm{\tau}_{C^{1,\alpha}(\T)}\,.
\end{equation}

Meanwhile, the main term on the right hand side is given explicitly by
\begin{equation}
  \mc{G}_0[\bm{g}] = (I-\overline{\mc{L}}_\epsilon^{\rm tang}\p_{ss})^{-1}\p_s\overline{\mc{L}}_\epsilon^{\rm tang}\bigg[\be_{\rm t}\cdot\bm{g}-\int_\T\be_{\rm t}\cdot\bm{g}\,ds \bigg]\,,
\end{equation}
while the remainder terms satisfy 
\begin{equation}
\begin{aligned}
  \norm{\mc{G}_\epsilon[\bm{g}]}_{C^{1,\alpha}}&\le c(\|\X\|_{C^{2,\alpha^+}},c_\Gamma)\,\epsilon^{1-\alpha^+}\abs{\log\epsilon}^3\norm{\bm{g}}_{C^{0,\alpha}}\\
  \norm{\mc{G}_+'[\bm{g}]}_{C^{1,\gamma}}&\le c(\epsilon,\norm{\X}_{C^{3,\alpha}},c_\Gamma)\norm{\bm{g}}_{C^{0,\alpha}}
\end{aligned}
\end{equation}
along with the Lipschitz bounds
\begin{equation}
\begin{aligned}
  &\norm{(\mc{G}_\epsilon^{(a)}-\mc{G}_\epsilon^{(b)})[\bm{g}]}_{C^{1,\alpha}}\\
  &\quad\qquad\le c(\|\X^{(a)}\|_{C^{2,\alpha^+}},\|\X^{(b)}\|_{C^{2,\alpha^+}},c_\Gamma)\,\epsilon^{1-\alpha^+}\abs{\log\epsilon}^3\norm{\X^{(a)}-\X^{(b)}}_{C^{2,\alpha^+}}\norm{\bm{g}}_{C^{0,\alpha}}\\
  &\norm{(\mc{G}_+^{'(a)}-\mc{G}_+^{'(b)})[\bm{g}]}_{C^{1,\gamma}}\le c(\epsilon,\|\X^{(a)}\|_{C^{3,\alpha}},\|\X^{(b)}\|_{C^{3,\alpha}},c_\Gamma)\norm{\X^{(a)}-\X^{(b)}}_{C^{2,\gamma^+}}\norm{\bm{g}}_{C^{0,\alpha}}
\end{aligned}
\end{equation}
for any $\alpha^+>\alpha$ and $\gamma^+>\gamma$.
\end{proposition}

\begin{proof}
Using the decomposition of Proposition \ref{prop:Qdecomp}, we first rewrite \eqref{eq:TDP} as
\begin{equation}
  \big(-I+ \overline{\mc{L}}_\epsilon^{\rm tang}\p_{ss} +\p_s\mc{R}_{Q1} +\p_s\mc{R}_{Q2} +\mc{R}_{Q3}+I\big)[\tau] 
  = -\big(\mc{L}_\epsilon[\bm{g}]\cdot\X_s\big)_s + \mc{L}_\epsilon[\bm{g}]\cdot\X_{ss}\,,
\end{equation}
where $I$ denotes the identity operator and we have rewritten the right hand side as well. Applying $-(I-\overline{\mc{L}}_\epsilon^{\rm tang}\p_{ss})^{-1}$ to both sides, the solution to the tension determination problem must satisfy
\begin{equation}\label{eq:rewrite1} 
\begin{aligned}
  (I+ \mc{K}_1 + \mc{K}_2)[\tau] &= \mc{G}_1[\bm{g}] + \mc{G}_2[\bm{g}]\,,\\
  \mc{K}_1 &= -(I- \overline{\mc{L}}_\epsilon^{\rm tang}\p_{ss})^{-1}\p_s\mc{R}_{Q1} \\
  \mc{K}_2 &= -(I- \overline{\mc{L}}_\epsilon^{\rm tang}\p_{ss})^{-1}\p_s\mc{R}_{Q2}- (I- \overline{\mc{L}}_\epsilon^{\rm tang}\p_{ss})^{-1}(\mc{R}_{Q3}+I) \\
  \mc{G}_1[\bm{g}] &= (I- \overline{\mc{L}}_\epsilon^{\rm tang}\p_{ss})^{-1}\p_s(\mc{L}_\epsilon[\bm{g}]\cdot\X_s)\\
  \mc{G}_2[\bm{g}] &= -(I- \overline{\mc{L}}_\epsilon^{\rm tang}\p_{ss})^{-1}\mc{L}_\epsilon[\bm{g}]\cdot\X_{ss}\,.
\end{aligned}
\end{equation}
By Lemma \ref{lem:TDP_auxlem}, we may split $\mc{K}_1$ as $\mc{K}_{1,\epsilon}$ and $\mc{K}_{1,+}$ satisfying 
\begin{equation}
\begin{aligned}
  \norm{\mc{K}_{1,\epsilon}[\psi]}_{C^{1,\alpha}} &\le c\,\epsilon\norm{\mc{R}_{Q1}[\psi]}_{C^{1,\alpha}} \le c(\norm{\X}_{C^{2,\alpha^+}},c_\Gamma)\,\epsilon^{1-\alpha^+}\abs{\log\epsilon}^2\norm{\psi}_{C^{1,\alpha}} \\
  \norm{\mc{K}_{1,+}[\psi]}_{C^{2,\alpha}} &\le c(\epsilon)\norm{\mc{R}_{Q1}[\psi]}_{C^{1,\alpha}} \le c(\epsilon,\norm{\X}_{C^{3,\alpha}},c_\Gamma)\norm{\psi}_{C^{1,\alpha}}
\end{aligned}
\end{equation}
for $\alpha^+>\alpha$.
Again using Lemma \ref{lem:TDP_auxlem}, $\mc{K}_2$ may be bounded as 
\begin{equation}
\begin{aligned}
  \norm{\mc{K}_2[\psi]}_{C^{1,\gamma}} \le c(\epsilon)\big(\norm{\mc{R}_{Q2}[\psi]}_{C^{1,\gamma}}+\norm{(\mc{R}_{Q3}+I)[\psi]}_{C^{0,\gamma}}
  \le c(\epsilon,\norm{\X}_{C^{3,\alpha}},c_\Gamma)\norm{\psi}_{C^{1,\alpha}}
\end{aligned}
\end{equation}
for any $\alpha<\gamma<1$. In addition, by Proposition \ref{prop:Qdecomp}, given two nearby filaments with centerlines $\X^{(a)}(s)$, $\X^{(b)}(s)$, the differences between the corresponding operators $\mc{K}_j^{(a)}-\mc{K}_j^{(b)}$ satisfy
\begin{equation}
\begin{aligned}
  &\norm{(\mc{K}_{1,\epsilon}^{(a)}- \mc{K}_{1,\epsilon}^{(b)})[\psi]}_{C^{1,\alpha}}\\
  &\qquad \le c(\|\X^{(a)}\|_{C^{2,\alpha^+}},\|\X^{(b)}\|_{C^{2,\alpha^+}},c_\Gamma)\, \epsilon^{-\alpha^+}\abs{\log\epsilon}^2\norm{\X^{(a)}-\X^{(b)}}_{C^{2,\alpha^+}}\|\psi\|_{C^{1,\alpha}}\\
  &\norm{(\mc{K}_{1,+}^{(a)}-\mc{K}_{1,+}^{(b)})[\psi]}_{C^{2,\alpha}} \le c(\epsilon,\|\X^{(a)}\|_{C^{3,\alpha}},\|\X^{(b)}\|_{C^{3,\alpha}},c_\Gamma)\norm{\X^{(a)}-\X^{(b)}}_{C^{2,\alpha^+}}\norm{\psi}_{C^{1,\alpha}}\\
  &\norm{(\mc{K}_2^{(a)}-\mc{K}_2^{(b)})[\psi]}_{C^{1,\gamma}} \le c(\epsilon,\|\X^{(a)}\|_{C^{3,\alpha}},\|\X^{(b)}\|_{C^{3,\alpha}},c_\Gamma)\norm{\X^{(a)}-\X^{(b)}}_{C^{3,\alpha}}\norm{\psi}_{C^{1,\alpha}}\,.
\end{aligned}
\end{equation}

We next consider the right hand side of \eqref{eq:rewrite1}. First, using the decomposition of Theorem \ref{thm:decomp} for $\mc{L}_\epsilon$, we may write 
\begin{equation}
\begin{aligned}
  \mc{L}_\epsilon[\bm{g}]\cdot\X_s &= \be_z\cdot\overline{\mc{L}}_\epsilon[\Phi^{-1}\bm{g}_0^\Phi(s)]  + \be_{\rm t}\cdot\mc{R}_{\rm n,\epsilon}\big[\Phi\overline{\mc{L}}_\epsilon[\Phi^{-1}\bm{g}_0^\Phi(s)]\big] +\be_{\rm t}\cdot\mc{R}_{\rm n,+}[\bm{g}(s)]\\
  &= \overline{\mc{L}}_\epsilon^{\rm tang}\big[\be_{\rm t}\cdot\bm{g}-\int_{\T}\be_{\rm t}\cdot\bm{g}\,ds \big] + \be_{\rm t}\cdot\mc{R}_{\rm n,\epsilon}\big[\Phi\overline{\mc{L}}_\epsilon[\Phi^{-1}\bm{g}_0^\Phi(s)]\big] +\be_{\rm t}\cdot\mc{R}_{\rm n,+}[\bm{g}(s)]\,.
\end{aligned}
\end{equation}
We may thus decompose
\begin{equation}
\begin{aligned}
  \mc{G}_1[\bm{g}] &= \mc{G}_0[\bm{g}] + \mc{G}_{1,1}[\bm{g}] +\mc{G}_{1,2}[\bm{g}]\,,\\
  \mc{G}_0[\bm{g}] &=(I- \overline{\mc{L}}_\epsilon^{\rm tang}\p_{ss})^{-1}\p_s\overline{\mc{L}}_\epsilon^{\rm tang}\big[\be_{\rm t}\cdot\bm{g}-\int_{\T}\be_{\rm t}\cdot\bm{g}\,ds \big] \\
  \mc{G}_{1,1}[\bm{g}] &= (I- \overline{\mc{L}}_\epsilon^{\rm tang}\p_{ss})^{-1}\p_s\big(\be_{\rm t}\cdot\mc{R}_{\rm n,\epsilon}\big[\Phi\overline{\mc{L}}_\epsilon[\Phi^{-1}\bm{g}_0^\Phi(s)]\big]\big)\\
  \mc{G}_{1,2}[\bm{g}] &= (I- \overline{\mc{L}}_\epsilon^{\rm tang}\p_{ss})^{-1}\p_s\big(\be_{\rm t}\cdot\mc{R}_{\rm n,+}[\bm{g}(s)] \big)\,.
\end{aligned}
\end{equation}
Note that by Lemmas \ref{lem:TDP_auxlem} and \ref{lem:straight_Leps}, we may write $\mc{G}_0$ as $\mc{G}_0=\mc{G}_{0\epsilon}+\mc{G}_{0+}$, where
\begin{equation}\label{eq:G0ests}
\begin{aligned}
  \norm{\mc{G}_{0\epsilon}[\bm{g}]}_{C^{1,\alpha}} &\le c\,\epsilon\norm{\overline{\mc{L}}_\epsilon^{\rm tang}\big[\be_{\rm t}\cdot\bm{g}-\int_{\T}\be_{\rm t}\cdot\bm{g}\,ds\big]}_{C^{1,\alpha}}\\
  &\le c\abs{\log\epsilon}\norm{\be_{\rm t}\cdot\bm{g}-\int_{\T}\be_{\rm t}\cdot\bm{g}\,ds}_{C^{0,\alpha}} 
  \le c\abs{\log\epsilon}\norm{\be_{\rm t}\cdot\bm{g}}_{C^{0,\alpha}} \\
  \norm{\mc{G}_{0+}[\bm{g}]}_{C^{2,\alpha}} &\le c(\epsilon)\norm{\overline{\mc{L}}_\epsilon^{\rm tang}\big[\be_{\rm t}\cdot\bm{g}-\int_{\T}\be_{\rm t}\cdot\bm{g}\,ds\big]}_{C^{1,\alpha}}
  \le c(\epsilon)\norm{\be_{\rm t}\cdot\bm{g}}_{C^{0,\alpha}}\,.
\end{aligned}
\end{equation}
By Lemma \ref{lem:TDP_auxlem}, we may also write $\mc{G}_{1,1}[\bm{g}]=\mc{G}_{1,1\epsilon}[\bm{g}]+\mc{G}_{1,1+}[\bm{g}]$ where, using Theorem \ref{thm:decomp} and Lemma \ref{lem:straight_Leps}, we have
\begin{equation}
\begin{aligned}
  \norm{\mc{G}_{1,1\epsilon}[\bm{g}]}_{C^{1,\alpha}} &\le c\,\epsilon\norm{\be_{\rm t}\cdot\mc{R}_{\rm n,\epsilon}\big[\Phi\overline{\mc{L}}_\epsilon[\Phi^{-1}\bm{g}_0^\Phi(s)]\big]}_{C^{1,\alpha}}\\
  &\le c(\|\X\|_{C^{2,\alpha^+}},c_\Gamma)\,\epsilon^{2-\alpha^+}\abs{\log\epsilon}\norm{\overline{\mc{L}}_\epsilon[\Phi^{-1}\bm{g}_0^\Phi(s)]\big]}_{C^{1,\alpha}}\\
  &\le c(\|\X\|_{C^{2,\alpha^+}},c_\Gamma)\,\epsilon^{1-\alpha^+}\abs{\log\epsilon}^2\norm{\bm{g}}_{C^{0,\alpha}}\\
  \norm{\mc{G}_{1,1+}[\bm{g}]}_{C^{2,\alpha}} &\le 
  c(\epsilon)\norm{\be_{\rm t}\cdot\mc{R}_{\rm n,\epsilon}\big[\Phi\overline{\mc{L}}_\epsilon[\Phi^{-1}\bm{g}_0^\Phi(s)]\big]}_{C^{1,\alpha}} \\
  &\le c(\epsilon,\|\X\|_{C^{2,\alpha^+}},c_\Gamma)\norm{\bm{g}}_{C^{0,\alpha}}\,.
\end{aligned}
\end{equation}
In addition, given two nearby filaments with centerlines $\X^{(a)}(s)$ and $\X^{(b)}(s)$, we may obtain the following Lipschitz bounds: 
\begin{equation}
\begin{aligned}
  &\norm{(\mc{G}_{1,1\epsilon}^{(a)}-\mc{G}_{1,1\epsilon}^{(b)})[\bm{g}]}_{C^{1,\alpha}} \\
  &\qquad \le c(\|\X^{(a)}\|_{C^{2,\alpha^+}},\|\X^{(b)}\|_{C^{2,\alpha^+}},c_\Gamma)\,\epsilon^{1-\alpha^+}\abs{\log\epsilon}^2\norm{\X^{(a)}-\X^{(b)}}_{C^{2,\alpha^+}}\norm{\bm{g}}_{C^{0,\alpha}}\\
  &\norm{(\mc{G}_{1,1+}^{(a)}-\mc{G}_{1,1+}^{(b)})[\bm{g}]}_{C^{2,\alpha}} 
  \le c(\epsilon,\|\X^{(a)}\|_{C^{2,\alpha^+}},\|\X^{(b)}\|_{C^{2,\alpha^+}},c_\Gamma)\norm{\X^{(a)}-\X^{(b)}}_{C^{2,\alpha^+}}\norm{\bm{g}}_{C^{0,\alpha}}\,.
\end{aligned}
\end{equation}
Furthermore, using Lemma \ref{lem:TDP_auxlem} and Theorem \ref{thm:decomp}, for any $\alpha<\gamma<1$, we may estimate $\mc{G}_{1,2}$ as
\begin{equation}
\begin{aligned}
  \norm{\mc{G}_{1,2}[\bm{g}]}_{C^{1,\gamma}} &\le c(\epsilon)\norm{\be_{\rm t}\cdot\mc{R}_{\rm n,+}[\bm{g}]}_{C^{1,\gamma}} \\
  &\le c(\epsilon,\|\X\|_{C^{3,\alpha}},c_\Gamma)\norm{\bm{g}}_{C^{0,\alpha}}\\
  \norm{(\mc{G}_{1,2}^{(a)}-\mc{G}_{1,2}^{(b)})[\bm{g}]}_{C^{1,\gamma}} &\le
  c(\epsilon,\|\X^{(a)}\|_{C^{3,\alpha}},\|\X^{(b)}\|_{C^{3,\alpha}},c_\Gamma)\norm{\X^{(a)}-\X^{(b)}}_{C^{2,\alpha^+}}\norm{\bm{g}}_{C^{0,\alpha}}\,.
\end{aligned}
\end{equation}
Finally, again using Lemma \ref{lem:TDP_auxlem} and Theorem \ref{thm:decomp}, we may estimate $\mc{G}_2$ as 
\begin{equation}
\begin{aligned}
  \norm{\mc{G}_2[\bm{g}]}_{C^{2,\alpha}} &\le c(\epsilon)\norm{\mc{L}_\epsilon[\bm{g}]\cdot\X_{ss}}_{C^{1,\alpha}}
  \le c(\epsilon,\norm{\X}_{C^{3,\alpha}},c_\Gamma)\norm{\bm{g}}_{C^{0,\alpha}} \\
  \norm{(\mc{G}_2^{(a)}-\mc{G}_2^{(b)})[\bm{g}]}_{C^{2,\alpha}} &\le 
  c(\epsilon,\|\X^{(a)}\|_{C^{3,\alpha}},\|\X^{(b)}\|_{C^{3,\alpha}},c_\Gamma)\norm{\X^{(a)}-\X^{(b)}}_{C^{2,\alpha^+}}\norm{\bm{g}}_{C^{0,\alpha}}\,.
\end{aligned}
\end{equation}

In total, we may rewrite \eqref{eq:rewrite1} as
\begin{equation}\label{eq:Keps_def}
  (I+ \mc{K}_\epsilon + \mc{K}_+)[\tau] = (\mc{G}_0+\wt{\mc{G}}_\epsilon+\wt{\mc{G}}_+)[\bm{g}]\,,
\end{equation}
where the left hand side satisfies
\begin{equation}\label{eq:Kbds1}
\begin{aligned}
  \norm{\mc{K}_\epsilon[\psi]}_{C^{1,\alpha}} &\le c(\norm{\X}_{C^{2,\alpha^+}},c_\Gamma)\,\epsilon^{1-\alpha^+}\abs{\log\epsilon}^2\norm{\psi}_{C^{1,\alpha}} \\
  \norm{\mc{K}_+[\psi]}_{C^{1,\gamma}} &\le c(\epsilon,\norm{\X}_{C^{3,\alpha}},c_\Gamma)\norm{\psi}_{C^{1,\alpha}}
\end{aligned}
\end{equation}
as well as 
\begin{equation}\label{eq:Kbds2}
\begin{aligned}
  &\norm{(\mc{K}_\epsilon^{(a)}-\mc{K}_\epsilon^{(b)})[\psi]}_{C^{1,\alpha}}\\ 
  &\qquad \le c(\|\X^{(a)}\|_{C^{2,\alpha^+}},\|\X^{(b)}\|_{C^{2,\alpha^+}},c_\Gamma)\,\epsilon^{1-\alpha^+}\abs{\log\epsilon}^2\norm{\X^{(a)}-\X^{(b)}}_{C^{2,\alpha^+}}\norm{\psi}_{C^{1,\alpha}} \\
  &\norm{(\mc{K}_+^{(a)}[\psi]-\mc{K}_+^{(b)}[\psi])}_{C^{1,\gamma}} \le c(\epsilon,\|\X^{(a)}\|_{C^{3,\alpha}},\|\X^{(b)}\|_{C^{3,\alpha}},c_\Gamma)\norm{\X^{(a)}-\X^{(b)}}_{C^{3,\alpha}}\norm{\psi}_{C^{1,\alpha}}\,.
\end{aligned}
\end{equation}
In addition, recalling \eqref{eq:G0ests}, the right hand side satisfies 
\begin{equation}\label{eq:Gests1}
\begin{aligned}
  \norm{\mc{G}_{0\epsilon}[\bm{g}]}_{C^{1,\alpha}}&\le c\abs{\log\epsilon}\norm{\be_{\rm t}\cdot\bm{g}}_{C^{0,\alpha}}\,, \quad 
  \norm{\mc{G}_{0+}[\bm{g}]}_{C^{2,\alpha}}\le c(\epsilon)\norm{\be_{\rm t}\cdot\bm{g}}_{C^{0,\alpha}}
  \\
  \norm{\wt{\mc{G}}_\epsilon[\bm{g}]}_{C^{1,\alpha}}&\le c(\|\X\|_{C^{2,\alpha^+}},c_\Gamma)\,\epsilon^{1-\alpha^+}\abs{\log\epsilon}^2\norm{\bm{g}}_{C^{0,\alpha}}\\
  \norm{\wt{\mc{G}}_+[\bm{g}]}_{C^{1,\gamma}}&\le c(\epsilon,\norm{\X}_{C^{3,\alpha}},c_\Gamma)\norm{\bm{g}}_{C^{0,\alpha}}
\end{aligned}
\end{equation}
as well as 
\begin{equation}\label{eq:Gests2}
\begin{aligned}
  %
  &\norm{(\wt{\mc{G}}_\epsilon^{(a)}-\wt{\mc{G}}_\epsilon^{(b)})[\bm{g}]}_{C^{1,\alpha}}\\
  &\quad\qquad\le c(\|\X^{(a)}\|_{C^{2,\alpha^+}},\|\X^{(b)}\|_{C^{2,\alpha^+}},c_\Gamma)\,\epsilon^{1-\alpha^+}\abs{\log\epsilon}^2\norm{\X^{(a)}-\X^{(b)}}_{C^{2,\alpha^+}}\norm{\bm{g}}_{C^{0,\alpha}}\\
  &\norm{(\wt{\mc{G}}_+^{(a)}-\wt{\mc{G}}_+^{(b)})[\bm{g}]}_{C^{1,\gamma}}\le c(\epsilon,\|\X^{(a)}\|_{C^{3,\alpha}},\|\X^{(b)}\|_{C^{3,\alpha}},c_\Gamma)\norm{\X^{(a)}-\X^{(b)}}_{C^{2,\alpha^+}}\norm{\bm{g}}_{C^{0,\alpha}}\,.
\end{aligned}
\end{equation}

Now, for $\epsilon$ sufficiently small, we may use a Neumann series expansion to write 
\begin{equation}\label{eq:IplusKinv}
\begin{aligned}
  &(I+\mc{K}_\epsilon)^{-1} = I+\sum_{j=1}^\infty \mc{K}_\epsilon^j=: I+\wt{\mc{K}}_\epsilon\,,\\
  &\norm{\wt{\mc{K}}_\epsilon[\psi]}_{C^{1,\alpha}}\le c(\norm{\X}_{C^{2,\alpha^+}},c_\Gamma)\,\epsilon^{1-\alpha^+}\abs{\log\epsilon}^2\norm{\psi}_{C^{1,\alpha}}\\
  &\norm{(\wt{\mc{K}}_\epsilon^{(a)}-\wt{\mc{K}}_\epsilon^{(b)})[\psi]}_{C^{1,\alpha}}\\
  &\qquad \le c(\|\X^{(a)}\|_{C^{2,\alpha^+}},\|\X^{(b)}\|_{C^{2,\alpha^+}},c_\Gamma)\,\epsilon^{1-\alpha^+}\abs{\log\epsilon}^2\norm{\X^{(a)}-\X^{(b)}}_{C^{2,\alpha^+}}\norm{\psi}_{C^{1,\alpha}}\,.
\end{aligned}
\end{equation}
Then, writing
\begin{equation}
  (I+ (I+\mc{K}_\epsilon)^{-1}\mc{K}_+)[\tau] = (I+\mc{K}_\epsilon)^{-1}(\mc{G}_0+\wt{\mc{G}}_\epsilon+\wt{\mc{G}}_+)[\bm{g}]\,,
\end{equation}
for $\epsilon$ sufficiently small, we may define
\begin{equation}
\begin{aligned}
  \mc{K} &= (I+\wt{\mc{K}}_\epsilon)\mc{K}_+ \\
  \mc{G}_\epsilon &= \wt{\mc{K}}_\epsilon\mc{G}_{0\epsilon} + (I+\wt{\mc{K}}_\epsilon)\wt{\mc{G}}_\epsilon\\
  \mc{G}_+' &= \wt{\mc{K}}_\epsilon\mc{G}_{0+} +(I+\wt{\mc{K}}_\epsilon)\wt{\mc{G}}_+\,.
\end{aligned}
\end{equation}
Applying the estimates \eqref{eq:Kbds1}, \eqref{eq:Kbds2}, \eqref{eq:Gests1}, \eqref{eq:Gests2}, and \eqref{eq:IplusKinv}, we obtain Proposition \ref{prop:TDP_reform}.
\end{proof}

\subsection{Solvability of tension determination problem}\label{subsec:qual_TDP}
Here we prove solvability of the reformulated tension determination problem of Proposition \ref{prop:TDP_reform} and derive a coarse bound for the solution. We begin with general invertibility of the operator $I+\mc{K}$ appearing on the left hand side of \eqref{eq:TDP_reform}. In particular, given $h(s)\in C^{1,\alpha}(\T)$, we consider solving the equation 
\begin{equation}\label{eq:TDP_h}
(I+\mc{K})[\psi] = h(s)
\end{equation}
for $\psi$. Noting that $\mc{K}$ is compact from $C^{1,\alpha}(\T)$ to $C^{1,\alpha}(\T)$, invertibility follows by the Fredholm alternative provided the nullspace of $I+\mc{K}$ is trivial. Unpacking, we note that the operator $I+\mc{K}$ may be written as 
\begin{equation}\label{eq:unpack_IK}
  I+\mc{K} = (I+\mc{K}_\epsilon)^{-1}(I- \overline{\mc{L}}_\epsilon^{\rm tang}\p_{ss})^{-1}\mc{Q}_\epsilon
\end{equation}
where $\mc{K}_\epsilon$ is as in \eqref{eq:Keps_def} and $\overline{\mc{L}}_\epsilon^{\rm tang}$ is as in \eqref{eq:tang_and_nor}. Given the bound \eqref{eq:Kbds1} for $\mc{K}_\epsilon$ and Lemma \ref{lem:TDP_auxlem} for $(I-\overline{\mc{L}}_\epsilon^{\rm tang}\p_{ss})^{-1}$, it will suffice to show that the nullspace of $\mc{Q}_\epsilon$ is trivial.
After incorporating the curve dependence into the resulting bound, we obtain the following proposition. 
\begin{proposition}\label{prop:tdp_qual}
  Given $\Sigma_\epsilon$ as in \eqref{eq:SigmaEps}-\eqref{eq:rstar} with centerline $\X\in C^{3,\alpha}(\T)$, $0<\alpha<1$, and given $h\in C^{1,\alpha}(\T)$, there exists a unique solution $\psi\in C^{1,\alpha}(\T)$ to the equation \eqref{eq:TDP_h} which satisfies
  \begin{equation}\label{eq:tauh_bd}
  \norm{\psi}_{C^{1,\alpha}(\T)}\le c(\epsilon,\norm{\X}_{C^{3,\alpha}},c_\Gamma)\norm{h}_{C^{1,\alpha}(\T)}\,.
  \end{equation}

Moreover, given nearby filaments with centerlines $\X^{(a)}$, $\X^{(b)}$ in $C^{3,\alpha}(\T)$, the difference $\psi^{(a)}-\psi^{(b)}= \big((I+\mc{K}^{(a)})^{-1}-(I+\mc{K}^{(b)})^{-1}\big)[h]$ satisfies  
  \begin{equation}\label{eq:tauh_lip}
  \norm{\psi^{(a)}-\psi^{(b)}}_{C^{1,\alpha}(\T)}\le c(\epsilon,\|\X^{(a)}\|_{C^{3,\alpha}},\|\X^{(b)}\|_{C^{3,\alpha}},c_\Gamma)\norm{\X^{(a)}-\X^{(b)}}_{C^{3,\alpha}}\norm{h}_{C^{1,\alpha}(\T)}
  \end{equation}
  for $\alpha^+>\alpha$.
\end{proposition}

\begin{proof}
We begin by proving the solvability claim along with the qualitative bound 
\begin{equation}\label{eq:qual_bd}
  \norm{\psi}_{C^{1,\alpha}(\T)}\le c(\X,\epsilon)\norm{h}_{C^{1,\alpha}(\T)}
\end{equation}
where $\X$ is any $C^{2,\gamma^+}$ curve satisfying \eqref{eq:star_norm}.
As mentioned, given the form \eqref{eq:unpack_IK} of $I+\mc{K}$, the proof of solvability of \eqref{eq:TDP_h} amounts to verifying that the operator $\mc{Q}_\epsilon$ has only a trivial kernel for any filament $\Sigma_\epsilon$ as in \eqref{eq:SigmaEps}-\eqref{eq:rstar}.

Note that if $\mc{Q}_\epsilon[\psi]=0$, then, using the definition \eqref{eq:Qeps} of $\mc{Q}_\epsilon$, we have 
\begin{equation}
0= \int_\T \big(\mc{L}_\epsilon[(\psi\X_s)_s]\big)_s\cdot\X_s\,\phi(s)\,ds 
= -\int_\T \mc{L}_\epsilon[(\psi\X_s)_s]\cdot\big(\X_s\,\phi(s)\big)_s\,ds
\end{equation}
for any $\phi\in C^\infty(\T)$; in particular, either $\mc{L}_\epsilon[(\psi\X_s)_s]\perp\{\X_s,\X_{ss}\}$ or $\mc{L}_\epsilon[(\psi\X_s)_s]$ is a constant vector. In the first case, using the energy identity \eqref{eq:energyID} with $\bm{f}=(\psi\X_s)_s$, we have
\begin{equation}
\int_{\Omega_\epsilon}2\abs{\E(\bu)}^2\,d\bx = \int_\T \mc{L}_\epsilon[(\psi\X_s)_s]\cdot \big(\psi\X_{ss}+\psi_s\X_s \big)\,ds =0\,,
\end{equation}
so, by the Korn inequality in $\Omega_\epsilon$ \cite[Lemma 2.6]{closed_loop} and the decay of $\abs{\bu}$ as $\abs{\bx}\to\infty$, we have $\bu\equiv 0$ throughout $\Omega_\epsilon$ and hence $\bu\big|_{\Gamma_\epsilon}=\mc{L}_\epsilon[(\psi\X_s)_s]=0$. By uniqueness of solutions to the slender body boundary value problem (Proposition \ref{prop:SB_BVP}), we must then have $(\psi\X_s)_s=0$. Since $(\psi\X_s)_s=\psi\X_{ss}+\psi_s\X_s$, using that $\X_{ss}\perp\X_s$ and $\abs{\X_s}=1$, we obtain that $\psi=0$. 
Similarly, in the second case, if $\mc{L}_\epsilon[(\psi\X_s)_s]$ is constant, we have
\begin{equation}
\int_{\Omega_\epsilon}2\abs{\E(\bu)}^2\,d\bx = \mc{L}_\epsilon[(\psi\X_s)_s]\cdot\int_\T (\psi\X_s)_s\,ds =0\,,
\end{equation}
by periodicity, leading again to the conclusion that $\psi=0$. The qualitative bound \eqref{eq:qual_bd} follows from the mapping properties of $I+\mc{K}$ in Proposition \ref{prop:TDP_reform}.

We next consider the dependence of the bound \eqref{eq:qual_bd} on the curve $\X$. Letting $\norm{\X}_{C^{3,\alpha}}\le M$, we claim that the bound \eqref{eq:qual_bd} is in fact uniform in $M$. If not, fixing $h\in C^{1,\alpha}$ with $\norm{h}_{C^{1,\alpha}}=1$, we may select a sequence of curves $\X_j$ satisfying $\norm{\X_j}_{C^{3,\alpha}}\le M$ and $\abs{\X}_\star\ge c_\Gamma$ such that
\begin{equation}
  \norm{(I+\mc{K}_j)^{-1}[h]}_{C^{1,\alpha}}>j\,.
\end{equation}
Then, letting
\begin{equation}
  \varphi_j = \frac{h}{\norm{(I+\mc{K}_j)^{-1}[h]}_{C^{1,\alpha}}}\,, \quad
  \eta_j = \frac{(I+\mc{K}_j)^{-1}[h]}{\norm{(I+\mc{K}_j)^{-1}[h]}_{C^{1,\alpha}}}\,,
\end{equation}
we have that $(I+\mc{K}_j)[\eta_j]=\varphi_j\to 0$ in $C^{1,\alpha}$. By Proposition \ref{prop:TDP_reform}, we have $\mc{K}_j[\eta_j]\in C^{1,\gamma}$, $\gamma>\alpha$, with a uniform bound in $j$. Therefore, along some subsequence $j_\ell$, we have
\begin{equation}
  \eta_{j_\ell} = \varphi_{j_\ell} - \mc{K}_{j_\ell}[\eta_{j_\ell}]\to \eta_\infty\,, \qquad \norm{\eta_\infty}_{C^{1,\alpha}}=1\,.
\end{equation}
Furthermore, since $\norm{\X_j}_{C^{3,\alpha}}\le M$, there exists a subsequence $\X_{j_k}$ of curves converging strongly in $C^{2,\gamma^+}$ to some limit curve $\X_{\infty}$ satisfying $\abs{\X}_\infty\ge c_\Gamma$.
By Proposition \ref{prop:TDP_reform}, the corresponding operators $\mc{K}$ satisfy
\begin{equation} 
\norm{(\mc{K}_{j_k}-\mc{K}_{\infty})[\psi]}_{C^{1,\gamma}}\le c(\epsilon,M,c_\Gamma)\norm{\X_{j_k}-\X_\infty}_{C^{2,\gamma^+}}\norm{\psi}_{C^{1,\alpha}}
\to 0 
\end{equation}
as $j_k\to\infty$. In particular, by a diagonalization argument, 
\begin{equation}
  \mc{K}_j[\eta_j]\to  \mc{K}_\infty[\eta_\infty]
\end{equation}
along a subsequence, and $(I+\mc{K}_\infty)[\eta_\infty]=0$. By the qualitative solvability \eqref{eq:qual_bd}, we have $\eta_\infty=0$, contradicting $\norm{\eta_\infty}_{C^{1,\alpha}}=1$.

Finally, we turn to the Lipschitz bound \eqref{eq:tauh_lip}. Given two nearby filaments with centerlines $\X^{(a)},\X^{(b)}\in C^{3,\alpha}$, we seek a bound for the corresponding tension difference $\psi^{(a)}-\psi^{(b)}$, where $\psi^{(a)}$ and $\psi^{(b)}$ satisfy
\begin{equation}
  (I+\mc{K}^{(a)})[\psi^{(a)}] = h(s) = (I+\mc{K}^{(b)})[\psi^{(b)}]\,.
\end{equation}
Rearranging, we may write 
\begin{equation}
\begin{aligned}
  \psi^{(a)}-\psi^{(b)}&= \mc{K}^{(b)}[\psi^{(b)}] - \mc{K}^{(a)}[\psi^{(a)}]\\
  &= (\mc{K}^{(b)}-\mc{K}^{(a)})[\psi^{(b)}] - \mc{K}^{(a)}[\psi^{(a)}-\psi^{(b)}]\,,
\end{aligned}
\end{equation}
or, moving the second term on the right hand side to the left hand side,
\begin{equation}
  (I+\mc{K}^{(a)})[\psi^{(a)}-\psi^{(b)}] = (\mc{K}^{(b)}-\mc{K}^{(a)})[\psi^{(b)}]\,.
\end{equation}
Using \eqref{eq:tauh_bd} and the Lipschitz bound of Proposition \ref{prop:TDP_reform}, we may then estimate 
\begin{equation}
\begin{aligned}
  \norm{\psi^{(a)}-\psi^{(b)}}_{C^{1,\alpha}} &= \norm{(I+\mc{K}^{(a)})^{-1}(\mc{K}^{(b)}-\mc{K}^{(a)})[\psi^{(b)}]}_{C^{1,\alpha}}\\
  &\le c(\epsilon,\|\X^{(a)}\|_{C^{3,\alpha}},c_\Gamma)\norm{(\mc{K}^{(b)}-\mc{K}^{(a)})[\psi^{(b)}]}_{C^{1,\alpha}}\\
  &\le c(\epsilon,\|\X^{(a)}\|_{C^{3,\alpha}},\|\X^{(b)}\|_{C^{3,\alpha}},c_\Gamma)\norm{\X^{(a)}-\X^{(b)}}_{C^{3,\alpha}}\norm{\psi^{(b)}}_{C^{1,\alpha}}\\
  &\le c(\epsilon,\|\X^{(a)}\|_{C^{3,\alpha}},\|\X^{(b)}\|_{C^{3,\alpha}},c_\Gamma)\norm{\X^{(a)}-\X^{(b)}}_{C^{3,\alpha}}\norm{h}_{C^{1,\alpha}}\,.
\end{aligned}
\end{equation}
\end{proof}

\begin{remark}\label{rem:Peskin_TDP}
Proposition \ref{prop:tdp_qual} may be contrasted with the tension determination problem for the inextensible 2D Peskin problem \cite{kuo2023tension}, where, for a circular filament, the analogue to the operator $\mc{Q}_\epsilon$ does have a nontrivial kernel consisting of constants. 
In the 2D Peskin problem, any closed filament separates the domain into an interior and exterior component. In the special case of a circular filament, the tension cannot be uniquely determined because any constant may be added to the pressure of the fluid enclosed by filament, and this additional constant pressure jump between the interior and exterior may be absorbed into the filament tension.  

More directly, in the 2D Peskin problem, the operator $\mc{L}_\epsilon$ in \eqref{eq:Qeps} is replaced with the 2D Stokes single layer operator $\mc{S}_{2D}$: 
\begin{equation}
\begin{aligned}
\mc{Q}_{\rm Peskin}[\tau] &=  \big(\mc{S}_{2D}[(\tau\X_s)_s]\big)_s\cdot\X_s\,, \\
\mc{S}_{2D}[\bm{f}](s) &:= \frac{1}{4\pi}\int_{\T}\bigg(-\log(\abs{\bR}){\bf I}+\frac{\bR\bR^{\rm T}}{|\bR|^2}\bigg)\bm{f}(s')\,ds'\,, \quad \bR = \X(s)-\X(s')\,.
\end{aligned}
\end{equation} 
Here ${\bf I}$ is the $2\times2$ identity matrix.
The Stokes single layer operator along an interface in both 2D and 3D has a well-known kernel, namely, the vector field normal to the interface \cite{pozrikidis1992boundary}.
For a circle, we may write $\X_{ss}=\kappa\be_{\rm n}$ for constant $\kappa$ and $\be_{\rm n}$ the Frenet normal vector to the curve. Since the interface is truly 1D in the Peskin problem, the vector field given by $\be_{\rm n}$ is exactly the vector field normal to the interface, and hence, for any constant $c$, 
\begin{equation}
\mc{S}_{2D}[c\X_{ss}] = c\kappa\,\mc{S}_{2D}[\be_{\rm n}] = 0\,.
\end{equation}

In contrast, $\mc{L}_\epsilon[\be_{\rm n}]$ does not vanish even for the planar circle. In a sense, this makes the tension determination problem more straightforward for the 3D slender body problem than for the 2D Peskin problem, since any closed filament centerline satisfying \eqref{eq:star_norm} may be treated by the same arguments. 
\end{remark}


\begin{corollary}[Solvability of tension determination problem]\label{cor:TDP_solve}
Let $\X\in C^{3,\alpha}(\T)$, $0<\alpha<1$, be the centerline of a filament $\Sigma_\epsilon$ as in \eqref{eq:SigmaEps}-\eqref{eq:rstar}. Given $\bm{g}\in C^{0,\alpha}(\T)$, the reformulated tension determination problem of Proposition \ref{prop:TDP_reform} admits a unique solution $\tau\in C^{1,\alpha}(\T)$ satisfying 
  \begin{equation}
    \norm{\tau}_{C^{1,\alpha}(\T)} \le c(\epsilon,\norm{\X}_{C^{3,\alpha}},c_\Gamma)\norm{\bm{g}}_{C^{0,\alpha}}\,.
  \end{equation}
Moreover, given nearby filaments with centerlines $\X^{(a)}$, $\X^{(b)}$ in $C^{3,\alpha}(\T)$, the difference between the solutions to the corresponding tension determination problems satisfies 
  \begin{equation}
    \norm{\tau^{(a)}-\tau^{(b)}}_{C^{1,\alpha}(\T)} \le c(\epsilon,\|\X^{(a)}\|_{C^{3,\alpha}},\|\X^{(b)}\|_{C^{3,\alpha}},c_\Gamma)\norm{\X^{(a)}-\X^{(b)}}_{C^{3,\alpha}(\T)}\norm{\bm{g}}_{C^{0,\alpha}(\T)}\,.
  \end{equation}  
\end{corollary}

\begin{proof}
Given the reformulation of the tension determination problem in Proposition \ref{prop:TDP_reform}, we may write 
\begin{equation}
  \tau = (I+\mc{K})^{-1}\big(\mc{G}_0+ \mc{G}_\epsilon + \mc{G}_+'\big)[\bm{g}] \,.
\end{equation}
Using Proposition \ref{prop:tdp_qual}, we may then estimate
\begin{equation}
\begin{aligned}
  \norm{\tau}_{C^{1,\alpha}}&\le c(\epsilon,\norm{\X}_{C^{3,\alpha}},c_\Gamma)\norm{(\mc{G}_0+ \mc{G}_\epsilon + \mc{G}_+')[\bm{g}]}_{C^{1,\alpha}}\\
  &\le c(\epsilon,\norm{\X}_{C^{3,\alpha}},c_\Gamma)\norm{\bm{g}}_{C^{0,\alpha}}\,,
\end{aligned}
\end{equation}
where we have used the coarsest bounds for $\bm{g}$ in Proposition \ref{prop:TDP_reform}.

Furthermore, given two nearby filaments with centerlines $\X^{(a)}$ and $\X^{(b)}$, we may bound
\begin{equation}
\begin{aligned}
  \norm{\tau^{(a)}-\tau^{(b)}}_{C^{1,\alpha}} &\le 
  \norm{ \big((I+\mc{K}^{(a)})^{-1}-(I+\mc{K}^{(b)})^{-1}\big)\big(\mc{G}_0^{(a)}+ \mc{G}_\epsilon^{(a)} + \mc{G}_+^{'(a)}\big)[\bm{g}] }_{C^{1,\alpha}}\\
  &\quad + \norm{ (I+\mc{K}^{(b)})^{-1}\big(\mc{G}_0^{(a)}-\mc{G}_0^{(b)}+ \mc{G}_\epsilon^{(a)}-\mc{G}_\epsilon^{(b)} + \mc{G}_+^{'(a)}-\mc{G}_+^{'(b)}\big)[\bm{g}] }_{C^{1,\alpha}}\\
  &\hspace{-2cm}\le c(\epsilon,\|\X^{(a)}\|_{C^{3,\alpha}},\|\X^{(b)}\|_{C^{3,\alpha}},c_\Gamma)\norm{\X^{(a)}-\X^{(b)}}_{C^{3,\alpha}}\norm{\big(\mc{G}_0^{(a)}+ \mc{G}_\epsilon^{(a)} + \mc{G}_+^{'(a)}\big)[\bm{g}] }_{C^{1,\alpha}}\\
  &\hspace{-2cm}\quad + c(\epsilon,\norm{\X^{(b)}}_{C^{3,\alpha}},c_\Gamma)\norm{\big(\mc{G}_0^{(a)}-\mc{G}_0^{(b)}+ \mc{G}_\epsilon^{(a)}-\mc{G}_\epsilon^{(b)} + \mc{G}_+^{'(a)}-\mc{G}_+^{'(b)}\big)[\bm{g}] }_{C^{1,\alpha}}\\
  &\hspace{-2cm}\le c(\epsilon,\|\X^{(a)}\|_{C^{3,\alpha}},\|\X^{(b)}\|_{C^{3,\alpha}},c_\Gamma)\norm{\X^{(a)}-\X^{(b)}}_{C^{3,\alpha}}\norm{\bm{g}}_{C^{0,\alpha}}\,,
\end{aligned}
\end{equation}
where we have used Propositions \ref{prop:TDP_reform} and \ref{prop:tdp_qual}.
\end{proof}

\subsection{Proof of Theorem \ref{thm:TDP}}\label{subsec:pf_TDP}
We may finally combine the results of sections \ref{subsec:decomp_Q}-\ref{subsec:qual_TDP} to obtain Theorem \ref{thm:TDP}.
Using Proposition \ref{prop:TDP_reform}, we write
\begin{equation}
\tau = (\mc{G}_0+\mc{G}_\epsilon+\mc{G}_+')[\bm{g}] - \mc{K}[\tau]\,.
\end{equation}
By Proposition \ref{prop:TDP_reform} and Corollary \ref{cor:TDP_solve}, for any $\alpha<\gamma<1$, we may bound 
\begin{equation}
\begin{aligned}
  \norm{\mc{K}[\tau]}_{C^{1,\gamma}} &\le c(\epsilon,\norm{\X}_{C^{3,\alpha}},c_\Gamma)\norm{\tau}_{C^{1,\alpha}} 
  \le c(\epsilon,\norm{\X}_{C^{3,\alpha}},c_\Gamma)\norm{\bm{g}}_{C^{0,\alpha}}\,.
\end{aligned}
\end{equation}
Furthermore, given two filaments with centerlines $\X^{(a)}$, $\X^{(b)}$ in $C^{3,\alpha}(\T)$, we have 
\begin{equation}
\begin{aligned}
  \norm{\mc{K}^{(a)}[\tau^{(a)}]-\mc{K}^{(b)}[\tau^{(b)}]}_{C^{1,\gamma}} &\le \norm{(\mc{K}^{(a)}-\mc{K}^{(b)})[\tau^{(b)}]}_{C^{1,\gamma}} + \norm{\mc{K}^{(a)}[\tau^{(a)}-\tau^{(b)}]}_{C^{1,\gamma}} \\
  &\le c(\epsilon,\|\X^{(a)}\|_{C^{3,\alpha}},\|\X^{(b)}\|_{C^{3,\alpha}},c_\Gamma)\norm{\X^{(a)}-\X^{(b)}}_{C^{3,\alpha}}\|\tau^{(b)}\|_{C^{1,\alpha}}\\
  &\quad + c(\epsilon,\|\X^{(a)}\|_{C^{3,\alpha}},c_\Gamma)\norm{\tau^{(a)}-\tau^{(b)}}_{C^{1,\alpha}}\\
  &\le c(\epsilon,\|\X^{(a)}\|_{C^{3,\alpha}},\|\X^{(b)}\|_{C^{3,\alpha}},c_\Gamma)\norm{\X^{(a)}-\X^{(b)}}_{C^{3,\alpha}}\|\bm{g}\|_{C^{0,\alpha}}\,.
\end{aligned}
\end{equation}

We may therefore absorb the operator $\mc{K}[\tau]$ into the more regular right hand side remainder terms. In particular, defining $\mc{G}_+[\bm{g}]= \mc{G}_+'[\bm{g}]+\mc{K}[\tau]$, we obtain Theorem \ref{thm:TDP}. 
\hfill $\square$

\subsection{Proof of Corollary \ref{cor:TDP_Xssss}}\label{subsec:cor_TDP_Xssss}
Given an inextensible curve $\X\in C^{4,\alpha}(\T)$, upon differentiating the inextensibility constraint $\abs{\X_s}^2=1$, we obtain the identities 
\begin{equation}\label{eq:X_IDs}
  \X_s\cdot\X_{ss}=0\,, \quad 
  \X_s\cdot\X_{sss}=-\abs{\X_{ss}}^2\,, \quad 
  \X_s\cdot\X_{ssss}=-3\X_{ss}\cdot\X_{sss}\,.
\end{equation}
Consider the solution to the tension determination problem \eqref{eq:TDP} with data $\bm{g}=\X_{ssss}$. Using \eqref{eq:X_IDs}, the main term $\mc{G}_0[\X_{ssss}]$ in the decomposition of Theorem \ref{thm:TDP} may be rewritten as 
\begin{equation}
  \mc{G}_0[\X_{ssss}] = -(I- \overline{\mc{L}}_\epsilon^{\rm tang}\p_{ss})^{-1}\p_s\overline{\mc{L}}_\epsilon^{\rm tang}\bigg[\X_{ss}\cdot\X_{sss}-\int_{\T}\X_{ss}\cdot\X_{sss}\,ds \bigg]\,.
\end{equation}
Using Lemma \ref{lem:TDP_auxlem}, for $\alpha<\gamma<1$, we may then estimate 
\begin{equation}
\begin{aligned}
  \norm{\mc{G}_0[\X_{ssss}]}_{C^{1,\gamma}} &\le c(\epsilon)\norm{\overline{\mc{L}}_\epsilon^{\rm tang}\bigg[\X_{ss}\cdot\X_{sss}-\int_{\T}\X_{ss}\cdot\X_{sss}\,ds \bigg]}_{C^{1,\gamma}}\\
  &\le c(\epsilon)\norm{\X_{ss}\cdot\X_{sss}}_{C^{0,\gamma}}
  \le c(\epsilon,\norm{\X}_{C^{2,\gamma}})\norm{\X}_{C^{3,\gamma}}\,.
\end{aligned}
\end{equation}

Furthermore, given two inextensible filaments with centerlines $\X^{(a)}$, $\X^{(b)}$ in $C^{4,\alpha}(\T)$, the difference between the main terms in the solutions to the corresponding tension determination problems satisfies 
\begin{equation}
\begin{aligned}
  \norm{\mc{G}_0^{(a)}[\X_{ssss}^{(a)}]-\mc{G}_0^{(b)}[\X_{ssss}^{(b)}]}_{C^{1,\gamma}} 
  &= \bigg\|(I- \overline{\mc{L}}_\epsilon^{\rm tang}\p_{ss})^{-1}\p_s\overline{\mc{L}}_\epsilon^{\rm tang}\bigg[\X_{ss}^{(a)}\cdot\X_{sss}^{(a)}-\X_{ss}^{(b)}\cdot\X_{sss}^{(b)}\\
  &\qquad -\int_{\T}\big(\X_{ss}^{(a)}\cdot\X_{sss}^{(a)}-\X_{ss}^{(b)}\cdot\X_{sss}^{(b)}\big)\,ds \bigg]\bigg\|_{C^{1,\gamma}}\\
  &\le c(\epsilon)\norm{\X_{ss}^{(a)}\cdot\X_{sss}^{(a)}-\X_{ss}^{(b)}\cdot\X_{sss}^{(b)}}_{C^{0,\gamma}}\\
  &\hspace{-2cm}\le c(\epsilon)\big(\|\X^{(a)}\|_{C^{3,\gamma}}\|\X^{(a)}-\X^{(b)}\|_{C^{2,\gamma}} + \|\X^{(b)}\|_{C^{2,\gamma}}\|\X^{(a)}-\X^{(b)}\|_{C^{3,\gamma}} \big)\,.
\end{aligned}
\end{equation}
Thus the contribution from $\mc{G}_0[\X_{ssss}]$ may be absorbed into the more regular terms $\tau_+[\X_{ssss}]$. The estimates \eqref{eq:tauXssss}, \eqref{eq:tauXssss_lip} for $\tau_\epsilon[\X_{ssss}]$ and $\tau_+[\X_{ssss}]$ follow directly from taking $\bm{g}=\X_{ssss}$ in the terms $\mc{G}_\epsilon$ and $\mc{G}_+$ from Theorem \ref{thm:TDP}.
\hfill $\square$

\section{Evolution}\label{sec:evolution}
Given the solution theory and bounds of Theorem \ref{thm:TDP} and Corollary \ref{cor:TDP_Xssss} for the tension determination problem, we may now turn to the proof of our main result, Theorem \ref{thm:main}. We begin by using the decomposition of Corollary \ref{cor:four_derivs} to rewrite the evolution \eqref{eq:evolution} as 
\begin{equation}\label{eq:X_evolution}
\begin{aligned}
\frac{\p\X}{\p t} 
&= - ({\bf I}+\mc{R}_{\rm n,\epsilon})\overline{\mc{L}}_\epsilon[\X_{ssss}] - \wt{\mc{R}}_{\rm n,+}[\X_{ssss}] + \mc{L}_\epsilon\big[(\tau\X_s)_s\big]\,,
\end{aligned}
\end{equation}
where, here and throughout this section, we have $\tau=\tau[\X_{ssss}]$. This dependence is crucial for the improvement in regularity of Corollary \ref{cor:TDP_Xssss}, but we will not denote this explicitly to avoid clutter.
Due to this decomposition, we see that the main part of the evolution \eqref{eq:X_evolution} is given by $\overline{\mc{L}}_\epsilon\p_s^4$, which, by Lemma \ref{lem:straight_Leps}, has the effect of taking 3 derivatives, making the evolution third order parabolic. 

Using Duhamel's formula, we may then write
\begin{equation}\label{eq:duhamel}
\begin{aligned}
\X(s,t) &= e^{-t\overline{\mc{L}}_\epsilon\p_s^4}\X^{\rm in}(s) \\
&\quad + \int_0^t e^{-(t-t')\overline{\mc{L}}_\epsilon\p_s^4}\bigg(-\mc{R}_{\rm n,\epsilon}\big[\overline{\mc{L}}_\epsilon[\X_{ssss}]\big] - \wt{\mc{R}}_{\rm n,+}[\X_{ssss}]+ \mc{L}_\epsilon[(\tau\X_s)_s]\bigg)\,dt'\,.
\end{aligned}
\end{equation}

Since the behavior of $\overline{\mc{L}}_\epsilon\p_s^4$ is given by an explicit Fourier multiplier, we may directly obtain the following estimates for the semigroup $e^{-t\overline{\mc{L}}_\epsilon\p_s^4}$. 
\begin{lemma}[Semigroup estimates]\label{lem:semigroup}
Let $n,m$ be nonnegative integers and $\alpha,\gamma\in(0,1)$ with $n+\alpha\ge m+\gamma$.
For $\bm{V}(s)\in C^{m,\gamma}(\T)$, we have 
\begin{equation}\label{eq:semigroup}
\norm{e^{-t\overline{\mc{L}}_\epsilon\p_s^4}\bm{V}}_{C^{n,\alpha}(\T)} 
\le c\, t^{-\frac{n+\alpha - (m +\gamma)}{3}} \norm{\bm{V}}_{C^{m,\gamma}(\T)}\,.
\end{equation}
In each case, $c$ may be taken to be an absolute constant independent of $\epsilon$.
If, in addition, $\bm{V}$ belongs to the little H\"older space $h^{m,\gamma}(\T)$, then $e^{-t\overline{\mc{L}}_\epsilon\p_s^4}\bm{V}\in C([0,+\infty),h^{m,\gamma}(\T))$.
\end{lemma}
The proof of the bound \eqref{eq:semigroup} appears in \cite[Lemma 2.2]{ohm2024free}. The final statement is shown in Appendix \ref{sec:h4alpha}.

Due to the presence of remainder terms $\mc{R}_{\rm n,\epsilon}$ and $\tau_\epsilon$ in \eqref{eq:duhamel} which are small in $\epsilon$ but not more regular than the leading $\overline{\mc{L}}_\epsilon\p_s^4$ term in \eqref{eq:X_evolution}, Lemma \ref{lem:semigroup} will not be sufficient to close a fixed point argument. We will additionally require a lemma making use of the maximal smoothing power of the semigroup $e^{-t\overline{\mc{L}}_\epsilon\p_s^4}$ while keeping track of its $\epsilon$-dependence. Given $T>0$ and $0\le k\in \Z$, we first define the function spaces 
\begin{equation}\label{eq:Ykspace_def}
\mc{Y}_k(T) = \big\{ \bm{V}\in C\big([0,T];h^{k,\alpha}(\T)\big)\,:\,\norm{\bm{V}}_{\mc{Y}_k}<\infty \big\}\,, \qquad \norm{\cdot}_{\mc{Y}_k}=\sup_{t\in [0,T]}\norm{\cdot}_{C^{k,\alpha}(\T)} \,,
\end{equation}
where we recall that $h^{k,\alpha}(\T)$ denotes the closure of smooth functions in $C^{k,\alpha}(\T)$.
We have the following result.
\begin{lemma}[Maximal regularity bound]\label{lem:max_reg}
Given $\bm{g}\in L^\infty(0,T; C^{1,\alpha}(\T))$, let
\begin{equation}
\bm{h}(s,t)= \int_0^t e^{-(t-t')\overline{\mc{L}}_\epsilon\p_s^4}\;\bm{g}(s,t')\,dt'\,.
\end{equation}
Then, for some absolute constant $c>0$,
\begin{equation}\label{eq:max_semi}
\norm{\bm{h}}_{L^\infty(0,T;C^{4,\alpha}(\T))} \le \big(\epsilon + T^{1/4}\abs{\log\epsilon}^{-3/4}+T \big)\,c\norm{\bm{g}}_{L^\infty(0,T;C^{1,\alpha}(\T))} \,.
\end{equation}
If, in addition, $\bm{g}\in \mc{Y}_1(T)$, then $\bm{h}\in \mc{Y}_4(T)$.
\end{lemma}
The proof of the bound \eqref{eq:max_semi} appears in \cite[Lemma 2.3]{ohm2024free}, and we show the final statement in Appendix \ref{sec:h4alpha}.

We will close a fixed point argument in the space
\begin{equation}\label{eq:fspace}
   \mc{Y}_4(T) \cap \bigg\{ \norm{\X-e^{-t\overline{\mc{L}}_\epsilon\p_s^4}\X^{\rm in}}_{\mc{Y}_1(T)}\le \frac{1}{2}\abs{\X^{\rm in}}_\star \bigg\} 
\end{equation}
for some choice of $\epsilon$, $T>0$. Note that the second condition serves to propagate the non-intersection condition \eqref{eq:star_norm}. 
In particular, we have
\begin{equation}
\begin{aligned}
  \abs{\X(\cdot,t)}_\star &\ge \abs{\X^{\rm in}}_\star - \norm{\X(\cdot,t)-\X^{\rm in}}_{C^{1,\alpha}} \\
  &\ge \abs{\X^{\rm in}}_\star - \norm{\X(\cdot,t)-e^{-t\overline{\mc{L}}_\epsilon\p_s^4}\X^{\rm in}}_{C^{1,\alpha}} - \norm{e^{-t\overline{\mc{L}}_\epsilon\p_s^4}\X^{\rm in}-\X^{\rm in}}_{C^{1,\alpha}} \,.
\end{aligned}
\end{equation}
For $\Y$ satisfying the linear equation
\begin{equation}
  \p_t\Y = -\overline{\mc{L}_\epsilon}\p_s^4\Y\,, \qquad \Y^{\rm in}= \X^{\rm in}\,,
\end{equation}
we note that at each $t$, we have 
\begin{equation}
  \norm{\p_t\Y}_{C^{1,\alpha}(\T)} \le c(\epsilon)\norm{\Y}_{C^{4,\alpha}(\T)} \le c(\epsilon)\norm{\X^{\rm in}}_{C^{4,\alpha}(\T)}\,. 
\end{equation}
We may then estimate
\begin{equation}
  \norm{\Y - \X^{\rm in}}_{C^{1,\alpha}(\T)} = \norm{\int_0^t\p_t\Y\,dt'}_{C^{1,\alpha}(\T)} \le c(\epsilon)\,t\,\norm{\X^{\rm in}}_{C^{4,\alpha}(\T)}\,.
\end{equation}
Thus we will have
\begin{equation}\label{eq:T1_def}
  \norm{e^{-t\overline{\mc{L}}_\epsilon\p_s^4}\X^{\rm in}-\X^{\rm in}}_{\mc{Y}_1} \le c(\epsilon)\,T\norm{\X^{\rm in}}_{C^{4,\alpha}(\T)} 
  \le \frac{1}{4}\abs{\X^{\rm in}}_\star
\end{equation}
for $T$ sufficiently small (depending on $\epsilon$). Therefore the second condition of \eqref{eq:fspace} guarantees that
\begin{equation}
  \sup_{t\in[0,T]} \abs{\X(\cdot,t)}_\star \ge \frac{1}{4}\abs{\X^{\rm in}}_\star\,.
\end{equation}

\subsection{Ball to ball}
Given $\X^{\rm in}\in h^{4,\alpha}(\T)$ with $\abs{\X^{\rm in}}_\star>0$, we define the map
\begin{equation}\label{eq:Lambda_def}
\begin{aligned}
  \Lambda[\X] &= e^{-t\overline{\mc{L}}_\epsilon\p_s^4}\X^{\rm in}(s) \\
&\quad + \int_0^t e^{-(t-t')\overline{\mc{L}}_\epsilon\p_s^4}\bigg(-\mc{R}_{\rm n,\epsilon}\big[\overline{\mc{L}}_\epsilon[\X_{ssss}]\big] - \wt{\mc{R}}_{\rm n,+}[\X_{ssss}]+ \mc{L}_\epsilon[(\tau\X_s)_s]\bigg)\,dt'
\end{aligned}
\end{equation}
and show that for some choice of $\epsilon$, $T>0$, $\Lambda$ maps the closed ball
\begin{equation}
  \mc{B}_{M,\star} = \bigg\{ \X\in \mc{Y}_4(T)\,: \, \norm{\X}_{\mc{Y}_4}\le M\,, \; \norm{\X-e^{-t\overline{\mc{L}}_\epsilon\p_s^4}\X^{\rm in}}_{\mc{Y}_1}\le \frac{1}{2}\abs{\X^{\rm in}}_\star \bigg\} 
\end{equation}
to itself.

We begin with the tension terms. Recalling that here $\tau=\tau[\X_{ssss}]$, we may use Corollary \ref{cor:TDP_Xssss} along with the $\mc{L}_\epsilon$ decomposition of Theorem \ref{thm:decomp} to write 
\begin{equation}\label{eq:tension_terms}
\begin{aligned}
  \mc{L}_\epsilon\big[(\tau\X_s)_s\big] &= \mc{L}_\epsilon\big[(\tau_\epsilon\X_s)_s\big] + \mc{L}_\epsilon\big[(\tau_+\X_s)_s\big] \\
  &= \underbrace{({\bf I}+\mc{R}_{\rm n,\epsilon})\left[\Phi\overline{\mc{L}}_\epsilon\big[\Phi^{-1}\big((\tau_\epsilon\X_s)_s\big)^\Phi_0\big] \right]}_{=: \,J_\epsilon[\X]} 
  + \underbrace{\mc{R}_{\rm n,+}[(\tau_\epsilon\X_s)_s]+ \mc{L}_\epsilon\big[(\tau_+\X_s)_s\big]}_{=: \,J_+[\X]}\,.
\end{aligned}
\end{equation}
Here the notation $(\cdot)_0^\Phi$ is as defined in \eqref{eq:subtract_mean}. We may obtain spatial estimates for the terms $J_\epsilon[\X]$ and $J_+[\X]$ as follows. Using Theorem \ref{thm:decomp}, we have
\begin{equation}
\begin{aligned}
  \norm{J_\epsilon[\X]}_{C^{1,\alpha}(\T)} &\le c(\norm{\X}_{C^{2,\alpha^+}},\abs{\X^{\rm in}}_\star)\norm{\overline{\mc{L}}_\epsilon\big[\Phi^{-1}\big((\tau_\epsilon\X_s)_s\big)^\Phi_0\big]}_{C^{1,\alpha}} \\
  &\le c(\norm{\X}_{C^{2,\alpha^+}},\abs{\X^{\rm in}}_\star)\,\abs{\log\epsilon}\epsilon^{-1}\norm{\big((\tau_\epsilon\X_s)_s\big)^\Phi_0}_{C^{0,\alpha}}\\
  &\le c(\norm{\X}_{C^{2,\alpha^+}},\abs{\X^{\rm in}}_\star)\,\abs{\log\epsilon}\epsilon^{-1}\norm{\tau_\epsilon}_{C^{1,\alpha}}\\
  &\le c(\norm{\X}_{C^{2,\alpha^+}},\abs{\X^{\rm in}}_\star)\,\epsilon^{-\alpha^+}\abs{\log\epsilon}^4\norm{\X}_{C^{4,\alpha}(\T)}\,,
\end{aligned}
\end{equation}
where we have also used Lemma \ref{lem:straight_Leps} in the second line and Corollary \ref{cor:TDP_Xssss} in the final line. In addition, we may bound
\begin{equation}\label{eq:Jplus}
\begin{aligned}
  \norm{J_+[\X]}_{C^{1,\gamma}(\T)} &\le
  \norm{\mc{R}_{\rm n,+}[(\tau_\epsilon\X_s)_s]}_{C^{1,\gamma}}+ \norm{\mc{L}_\epsilon\big[(\tau_+\X_s)_s\big]}_{C^{1,\gamma}}\\
  &\le c(\epsilon,\norm{\X}_{C^{3,\alpha}},\abs{\X^{\rm in}}_\star)\big(\norm{(\tau_\epsilon\X_s)_s}_{C^{0,\alpha}}
  + \norm{(\tau_+\X_s)_s}_{C^{0,\gamma}}\big)\\
  &\le c(\epsilon,\norm{\X}_{C^{3,\alpha}},\abs{\X^{\rm in}}_\star)\norm{\X}_{C^{4,\alpha}}\,,
\end{aligned}
\end{equation}
where we have used the mapping properties of $\mc{R}_{\rm n,+}$ and $\mc{L}_\epsilon$ from Theorem \ref{thm:decomp} and Lemma \ref{lem:straight_Leps} in the second line, and we have used the bounds for $\tau_\epsilon$ and $\tau_+$ from Corollary \ref{cor:TDP_Xssss} in the final line.

Using the estimate \eqref{eq:Jplus} in the Duhamel formula \eqref{eq:duhamel}, by Lemma \ref{lem:semigroup}, we have that the $J_+$ term satisfies 
\begin{equation}
\begin{aligned}
  \norm{\int_0^t e^{-(t-t')\overline{\mc{L}}_\epsilon\p_s^4}J_+[\X]\,dt'}_{C^{4,\alpha}(\T)}& \le c\int_0^t(t-t')^{-1+\frac{\gamma-\alpha}{3}}\norm{J_+[\X]}_{C^{1,\gamma}(\T)}\,dt'\\
  &\le \int_0^t(t-t')^{-1+\frac{\gamma-\alpha}{3}}c(\epsilon,\norm{\X}_{C^{3,\alpha}},\abs{\X^{\rm in}}_\star)\norm{\X}_{C^{4,\alpha}(\T)}\,dt'\\
  &\le t^{\frac{\gamma-\alpha}{3}}\,c(\epsilon,M,\abs{\X^{\rm in}}_\star)\norm{\X}_{\mc{Y}_4}
  \le t^{\frac{\gamma-\alpha}{3}}\,c(\epsilon,M,\abs{\X^{\rm in}}_\star)\,M\,.
\end{aligned}
\end{equation}
In addition, using Lemma \ref{lem:semigroup}, we may obtain the $C^{1,\alpha}(\T)$ bound
\begin{equation}
\begin{aligned}
  \norm{\int_0^t e^{-(t-t')\overline{\mc{L}}_\epsilon\p_s^4}J_+[\X]\,dt'}_{C^{1,\alpha}(\T)}& \le c\int_0^t(t-t')^{\frac{\gamma-\alpha}{3}}\norm{J_+[\X]}_{C^{1,\gamma}(\T)}\,dt'\\
  &\le \int_0^t(t-t')^{\frac{\gamma-\alpha}{3}}c(\epsilon,\norm{\X}_{C^{3,\alpha}},\abs{\X^{\rm in}}_\star)\norm{\X}_{C^{4,\alpha}(\T)}\,dt'\\
  &\le t^{1+\frac{\gamma-\alpha}{3}}\,c(\epsilon,M,\abs{\X^{\rm in}}_\star)\norm{\X}_{\mc{Y}_4}\\
  &\le t^{1+\frac{\gamma-\alpha}{3}}\,c(\epsilon,M,\abs{\X^{\rm in}}_\star)\,M\,.
\end{aligned}
\end{equation}
Furthermore, using Lemma \ref{lem:max_reg}, the term $J_\epsilon$ satisfies 
\begin{equation}
\begin{aligned}
  \norm{\int_0^t e^{-(t-t')\overline{\mc{L}}_\epsilon\p_s^4}J_\epsilon[\X]\,dt'}_{\mc{Y}_4} &\le (\epsilon + T^{1/4}\abs{\log\epsilon}^{-3/4}+T)\,c\norm{J_\epsilon[\X]}_{\mc{Y}_1}\\
  &\le (\epsilon + T^{1/4}\abs{\log\epsilon}^{-3/4}+T)\,c(M,\abs{\X^{\rm in}}_\star)\,\epsilon^{-\alpha^+}\abs{\log\epsilon}^4\norm{\X}_{\mc{Y}_4}\\
  &\le \epsilon^{1-\alpha^+}\abs{\log\epsilon}^4c(M,\abs{\X^{\rm in}}_\star)\,M\\
  &\qquad 
  + (T^{1/4}+T)\,c(\epsilon,M,\abs{\X^{\rm in}}_\star)\,M\,.
\end{aligned}
\end{equation}
Using Lemma \ref{lem:semigroup}, we additionally obtain the $C^{1,\alpha}(\T)$ bound
\begin{equation}
\begin{aligned}
  \norm{\int_0^t e^{-(t-t')\overline{\mc{L}}_\epsilon\p_s^4}J_\epsilon[\X]\,dt'}_{C^{1,\alpha}(\T)}& \le c\int_0^t\norm{J_\epsilon[\X]}_{C^{1,\alpha}(\T)}\,dt'
  \le t\,c(\epsilon,M,\abs{\X^{\rm in}}_\star)\,M\,.
\end{aligned}
\end{equation}

We next turn to the curve remainder terms $\wt{\mc{R}}_{\rm n,+}$ and $\mc{R}_{\rm n,\epsilon}$. Using Corollary \ref{cor:four_derivs} and Lemma \ref{lem:semigroup}, we may estimate $\wt{\mc{R}}_{\rm n,+}$ as
\begin{equation}
\begin{aligned}
  \norm{\int_0^t e^{-(t-t')\overline{\mc{L}}_\epsilon\p_s^4}\wt{\mc{R}}_{\rm n,+}[\X_{ssss}]\,dt'}_{C^{4,\alpha}(\T)}
  &\le c\int_0^t (t-t')^{-1+\frac{\gamma-\alpha}{3}}\norm{\wt{\mc{R}}_{\rm n,+}[\X_{ssss}]}_{C^{1,\gamma}(\T)}\,dt' \\
  &\hspace{-1cm}\le \int_0^t (t-t')^{-1+\frac{\gamma-\alpha}{3}}c(\epsilon,\norm{\X}_{C^{3,\gamma}},\abs{\X^{\rm in}}_\star)\norm{\X}_{C^{4,\alpha}(\T)}\,dt' \\
  &\hspace{-1cm}\le t^{\frac{\gamma-\alpha}{3}}\,c(\epsilon,M,\abs{\X^{\rm in}}_\star)\norm{\X}_{\mc{Y}_4}
  \le t^{\frac{\gamma-\alpha}{3}}\,c(\epsilon,M,\abs{\X^{\rm in}}_\star)\,M\,.
\end{aligned}
\end{equation}
In addition, by Corollary \ref{cor:four_derivs} and Lemma \ref{lem:max_reg}, the small-in-$\epsilon$ remainder $\mc{R}_{\rm n,\epsilon}$ satisfies 
\begin{equation}
\begin{aligned}
  &\norm{\int_0^t e^{-(t-t')\overline{\mc{L}}_\epsilon\p_s^4}\mc{R}_{\rm n,\epsilon}\big[\overline{\mc{L}}_\epsilon[\X_{ssss}]\big]\,dt'}_{\mc{Y}_4}
  \le (\epsilon+T^{1/4}\abs{\log\epsilon}^{-3/4}+T) \,c\norm{\mc{R}_{\rm n,\epsilon}\big[\overline{\mc{L}}_\epsilon[\X_{ssss}]\big]}_{\mc{Y}_1} \\
  &\qquad \le \epsilon^{1-\alpha^+}\abs{\log\epsilon}(\epsilon+T^{1/4}\abs{\log\epsilon}^{-3/4}+T)\,c(\norm{\X}_{\mc{Y}_4},\abs{\X^{\rm in}}_\star)\norm{\overline{\mc{L}}_\epsilon[\X_{ssss}]}_{\mc{Y}_1}\\
  &\qquad \le \epsilon^{1-\alpha^+}\abs{\log\epsilon}(\epsilon+T^{1/4}\abs{\log\epsilon}^{-3/4}+T)\,c(M,\abs{\X^{\rm in}}_\star)\big(\epsilon^{-1}\abs{\log\epsilon}\norm{\X}_{\mc{Y}_4} \big)\\
  &\qquad \le \epsilon^{1-\alpha^+}\abs{\log\epsilon}^2\,c(M,\abs{\X^{\rm in}}_\star)\,M +(T^{1/4}+T)\,c(\epsilon,M,\abs{\X^{\rm in}}_\star)\,M\,.
\end{aligned}
\end{equation}
For both remainders, we may obtain $C^{1,\alpha}(\T)$ bounds using Lemma \ref{lem:semigroup}:
\begin{equation}
\begin{aligned}
  \norm{\int_0^t e^{-(t-t')\overline{\mc{L}}_\epsilon\p_s^4}\wt{\mc{R}}_{\rm n,+}[\X_{ssss}]\,dt'}_{C^{1,\alpha}(\T)}
  &\le c\int_0^t (t-t')^{\frac{\gamma-\alpha}{3}}\norm{\wt{\mc{R}}_{\rm n,+}[\X_{ssss}]}_{C^{1,\gamma}(\T)}\,dt' \\
  &\le t^{1+\frac{\gamma-\alpha}{3}}\,c(\epsilon,M,\abs{\X^{\rm in}}_\star)\,M\,,\\
  \norm{\int_0^t e^{-(t-t')\overline{\mc{L}}_\epsilon\p_s^4}\mc{R}_{\rm n,\epsilon}[\X_{ssss}]\,dt'}_{C^{1,\alpha}(\T)}
  &\le c\int_0^t \norm{\mc{R}_{\rm n,\epsilon}[\X_{ssss}]}_{C^{1,\alpha}(\T)}\,dt' \\
  &\le t\,c(\epsilon,M,\abs{\X^{\rm in}}_\star)\,M\,.
\end{aligned}
\end{equation}

Finally, we use Lemma \ref{lem:semigroup} to estimate the contribution from $\X^{\rm in}$ in \eqref{eq:Lambda_def} as
\begin{equation}\label{eq:c_def}
  \norm{e^{-t\overline{\mc{L}}_\epsilon\p_s^4}\X^{\rm in}}_{\mc{Y}_4} \le c\norm{\X^{\rm in}}_{C^{4,\alpha}(\T)}\,, 
\end{equation}
where $c$ is an absolute constant. We will take $M=3c\norm{\X^{\rm in}}_{C^{4,\alpha}(\T)}$. 

Then, in total, we obtain the bounds 
\begin{equation}
\begin{aligned}
  &\norm{\Lambda[\X]}_{\mc{Y}_4} \le \bigg(\frac{1}{3}+ \epsilon^{1-\alpha^+}\abs{\log\epsilon}^4c(M,\abs{\X^{\rm in}}_\star)
  + (T^{\frac{\gamma-\alpha}{3}}+ T^{1/4}+T)\,c(\epsilon,M,\abs{\X^{\rm in}}_\star) \bigg)M\,,\\
  &\norm{\Lambda[\X]- e^{-t\overline{\mc{L}}_\epsilon\p_s^4}\X^{\rm in}}_{\mc{Y}_1} \le T(1+T^{\frac{\gamma-\alpha}{3}})\,c(\epsilon,M,\abs{\X^{\rm in}}_\star)\,M\,.
\end{aligned}  
\end{equation}
We first take $\epsilon$ small enough that both
\begin{equation}\label{eq:epscond1}
\begin{aligned}
  &\epsilon^{1-\alpha^+}\abs{\log\epsilon}^4c(M,\abs{\X^{\rm in}}_\star)\le \frac{1}{3}
  \quad \text{and} \\
  &\epsilon\le \min\bigg\{\frac{1}{2}\,\textstyle r_\star\big(\frac{|\X^{\rm in}|_\star}{4},M\big), \;\epsilon_{\rm n}\big(\frac{|\X^{\rm in}|_\star}{4},M\big),\;\epsilon_{\rm t}\big(\frac{|\X^{\rm in}|_\star}{4},M\big)\bigg\}
\end{aligned} 
\end{equation}
hold, where $r_\star$ is as in \eqref{eq:rstar}, $\epsilon_{\rm n}$ is as in Theorem \ref{thm:decomp}, and $\epsilon_{\rm t}$ is as in Theorem \ref{thm:TDP}. For this choice of $\epsilon$, we then take $T$ small enough that
\begin{equation}\label{eq:Tcond1}
  (T^{\frac{\gamma-\alpha}{3}}+ T^{1/4}+T)\,c(\epsilon,M,\abs{\X^{\rm in}}_\star)\le \frac{1}{3}\,, \qquad 
  T(1+T^{\frac{\gamma-\alpha}{3}})\,c(\epsilon,M,\abs{\X^{\rm in}}_\star)\,M \le \frac{1}{2}\abs{\X^{\rm in}}_\star\,,
\end{equation}
and, recalling \eqref{eq:T1_def},
\begin{equation}\label{eq:Tcond2}
  c(\epsilon)\,T\,M \le \frac{1}{4}\abs{\X^{\rm in}}_\star\,.
\end{equation}
We thus obtain that $\Lambda$ maps $\mc{B}_{M,\star}$ to itself.

\subsection{Contraction estimates}
We next show that for some choice of $\epsilon$ and $T$, the map $\Lambda$ is a contraction on $\mc{B}_{M,\star}$. Let $\X^{(a)}$, $\X^{(b)}$ denote the centerlines of two nearby filaments both belonging to $\mc{B}_{M,\star}$. We aim to bound $\Lambda[\X^{(a)}]-\Lambda[\X^{(b)}]$. 

We again begin with the tension terms, which we again split using the decomposition \eqref{eq:tension_terms}. As before, using the superscript $(\cdot)^{(a)}$ to denote terms corresponding to filament $(a)$ and $(\cdot)^{(b)}$ for filament $(b)$, we may bound the difference $J_\epsilon^{(a)}[\X^{(a)}]-J_\epsilon^{(b)}[\X^{(b)}]$ as
\begin{equation}\label{eq:Jeps_lip}
\begin{aligned}
  &\norm{J_\epsilon^{(a)}[\X^{(a)}]-J_\epsilon^{(b)}[\X^{(b)}]}_{C^{1,\alpha}(\T)} 
  \le \norm{(\mc{R}_{\rm n,\epsilon}^{(a)}-\mc{R}_{\rm n,\epsilon}^{(b)})\big[\Phi^{(a)}\overline{\mc{L}}_\epsilon[(\Phi^{(a)})^{-1}\big((\tau_\epsilon^{(a)}\X_s^{(a)})_s \big)_0^{\Phi^{(a)}}] \big]}_{C^{1,\alpha}}\\
  &\quad + c(\|\X^{(b)}\|_{C^{2,\alpha^+}},\abs{\X^{\rm in}}_\star)\norm{(\Phi^{(a)}-\Phi^{(b)})\overline{\mc{L}}_\epsilon\big[(\Phi^{(a)})^{-1}\big((\tau_\epsilon^{(a)}\X_s^{(a)})_s \big)_0^{\Phi^{(a)}}\big]}_{C^{1,\alpha}}\\
  &\quad + c(\|\X^{(b)}\|_{C^{2,\alpha^+}},\abs{\X^{\rm in}}_\star)\norm{\overline{\mc{L}}_\epsilon\big[(\Phi^{(a)})^{-1}\big((\tau_\epsilon^{(a)}\X_s^{(a)})_s \big)_0^{\Phi^{(a)}}- (\Phi^{(b)})^{-1}\big((\tau_\epsilon^{(b)}\X_s^{(b)})_s \big)_0^{\Phi^{(b)}}\big]}_{C^{1,\alpha}}\\
  &\le c(\|\X^{(a)}\|_{C^{2,\alpha^+}},\|\X^{(b)}\|_{C^{2,\alpha^+}},\abs{\X^{\rm in}}_\star)\times\\
  &\quad \times \bigg(\norm{\X^{(a)}-\X^{(b)}}_{C^{2,\alpha^+}}\norm{\overline{\mc{L}}_\epsilon\big[(\Phi^{(a)})^{-1}\big((\tau_\epsilon^{(a)}\X_s^{(a)})_s \big)_0^{\Phi^{(a)}}\big]}_{C^{1,\alpha}}\\
  &\quad + \epsilon^{-1}\abs{\log\epsilon}\norm{(\Phi^{(a)})^{-1}\big((\tau_\epsilon^{(a)}\X_s^{(a)})_s \big)_0^{\Phi^{(a)}}- (\Phi^{(b)})^{-1}\big((\tau_\epsilon^{(b)}\X_s^{(b)})_s \big)_0^{\Phi^{(b)}}}_{C^{0,\alpha}}\bigg) \\
  & \le c(\|\X^{(a)}\|_{C^{2,\alpha^+}},\|\X^{(b)}\|_{C^{2,\alpha^+}},\abs{\X^{\rm in}}_\star)\,\epsilon^{-1}\abs{\log\epsilon}\bigg(\norm{\X^{(a)}-\X^{(b)}}_{C^{2,\alpha^+}}\norm{(\tau_\epsilon^{(a)}\X_s^{(a)})_s }_{C^{0,\alpha}}\\
  &\quad + \norm{(\tau_\epsilon^{(a)}\X_s^{(a)})_s -(\tau_\epsilon^{(b)}\X_s^{(b)})_s }_{C^{0,\alpha}}\bigg)\\
  %
  %
  & \le c(\|\X^{(a)}\|_{C^{4,\alpha}},\|\X^{(b)}\|_{C^{2,\alpha^+}},\abs{\X^{\rm in}}_\star)\,\epsilon^{-\alpha^+}\abs{\log\epsilon}^4\norm{\X^{(a)}-\X^{(b)}}_{C^{4,\alpha}}\,.
\end{aligned}
\end{equation}
Here we have used the Lipschitz bounds of Theorem \ref{thm:decomp} and Corollary \ref{cor:TDP_Xssss} as well as the mapping properties of Lemma \ref{lem:straight_Leps}.

Next, we have that for $\alpha<\gamma<1$, the more regular tension remainder terms $J_+$ satisfy  
\begin{equation}\label{eq:Jplus_lip}
\begin{aligned}
  &\norm{J_+^{(a)}[\X^{(a)}]-J_+^{(b)}[\X^{(b)}]}_{C^{1,\gamma}(\T)} 
  \le \norm{\mc{R}_{\rm n,+}^{(a)}[(\tau_\epsilon^{(a)}\X_s^{(a)})_s] - \mc{R}_{\rm n,+}^{(b)}[(\tau_\epsilon^{(b)}\X_s^{(b)})_s]}_{C^{1,\gamma}}\\
  &\hspace{6cm} + \norm{\mc{L}_\epsilon^{(a)}[(\tau_+^{(a)}\X_s^{(a)})_s] - \mc{L}_\epsilon^{(b)}[(\tau_+^{(b)}\X_s^{(b)})_s]}_{C^{1,\gamma}}\\
  %
  %
    %
  %
  &\quad \le c(\epsilon,\|\X^{(a)}\|_{C^{3,\alpha}},\|\X^{(b)}\|_{C^{3,\alpha}},\abs{\X^{\rm in}}_\star)\bigg(\norm{\X^{(a)}-\X^{(b)}}_{C^{2,\alpha^+}}\norm{(\tau_\epsilon^{(a)}\X_s^{(a)})_s}_{C^{0,\alpha}} \\
  &\hspace{3cm}+ \norm{(\tau_\epsilon^{(a)}\X_s^{(a)})_s-(\tau_\epsilon^{(b)}\X_s^{(b)})_s}_{C^{0,\alpha}}\bigg)\\
  &\qquad + c(\epsilon,\|\X^{(a)}\|_{C^{3,\gamma}},\|\X^{(b)}\|_{C^{3,\gamma}},\abs{\X^{\rm in}}_\star)\bigg(\norm{\X^{(a)}-\X^{(b)}}_{C^{2,\gamma^+}}\norm{(\tau_+^{(a)}\X_s^{(a)})_s}_{C^{0,\gamma}} \\
  &\hspace{3cm}+ \norm{(\tau_+^{(a)}\X_s^{(a)})_s-(\tau_+^{(b)}\X_s^{(b)})_s}_{C^{0,\gamma}}\bigg)\\
  &\quad \le c(\epsilon,\|\X^{(a)}\|_{C^{3,\gamma}},\|\X^{(b)}\|_{C^{3,\gamma}},\abs{\X^{\rm in}}_\star)\norm{\X^{(a)}-\X^{(b)}}_{C^{4,\alpha}}\,.
\end{aligned}
\end{equation}
Here we have used the Lipschitz bounds of Theorem \ref{thm:decomp} for $\mc{R}_{\rm n,+}$ and $\mc{L}_\epsilon$, along with Corollary \ref{cor:TDP_Xssss} for $\tau_\epsilon[\X_{ssss}]$ and $\tau_+[\X_{ssss}]$.

Using \eqref{eq:Jeps_lip} in the Duhamel formula \eqref{eq:duhamel}, by Lemma \ref{lem:max_reg} we obtain the bound
\begin{equation}
\begin{aligned}
  &\norm{\int_0^t e^{-(t-t')\overline{\mc{L}}_\epsilon\p_s^4}\big(J_\epsilon^{(a)}[\X^{(a)}] - J_\epsilon^{(b)}[\X^{(b)}]\big)\,dt'}_{\mc{Y}_4} \\
  &\quad \le (\epsilon + T^{1/4}\abs{\log\epsilon}^{-3/4}+T)\,c\norm{J_\epsilon^{(a)}[\X^{(a)}]- J_\epsilon^{(b)}[\X^{(b)}]}_{\mc{Y}_1}\\
  %
  %
  &\quad \le \epsilon^{1-\alpha^+}\abs{\log\epsilon}^4c(M,\abs{\X^{\rm in}}_\star)\norm{\X^{(a)}-\X^{(b)}}_{\mc{Y}_4}\\
  &\qquad
  + (T^{1/4}+T)\,c(\epsilon,M,\abs{\X^{\rm in}}_\star)\norm{\X^{(a)}-\X^{(b)}}_{\mc{Y}_4}\,.
\end{aligned}
\end{equation}
Moreover, using \eqref{eq:Jplus_lip} and Lemma \ref{lem:semigroup}, we have 
\begin{equation}
\begin{aligned}
  &\norm{\int_0^t e^{-(t-t')\overline{\mc{L}}_\epsilon\p_s^4}\big(J_+^{(a)}[\X^{(a)}]-J_+^{(b)}[\X^{(b)}]\big)\,dt'}_{C^{4,\alpha}(\T)}\\
  &\qquad \le c\int_0^t(t-t')^{-1+\frac{\gamma-\alpha}{3}}\norm{J_+^{(a)}[\X^{(a)}]-J_+^{(b)}[\X^{(b)}]}_{C^{1,\gamma}(\T)}\,dt'\\
  &\qquad \le c(\epsilon,M,\abs{\X^{\rm in}}_\star)\int_0^t(t-t')^{-1+\frac{\gamma-\alpha}{3}}\norm{\X^{(a)}-\X^{(b)}}_{C^{4,\alpha}(\T)}\,dt'\\
  &\qquad \le t^{\frac{\gamma-\alpha}{3}}\,c(\epsilon,M,\abs{\X^{\rm in}}_\star)\norm{\X^{(a)}-\X^{(b)}}_{\mc{Y}_4}\,.
\end{aligned}
\end{equation}
Next, applying Lemma \ref{lem:semigroup} and Corollary \ref{cor:four_derivs}, we may bound the difference in remainder terms $\wt{\mc{R}}_{\rm n,+}$ as 
\begin{equation}
\begin{aligned}
  &\norm{\int_0^t e^{-(t-t')\overline{\mc{L}}_\epsilon\p_s^4}\big(\wt{\mc{R}}_{\rm n,+}^{(a)}[\X_{ssss}^{(a)}]- \wt{\mc{R}}_{\rm n,+}^{(b)}[\X_{ssss}^{(b)}]\big)\,dt'}_{C^{4,\alpha}(\T)} \\
  &\quad \le c\int_0^t (t-t')^{-1+\frac{\gamma-\alpha}{3}}\bigg(\norm{(\wt{\mc{R}}_{\rm n,+}^{(a)}-\wt{\mc{R}}_{\rm n,+}^{(b)})[\X_{ssss}^{(a)}]}_{C^{1,\gamma}} + \norm{\wt{\mc{R}}_{\rm n,+}^{(b)}[\X_{ssss}^{(a)}-\X_{ssss}^{(b)}]}_{C^{1,\gamma}}\bigg)\,dt' \\
  &\quad \le c(\epsilon,M,\abs{\X^{\rm in}}_\star)\int_0^t (t-t')^{-1+\frac{\gamma-\alpha}{3}}\norm{\X^{(a)}-\X^{(b)}}_{C^{4,\alpha}(\T)}\,dt' \\
  &\quad \le t^{\frac{\gamma-\alpha}{3}}\,c(\epsilon,M,\abs{\X^{\rm in}}_\star)\norm{\X^{(a)}-\X^{(b)}}_{\mc{Y}_4} \,.
\end{aligned}
\end{equation}
Finally, by Lemma \ref{lem:max_reg} and Corollary \ref{cor:four_derivs}, the difference in remainder terms $\mc{R}_{\rm n,\epsilon}$ satisfies 
\begin{equation}
\begin{aligned}
  &\norm{\int_0^t e^{-(t-t')\overline{\mc{L}}_\epsilon\p_s^4}\big(\mc{R}_{\rm n,\epsilon}^{(a)}\big[\overline{\mc{L}}_\epsilon[\X_{ssss}^{(a)}]\big]-\mc{R}_{\rm n,\epsilon}^{(b)}\big[\overline{\mc{L}}_\epsilon[\X_{ssss}^{(b)}]\big]\big)\,dt'}_{\mc{Y}_4}\\
  &\quad \le (\epsilon+T^{1/4}\abs{\log\epsilon}^{-3/4}+T)\,c\bigg(
  \norm{(\mc{R}_{\rm n,\epsilon}^{(a)}-\mc{R}_{\rm n,\epsilon}^{(b)})\big[\overline{\mc{L}}_\epsilon[\X_{ssss}^{(a)}]\big]}_{\mc{Y}_1}\\
  &\hspace{6cm}  + 
  \norm{\mc{R}_{\rm n,\epsilon}^{(b)}\big[\overline{\mc{L}}_\epsilon[\X_{ssss}^{(a)}-\X_{ssss}^{(b)}]\big]}_{\mc{Y}_1}\bigg) \\
  &\quad \le \epsilon^{1-\alpha^+}\abs{\log\epsilon}(\epsilon+T^{1/4}\abs{\log\epsilon}^{-3/4}+T)\,c(M,\abs{\X^{\rm in}}_\star)\times\\
  &\qquad \times\bigg(\norm{\X^{(a)}-\X^{(b)}}_{\mc{Y}_4} \norm{\overline{\mc{L}}_\epsilon[\X_{ssss}^{(a)}]}_{\mc{Y}_1}
  + \norm{\overline{\mc{L}}_\epsilon[\X_{ssss}^{(a)}-\X_{ssss}^{(b)}]}_{\mc{Y}_1}\bigg)\\
  &\quad \le \epsilon^{1-\alpha^+}\abs{\log\epsilon}(\epsilon+T^{1/4}\abs{\log\epsilon}^{-3/4}+T)\,c(M,\abs{\X^{\rm in}}_\star)\big(\epsilon^{-1}\abs{\log\epsilon}\norm{\X^{(a)}-\X^{(b)}}_{\mc{Y}_4}\big)\\
  &\quad \le \epsilon^{1-\alpha^+}\abs{\log\epsilon}^2\,c(M,\abs{\X^{\rm in}}_\star)\norm{\X^{(a)}-\X^{(b)}}_{\mc{Y}_4}\\
  &\qquad +(T^{1/4}+T)\,c(\epsilon,M,\abs{\X^{\rm in}}_\star)\norm{\X^{(a)}-\X^{(b)}}_{\mc{Y}_4}\,.
\end{aligned}
\end{equation}

Returning to the expression \eqref{eq:Lambda_def} for the map $\Lambda$, in total, we may estimate the difference $\Lambda[\X^{(a)}]-\Lambda[\X^{(b)}]$ as
\begin{equation}
\begin{aligned}
  &\norm{\Lambda[\X^{(a)}]-\Lambda[\X^{(b)}]}_{\mc{Y}_4}\\
  &\quad\le \bigg((T^{\frac{\gamma-\alpha}{3}} + T^{1/4}+T)\,c(\epsilon,M,\abs{\X^{\rm in}}_\star) + \epsilon^{1-\alpha^+}\abs{\log\epsilon}^4\,c(M,\abs{\X^{\rm in}}_\star) \bigg)\norm{\X^{(a)}-\X^{(b)}}_{\mc{Y}_4}\,.
\end{aligned}
\end{equation}
We first take $\epsilon$ small enough that 
\begin{equation}\label{eq:epscond2}
\epsilon^{1-\alpha^+}\abs{\log\epsilon}^4\,c(M,\abs{\X^{\rm in}}_\star)\le \frac{1}{3}
\end{equation}
and \eqref{eq:epscond1} hold.
We then take $T$ such that $(T^{\frac{\gamma-\alpha}{3}} + T^{1/4}+T)\,c(\epsilon,M,\abs{\X^{\rm in}}_\star)\le \frac{1}{3}$ and \eqref{eq:Tcond1}, \eqref{eq:Tcond2} hold. We thus obtain that the map $\Lambda$ admits a unique fixed point on $\mc{B}_{M,\star}$, which establishes Theorem \ref{thm:main}.



\appendix
\section{Estimates for the map $\Phi$}\label{sec:Phi_comm}
Here we record useful estimates for the map $\Phi$ identifying the tangential/normal directions along a curved filament with the tangential/normal directions along the straight filament $\mc{C}_\epsilon$. The definition of the map $\Phi$ \eqref{eq:mapPhi_def} involves specifying an orthonormal frame about the filament centerline, which we may define as $(\be_{\rm t}(s),\be_{\rm n_1}(s),\be_{\rm n_2}(s))$ satisfying $\be_{\rm t}(s)=\X_s$ along with the ODEs
\begin{equation}\label{eq:frame}
\frac{d}{ds}
\begin{pmatrix}
\be_{\rm t}(s)\\
\be_{\rm n_1}(s)\\
\be_{\rm n_2}(s)
\end{pmatrix}
 = \begin{pmatrix}
 0 & \kappa_1(s) & \kappa_2(s) \\
 -\kappa_1(s) & 0 & \kappa_3 \\
-\kappa_2(s) & -\kappa_3& 0
 \end{pmatrix}\begin{pmatrix}
\be_{\rm t}(s)\\
\be_{\rm n_1}(s)\\
\be_{\rm n_2}(s)
\end{pmatrix}\,, \quad 
\kappa_1^2+\kappa_2^2=\abs{\X_{ss}}^2\,.
\end{equation}
Here, by \cite[Lemma 1.1]{closed_loop}, $\kappa_3$ may be taken to be a constant satisfying $\abs{\kappa_3}\le \pi$.

Writing $\kappa_1=\X_{ss}\cdot\be_{\rm n_1}$ and $\kappa_2=\X_{ss}\cdot\be_{\rm n_2}$ and using \eqref{eq:frame}, we see that we may bound $\norm{\kappa_1}_{C^{\ell,\alpha}}\,,\,\norm{\kappa_2}_{C^{\ell,\alpha}}\le c\norm{\X}_{C^{2+\ell,\alpha}}$ for integers $\ell \ge 0$.
Immediately by the definitions \eqref{eq:mapPhi_def} and \eqref{eq:frame}, we may obtain the following single-filament commutator estimates: 
\begin{equation}\label{eq:Phi_commutator_est}
\begin{aligned}
\norm{[\p_s,\Phi]\,\bm{g}}_{L^\infty} &\le c(\norm{\X}_{C^2})\norm{\bm{g}}_{L^\infty}\,, \qquad 
\norm{[\p_s,\Phi]\,\bm{g}}_{C^{\ell,\alpha}} \le c(\norm{\X}_{C^{2+\ell,\alpha}})\norm{\bm{g}}_{C^{\ell,\alpha}} \\
\norm{[\p_s,\Phi^{-1}]\,\bm{h}}_{L^\infty}&\le c(\norm{\X}_{C^2})\norm{\bm{h}}_{L^\infty}\,, \qquad 
\norm{[\p_s,\Phi^{-1}]\,\bm{h}}_{C^{\ell,\alpha}} \le c(\norm{\X}_{C^{2+\ell,\alpha}})\norm{\bm{h}}_{C^{\ell,\alpha}}\,.
\end{aligned}
\end{equation}
However, due to the frame dependence of $\Phi$, Lipschitz bounds between pairs of curves are more subtle. In particular, to compare two nearby curves, we need a systematic way of choosing frames that are also ``nearby" in a quantitative sense. 
Consider two filaments $\Sigma_\epsilon^{(a)}$ and $\Sigma_\epsilon^{(b)}$ as in \eqref{eq:SigmaEps}-\eqref{eq:rstar} with centerlines parameterized by $\X^{(a)}(s)$ and $\X^{(b)}(s)$, respectively. Following \cite{ohm2024free}, let $\X^{(a)}(0)$, $\X^{(b)}(0)$ denote the nearest points on the curves, i.e.
\begin{equation}
\min_{s_1,s_2\in \T} \abs{\X^{(a)}(s_1)-\X^{(b)}(s_2)}\,.
\end{equation}
Note that this pair of points need not be unique, in which case the pair may be selected arbitrarily among all such pairs. These will serve as the base points for the parameterizations $\X^{(a)}(s)$ and $\X^{(b)}(s)$. We then consider pairs of curves that are ``close" in the sense that  
\begin{equation}\label{eq:XaXb_close}
\norm{\X^{(a)}-\X^{(b)}}_{C^{2,\alpha}(\T)} =\delta
\end{equation}
for some $\delta\ge 0$ sufficiently small. In \cite[Lemma 1.1]{ohm2024free}, we show that, given two such nearby curves, there exists a systematic choice of nearby orthonormal frames such that the following holds. 
\begin{lemma}[Nearby parameterizations for nearby curves]\label{lem:XaXb_C2beta}
Given $\Sigma_\epsilon^{(a)}$ and $\Sigma_\epsilon^{(b)}$ as in \eqref{eq:SigmaEps}-\eqref{eq:rstar} with centerlines $\X^{(a)}(s)$, $\X^{(b)}(s)$ satisfying \eqref{eq:XaXb_close}, there exists a choice of orthonormal frames $(\be_{\rm t}^{(a)},\be_{\rm n_1}^{(a)},\be_{\rm n_2}^{(a)})$ and $(\be_{\rm t}^{(b)},\be_{\rm n_1}^{(b)},\be_{\rm n_2}^{(b)})$, respectively, satisfying \eqref{eq:frame} such that the frame coefficients satisfy
\begin{equation}\label{eq:XaXb_C2beta}
\sum_{\ell=1}^3\norm{\kappa_\ell^{(a)}-\kappa_\ell^{(b)}}_{C^{0,\alpha}(\T)} \le c(\|\X^{(a)}\|_{C^{2,\alpha}},\|\X^{(b)}\|_{C^{2,\alpha}})\,\norm{\X^{(a)}-\X^{(b)}}_{C^{2,\alpha}(\T)}\,.
\end{equation}
\end{lemma}
Whenever the map $\Phi$ is needed for two sufficiently nearby filaments, we will always choose a parameterization satisfying Lemma \ref{lem:XaXb_C2beta}.

In particular, given such nearby filaments with centerlines $\X^{(a)}$ and $\X^{(b)}$, using the parameterization of Lemma \ref{lem:XaXb_C2beta}, it may be seen that the difference between the maps $\Phi^{(a)}$ and $\Phi^{(b)}$ satisfies
\begin{equation}\label{eq:Phi_lip_est}
\begin{aligned}
%
&\norm{(\Phi^{(a)})^{-1}\bm{h}_0^{\Phi^{(a)}}-(\Phi^{(b)})^{-1}\bm{h}_0^{\Phi^{(b)}}}_{C^{\ell,\alpha}} \\
&\hspace{3cm} \le c(\|\X^{(a)}\|_{C^{1+\ell,\alpha}},\|\X^{(b)}\|_{C^{1+\ell,\alpha}})\norm{\X^{(a)}-\X^{(b)}}_{C^{1+\ell,\alpha}}\norm{\bm{h}}_{C^{\ell,\alpha}} \,.
\end{aligned}
\end{equation}
Here the notation $\bm{h}_0^{\Phi^{(a)}}$ is as in \eqref{eq:subtract_mean}.
Furthermore, we may obtain the following commutator Lipschitz bounds:
\begin{equation}\label{eq:Phi_comm_lip_est}
\begin{aligned}
%
&\norm{\big[\p_s,(\Phi^{(a)})^{-1}\big]\bm{h}_0^{\Phi^{(a)}}-\big[\p_s,(\Phi^{(b)})^{-1}\big]\bm{h}_0^{\Phi^{(b)}}}_{C^{\ell,\alpha}}\\ 
&\hspace{3cm}\le c(\|\X^{(a)}\|_{C^{2+\ell,\alpha}},\|\X^{(b)}\|_{C^{2+\ell,\alpha}})\norm{\X^{(a)}-\X^{(b)}}_{C^{2+\ell,\alpha}}\norm{\bm{h}}_{C^{\ell,\alpha}} \,.
\end{aligned}
\end{equation}

\section{Proof of Lemma \ref{lem:TDP_auxlem}}\label{sec:TDP_auxlem_pf}
We begin by outlining the general proof strategy and tools. As in \cite[section 2]{ohm2025angle} and \cite[section 3]{ohm2024free}, we consider $C^{\ell,\alpha}(\T)$ as a subset of $C^{\ell,\alpha}(\R)$, which, using \cite[Theorem 2.36]{bahouri2011fourier}, we may characterize as the Besov space $B^{\ell+\alpha}_{\infty,\infty}(\R)$. 
Let $\phi(\xi)$ be a smooth cutoff function supported on the annulus $\{\frac{3}{4}\le \abs{\xi}\le \frac{8}{3}\}$ and satisfying $\sum_{j\in\Z}\phi(2^{-j}\xi)=1$. Let 
\begin{equation}\label{eq:dyadic_phi}
\phi_j(\xi) = \phi(2^{-j}\xi)
\end{equation}
and, given a function $h(s)$, $s\in\R$, let
\begin{equation}
 \mc{F}[h](\xi) = \int_{-\infty}^\infty h(s)\,e^{-2\pi i \xi s}\, ds 
\end{equation}
denote the Fourier transform of $h$. We denote the Littlewood-Paley projection $P_jh$ onto frequencies supported in annulus $j$ by
\begin{equation}\label{eq:littlewood_p}
\mc{F}[P_j h](\xi)  = \phi_j(\xi)\mc{F}[h](\xi)\,.
\end{equation}
We further denote
\begin{equation}
P_{\le j_*} := \sum_{j\le j_*} P_j \,.
\end{equation}
For $\nu\in \R$, we may then define the following Besov seminorm and norm:
\begin{equation}\label{eq:besov}
|h|_{\dot B^\nu_{\infty,\infty}} = \sup_{j\in \Z} 2^{j\nu}\norm{P_j h}_{L^\infty}\,, \qquad
\norm{h}_{B^\nu_{\infty,\infty}} =\sup\big(\norm{P_{\le 0} \,h}_{L^\infty}, \sup_{j> 0} \,2^{j\nu}\norm{P_j h}_{L^\infty}\big)\,.
\end{equation}

Now, given a Fourier multiplier $m(\xi)$ and a function $h(s)$, $s\in \R$, let 
\begin{equation}
T_m h = \mc{F}^{-1}[m\mc{F}[h]] \,, \qquad P_jT_m h = \mc{F}^{-1}[\phi_jm\mc{F}[h]] \,.
\end{equation}
Our general strategy for controlling the Besov norm \eqref{eq:besov} of the operator $T_m h$ proceeds as follows. Letting 
\begin{equation}
\wh M_j(\xi)= \phi_j(\xi)m(\xi)\,, \qquad M_j=\mc{F}^{-1}[\phi_jm]\,,
\end{equation}
we have that by \cite[Lemma 2.1]{ohm2025angle}, if 
\begin{equation}\label{eq:Linfty_Mj}
\abs{\wh M_j}\le A\,, \qquad
\abs{\p_\xi^2\wh M_j}\le B
\end{equation}
then
\begin{equation}\label{eq:L1_Mj}
\norm{M_j}_{L^1(\R)} \lesssim 2^j \sqrt{AB}\,.
\end{equation}
Noting that $P_jT_m h =M_j*h$, by Young's inequality for convolutions, we may then bound
\begin{equation}\label{eq:PjBound}
\norm{P_jT_m h}_{L^\infty(\R)}=\norm{M_j*h}_{L^\infty(\R)}\le \norm{M_j}_{L^1(\R)}\norm{h}_{L^\infty(\R)} 
\lesssim 2^j \sqrt{AB}\norm{h}_{L^\infty(\R)} \,.
\end{equation}
Thus, to show Lemma \ref{lem:TDP_auxlem}, our strategy will be to obtain bounds of the form \eqref{eq:Linfty_Mj} on the multipliers corresponding to $(I- \overline{\mc{L}}_\epsilon^{\rm tang}\p_{ss})^{-1}$ and $(I- \overline{\mc{L}}_\epsilon^{\rm tang}\p_{ss})^{-1}\p_s$.

We begin by defining
\begin{equation}
  m^{(1)}_\epsilon(\xi) = \frac{1}{1+(2\pi\xi)^2m_{\epsilon,\rm t}(\xi)}\,, \qquad
  m^{(2)}_\epsilon(\xi) = \frac{i2\pi\xi}{1+(2\pi\xi)^2m_{\epsilon,\rm t}(\xi)}
\end{equation}
where $m_{\epsilon,\rm t}$ is given by \eqref{eq:eigsT}.
By \cite[Lemma 3.4]{ohm2024free}, we have that $m_{\rm \epsilon,t}$ satisfies the bounds  
\begin{equation}\label{eq:mt_bds}
\begin{aligned}
\abs{\p_{\xi}^\ell m_{\epsilon,{\rm t}}^{-1}(\xi)} &\le 
\begin{cases}
c\,\epsilon\abs{\xi}^{1-\ell}\, , & \abs{\xi}\ge\frac{1}{2\pi\epsilon} \,, \\
c\,\abs{\log\epsilon}^{-1}\abs{\xi}^{-\ell}\,, & \abs{\xi}<\frac{1}{2\pi\epsilon} \,,
\end{cases} 
\quad \ell=0,1,2\,,\\
\abs{\p_{\xi}^\ell m_{\epsilon,{\rm t}}(\xi)} &\le 
\begin{cases}
c\,\epsilon^{-1} \abs{\xi}^{-\ell-1}\,, & \abs{\xi}\ge\frac{1}{2\pi\epsilon} \,, \\
c\,\abs{\log\epsilon}\abs{\xi}^{-\ell} \,, & \abs{\xi}<\frac{1}{2\pi\epsilon}\,,
\end{cases}
\quad \ell=0,1,2\,.
\end{aligned}
\end{equation}
From \eqref{eq:mt_bds}, we may then obtain the following bounds for $m^{(1)}_\epsilon(\xi)$ and $m^{(2)}_\epsilon(\xi)$: 
\begin{equation}\label{eq:meps_bds}
\begin{aligned}
\abs{\p_{\xi}^\ell m_\epsilon^{(1)}(\xi)} &\le 
\begin{cases}
c\,\epsilon\abs{\xi}^{-1-\ell}\, , & \abs{\xi}\ge\frac{1}{2\pi\epsilon} \,, \\
c\,\abs{\xi}^{-\ell}(1+\abs{\xi}^2\abs{\log\epsilon})^{-1}\,, & \abs{\xi}<\frac{1}{2\pi\epsilon} \,,
\end{cases} 
\quad \ell=0,1,2\,,\\
\abs{\p_{\xi}^\ell m_\epsilon^{(2)}(\xi)} &\le 
\begin{cases}
c\,\epsilon\abs{\xi}^{-\ell}\,, & \abs{\xi}\ge\frac{1}{2\pi\epsilon} \,, \\
c\,\abs{\xi}^{1-\ell}(1+\abs{\xi}^2\abs{\log\epsilon})^{-1} \,, & \abs{\xi}<\frac{1}{2\pi\epsilon}\,,
\end{cases}
\quad \ell=0,1,2\,.
\end{aligned}
\end{equation}
We next define the frequency-localized Fourier multipliers 
\begin{equation}
  \wh{M}_j^{(1)}(\xi)= \phi_j(\xi)m^{(1)}_\epsilon(\xi) \,, \qquad
  \wh{M}_j^{(2)}(\xi)= \phi_j(\xi)m^{(2)}_\epsilon(\xi)\,,
\end{equation}
where $\phi_j$ is as in \eqref{eq:dyadic_phi}.
Letting $j_\epsilon = \frac{\abs{\log(2\pi\epsilon)}}{\log(2)}$, by \eqref{eq:meps_bds}, we then have 
\begin{equation}
\begin{aligned}
  \norm{\p_\xi^\ell \wh{M}_j^{(1)}}_{L^\infty} &\le \begin{cases}
    c\,\epsilon \,2^{-j(\ell+1)}\,, & j\ge j_\epsilon\\
    c\,2^{-j\ell}(1+2^{2j}\abs{\log\epsilon})^{-1}\,, & j<j_\epsilon\,,
  \end{cases} 
  \quad \ell=0,1,2\,, \\
  \norm{\p_\xi^\ell \wh{M}_j^{(2)}}_{L^\infty} &\le \begin{cases}
    c\,\epsilon \,2^{-j\ell}\,, & j\ge j_\epsilon\\
    c\,2^{j(1-\ell)}(1+2^{2j}\abs{\log\epsilon})^{-1}\,, & j<j_\epsilon\,,
  \end{cases}
  \quad \ell=0,1,2\,.
\end{aligned}
\end{equation}
Using \eqref{eq:L1_Mj}, we obtain the physical space estimates
\begin{equation}\label{eq:phys_ests}
\begin{aligned}
  \norm{M_j^{(1)}}_{L^1} &\le \begin{cases}
    c\,\epsilon\,2^{-j}\,, & j\ge j_\epsilon\\
    c\,(1+2^{2j}\abs{\log\epsilon})^{-1} \,, & j<j_\epsilon\,,
  \end{cases}\\
  \norm{M_j^{(2)}}_{L^1} &\le \begin{cases}
    c\,\epsilon \,, & j\ge j_\epsilon\\
    c\,2^j(1+2^{2j}\abs{\log\epsilon})^{-1}\,, & j<j_\epsilon\,.
  \end{cases}
\end{aligned}
\end{equation}

Now, for $h(s)\in C^{0,\alpha}(\T)$, we may use the dyadic partition of unity \eqref{eq:dyadic_phi} to write
\begin{equation}
\begin{aligned}
  (I-\overline{\mc{L}}_\epsilon^{\rm tang}\p_{ss})^{-1}[h] &= T_{m^{(1)}_\epsilon}h = \sum_jP_jT_{m^{(1)}_\epsilon}h\\
  (I-\overline{\mc{L}}_\epsilon^{\rm tang}\p_{ss})^{-1}\p_s[h] &= T_{m^{(2)}_\epsilon}h = \sum_jP_jT_{m^{(2)}_\epsilon}h\,.
\end{aligned}
\end{equation}
We will estimate $T_{m^{(1)}_\epsilon}h$ and $T_{m^{(2)}_\epsilon}h$ using \eqref{eq:PjBound}.

We begin with the lowest frequencies. Since we are now working on the torus $\T$, we have that $P_{\le 0}T_{m^{(1)}_\epsilon}h = P_0T_{m^{(1)}_\epsilon}h$, and we may estimate 
\begin{equation}
  \norm{P_{\le 0}T_{m^{(1)}_\epsilon}h}_{L^\infty} = \norm{M_0^{(1)}*h}_{L^\infty} 
  \le \norm{M_0^{(1)}}_{L^1}\norm{h}_{L^\infty}
  \le c(\epsilon)\norm{h}_{L^\infty}\,.
\end{equation}
Similarly, $P_{\le 0}T_{m^{(1)}_\epsilon}h$ may be estimated as
\begin{equation}
  \norm{P_{\le 0}T_{m^{(2)}_\epsilon}h}_{L^\infty} = \norm{M_0^{(2)}*h}_{L^\infty} 
  \le \norm{M_0^{(2)}}_{L^1}\norm{h}_{L^\infty}
  \le c(\epsilon)\norm{h}_{L^\infty}\,.
\end{equation}
For frequencies  $j>0$, we will make use of the fattened frequency projection $\wt P_j=P_{j-1}+P_j+P_{j+1}$.
We first consider low frequencies $0<j<j_\epsilon$. Using \eqref{eq:phys_ests}, we may estimate
\begin{equation}
\begin{aligned}
  \sup_{0<j< j_\epsilon} 2^{j(\ell+2+\alpha)}\norm{P_jT_{m^{(1)}_\epsilon}h}_{L^\infty} 
  &= \sup_{0<j< j_\epsilon} 2^{j(\ell+2+\alpha)}\norm{P_jT_{m^{(1)}_\epsilon}\wt P_jh}_{L^\infty}\\
  &=\sup_{0<j< j_\epsilon} 2^{j(\ell+2+\alpha)}\norm{M_j^{(1)}*(\wt P_jh)}_{L^\infty} \\
  &\le \sup_{0<j< j_\epsilon} 2^{2j}\norm{M_j^{(1)}}_{L^1}2^{j(\ell+\alpha)}\norm{\wt P_jh}_{L^\infty}
  \le c(\epsilon)\abs{h}_{\dot B^{\ell+\alpha}_{\infty,\infty}}\,,
  \\
  \sup_{0<j< j_\epsilon} 2^{j(\ell+1+\alpha)}\norm{P_jT_{m^{(2)}_\epsilon}h}_{L^\infty}  &= \sup_{0<j< j_\epsilon} 2^{j(\ell+1+\alpha)}\norm{P_jT_{m^{(2)}_\epsilon}\wt P_jh}_{L^\infty}\\
  &=\sup_{0<j< j_\epsilon} 2^{j(\ell+1+\alpha)}\norm{M_j^{(2)}*(\wt P_jh)}_{L^\infty} \\
  &\le \sup_{0<j< j_\epsilon} 2^j\norm{M_j^{(2)}}_{L^1}2^{j(\ell+\alpha)}\norm{\wt P_jh}_{L^\infty}
  \le c(\epsilon)\abs{h}_{\dot B^{\ell+\alpha}_{\infty,\infty}}\,.
\end{aligned}
\end{equation}
Finally, for high frequencies $j\ge j_\epsilon$, using \eqref{eq:phys_ests}, we may estimate
\begin{equation}
\begin{aligned}
  \sup_{j\ge j_\epsilon} 2^{j(\ell+1+\alpha)}\norm{P_jT_{m^{(1)}_\epsilon}h}_{L^\infty} &\le \sup_{j\ge j_\epsilon} 2^j\norm{M_j^{(1)}}_{L^1} 2^{j(\ell+\alpha)}\norm{\wt P_jh}_{L^\infty}
  \le c\,\epsilon\abs{h}_{\dot B^{\ell+\alpha}_{\infty,\infty}} \\
  \sup_{j\ge j_\epsilon} 2^{j(\ell+\alpha)}\norm{P_jT_{m^{(2)}_\epsilon}h}_{L^\infty} &\le 
  \sup_{j\ge j_\epsilon} \norm{M_j^{(2)}}_{L^1}2^{j(\ell+\alpha)}\norm{\wt P_jh}_{L^\infty}
  \le c\,\epsilon\abs{h}_{\dot B^{\ell+\alpha}_{\infty,\infty}}\,.
\end{aligned}
\end{equation}
We define the operators $\mc{M}_{1,+}$ and $\mc{M}_{2,+}$ in \eqref{eq:M1defs}, \eqref{eq:M2defs} as the sum of low frequencies 
\begin{equation}
  \mc{M}_{1,+}[h] = \sum_{j< j_\epsilon} P_jT_{m^{(1)}_\epsilon}h\,, \quad
  \mc{M}_{2,+}[h] = \sum_{j< j_\epsilon} P_jT_{m^{(2)}_\epsilon}h\,,
\end{equation}
and similarly define
\begin{equation}
  \mc{M}_{1,\epsilon}[h] = \sum_{j\ge j_\epsilon} P_jT_{m^{(1)}_\epsilon}h\,, \quad
  \mc{M}_{2,\epsilon}[h] = \sum_{j\ge j_\epsilon} P_jT_{m^{(2)}_\epsilon}h\,.
\end{equation}
Using the equivalence of the $B^{\nu}_{\infty,\infty}$ \eqref{eq:besov} and $C^{\lfloor\nu\rfloor,\nu-\lfloor\nu\rfloor}$ norms, we obtain Lemma \ref{lem:TDP_auxlem}.

\section{Proof of semigroup mapping properties}\label{sec:h4alpha}
Here we complete the proofs of Lemmas \ref{lem:semigroup} and \ref{lem:max_reg}. Note that the bounds \eqref{eq:semigroup} and \eqref{eq:max_semi} were established in \cite[Lemmas 2.2 and 2.3]{ohm2024free}, respectively. It only remains to verify the assertions about little H\"older spaces $h^{k,\alpha}$.

We complete the proof of Lemma \ref{lem:max_reg} in detail; the proof of the analogous claim in Lemma \ref{lem:semigroup} for time-independent $\bm{V}(s)$ follows by nearly identical arguments. Let $\mc{U}$ denote the map
\begin{equation}
\mc{U}[\bm{g}](s,t):= \int_0^t e^{-(t-t')\overline{\mc{L}}_\epsilon\p_s^4}\;\bm{g}(s,t')\,dt'\,.
\end{equation}
By \eqref{eq:max_semi}, we have that $\mc{U}$ maps $L^\infty(0,T;C^{1,\beta}(\T))$ to $L^\infty(0,T;C^{4,\beta}(\T))$ for $\beta\in(0,1)$.
Given $\bm{g}\in \mc{Y}_1=C([0,T];h^{1,\alpha}(\T))$, we may approximate $\bm{g}$, e.g. by mollification, by smoother functions $\bm{\varphi}\in L^\infty(0,T;C^{1,\alpha+\delta})$, $\alpha<\alpha+\delta<1$. For any such $\bm{\varphi}$, we have $\mc{U}[\bm{\varphi}]\in L^\infty(0,T;C^{4,\alpha+\delta})\subset L^\infty(0,T;h^{4,\alpha})$. We aim to show that $\mc{U}[\bm{\varphi}]\in C([0,T];C^{4,\alpha+\delta})$, i.e., for any $\nu>0$, there is some $\eta>0$ such that if $\abs{t-t'}<\eta$, then 
\begin{equation}
  \norm{\mc{U}[\bm{\varphi}](\cdot,t)-\mc{U}[\bm{\varphi}](\cdot,t')}_{C^{4,\alpha}(\T)}<\nu\,.
\end{equation}
We will use the equivalence of the $C^{4,\alpha}$ norm and the Besov $B^{4+\alpha}_{\infty,\infty}$ norm \eqref{eq:besov}. Recalling the notation $P_j$ \eqref{eq:littlewood_p} for the Littlewood-Paley projection, we first note that 
\begin{equation}
  \norm{\mc{U}[\bm{\varphi}]}_{B^{4+\alpha+\delta}_{\infty,\infty}(\T)}=\sup \bigg(\norm{P_{\le 0}\,\mc{U}[\bm{\varphi}]}_{L^\infty(\T)},\sup_{j>0}2^{j(4+\alpha+\delta)}\norm{P_j\mc{U}[\bm{\varphi}]}_{L^\infty(\T)}\bigg)<\infty\,.
\end{equation}
For $j\ge 0$, we then have
\begin{equation}
\begin{aligned}
  2^{j(4+\alpha)}\norm{P_j\mc{U}[\bm{\varphi}](\cdot,t)-P_j\mc{U}[\bm{\varphi}](\cdot,t')}_{L^\infty(\T)}
  &\le 2^{-j\delta}\sup_{t\in[0,T]}\norm{\mc{U}[\bm{\varphi}]}_{B^{4+\alpha+\delta}_{\infty,\infty}(\T)}\,.
\end{aligned}
\end{equation}
Furthermore, for any fixed $j$, we have that 
$\norm{P_j\mc{U}[\bm{\varphi}]}_{C^{4,\alpha}(\T)}=\norm{\mc{U}[P_j\bm{\varphi}]}_{C^{4,\alpha}(\T)}$ is uniformly continuous in time, i.e.
\begin{equation}
  \norm{P_j\mc{U}[\bm{\varphi}](\cdot,t)-P_j\mc{U}[\bm{\varphi}](\cdot,t')}_{B^{4+\alpha}_{\infty,\infty}(\T)} \le c_j\abs{t-t'}\,.
\end{equation}

Given $\nu>0$, we may first choose $K\in\Z$ large enough that
\begin{equation}
  2^{-K\delta}\sup_{t\in[0,T]}\norm{\mc{U}[\bm{\varphi}]}_{B^{4+\alpha+\delta}_{\infty,\infty}(\T)}< \nu\,.
\end{equation}
Then, noting that we have $P_{\le 0}=P_0$ on $\T$, we may take 
\begin{equation}
  \abs{t-t'} < \frac{\nu}{\max_{0\le j\le K}c_j}\,.
\end{equation}
Altogether, we obtain
\begin{equation}
  \norm{\mc{U}[\bm{\varphi}](\cdot,t)-\mc{U}[\bm{\varphi}](\cdot,t')}_{B^{4+\alpha}_{\infty,\infty}(\T)}<\nu\,.
\end{equation}

Finally, since $\bm{g}\in \mc{Y}_1$ may be approximated by functions $\bm{\varphi}\in L^\infty(0,T;C^{1,\alpha+\delta})$ satisfying $\mc{U}[\bm{\varphi}]\in C([0,T];C^{4,\alpha+\delta})$, by density we obtain $\mc{U}[\bm{g}]\in C([0,T];h^{4,\alpha})$.


\subsubsection*{Acknowledgments} LO acknowledges support from NSF grant DMS-2406003 and from the Wisconsin Alumni Research Foundation.


\bibliographystyle{abbrv} 
\bibliography{StokesBib}

\end{document}